\documentclass[11pt,reqno]{amsart}
\usepackage{amssymb,epsfig,graphics}
\usepackage{enumerate}
\usepackage{xcolor}
\def\r{\rightarrow}

\newcommand{\E}{\mathbb{E}}
\renewcommand{\P}{\mathbb{P}}
\renewcommand{\L}{\mathbb{L}}

\newcommand{\NN}{\mathbb{N}}
\newcommand{\RR}{\mathbb{R}}
\newcommand{\EE}{\mathbb{E}}

\newcommand{\PP}{\mathbb{P}}
\def\II{\mathbb{I}}

\newcommand{\N}{\mathbb{N}}

\newcommand{\R}{\mathbb{R}}

\newcommand{\C}{\mathbb{C}}

\newcommand{\X}{\mathbb{X}}

\renewcommand{\dim}{\mathop{\rm dim}}
\renewcommand{\ker}{\mathop{\rm Ker}}

\renewcommand{\r}{\mathop{\rightarrow}}

\newcommand{\cB}{\mathcal B}
\newcommand{\cC}{\mathcal C}
\newcommand{\cD}{\mathcal D}

\newcommand{\cL}{\mathcal L}

\newcommand{\cX}{\mathcal X}

\newtheorem{rem}{Remark}

\newtheorem{atheo}{Theorem}[section]
\newtheorem{adefi}[atheo]{Definition}
\newtheorem{alem}[atheo]{Lemma}
\newtheorem{arem}[atheo]{Remark}
\newtheorem{acor}[atheo]{Corollary}
\newtheorem{apro}[atheo]{Proposition}
\newtheorem{aex}[atheo]{Example}
\newtheorem{acond}[atheo]{Hypothesis}

\DeclareMathOperator{\Leb}{Leb}

\begin{document}

\title[Multiplicative ergodicity of Laplace transforms]{Multiplicative ergodicity of Laplace transforms for additive functional of Markov chains
with application to age-dependent branching process.}
\author{Lo\"{\i}c Herv\'e}
\address{INSA de Rennes, F-35708, France; IRMAR CNRS-UMR 6625, F-35000, France; Universit\'e Europ\'eenne de Bretagne, France.}
\email{ Loic.Herve@insa-rennes.fr}
\author{Sana Louhichi}
\address{Universit\'e Grenoble Alpes. B\^atiment IMAG, 700 avenue centrale. 38400 Saint Martin d'H\`eres, France.}
\email{Sana.Louhichi@imag.fr }
\author{Fran\c{c}oise P\`ene}
\address{Universit\'e de Brest and Institut Universitaire de France,
UMR CNRS 6205, Laboratoire de Math\'ematique de Bretagne Atlantique,
6 avenue Le Gorgeu, 29238 Brest cedex, France.}
\email{francoise.pene@univ-brest.fr}

\maketitle
\bibliographystyle{plain}
\vspace*{-8mm}
\begin{abstract}
We study the exponential growth of bifurcating
processes with ancestral dependence. We suppose here that  the lifetimes of the cells are dependent random variables, that the numbers of new cells are random and dependent. Lifetimes and new  cells's numbers are also assumed to be dependent.
We illustrate our results by examples, including some Markov models.
Our approach is related to the behaviour of the Laplace transform
of nonnegative additive functional of Markov chains and require weak moment assumption (no exponential moment is needed).
\end{abstract}
\date{\today}
\maketitle
\bibliographystyle{plain}
\tableofcontents




\section{Introduction}
Mathematical models for the growth of populations have been widely studied and applied in many fields especially animal, cell biology (mitosis), plant and forestry sciences.  Branching process is a mathematical model often used  to model reproduction (see {{for instance \cite{BH}, \cite{H} and \cite{KA}}). It is described as follows. A single ancestor object (that may be a particle, a neutron, a cosmic ray,  a cell and so on) is born at time $0$. It lives for a random time. At the moment of death, the object produces a random number of progenies.
Each of the first generation progeny behaves, independently of each other and of the ancestor: the objects do not interfere with one another.
Many authors were interested by generalizing this classical model by trying to describe the interference of objects between them (see, for instance, \cite{LouhichiYcart15} and the references therein).

In the spirit of the branching processes and for the sake of generalization, we consider in this paper a model of reproduction with a random number of children and with random life duration. Our approach, to develop this mathematical model, is to associate to each object $v$ a parameter $x_v$, this parameter may depend on its energy, its position or on other non-observed factors and it can be seen as the characteristics of the given object. When $v$ runs over the set of all objects, we assume that the set of parameters $(x_v)_v$ is a realisation of a random process $(X_v)_v$.
We do not specify either  the type of interactions between the objects nor the dependence conditions on the sequence $(X_v)_v$.
As will be described later in this paper, our results are first announced  in the case where $(X_v)_v$ is a stationary process on  some paths of objects:  the parameters processes have the same law-behavior for all objects belonging to the same generation.  The classical independent identically distributed (i.i.d. for short) case is a particular case of stationarity. Next we will use, as was done by many authors, Markov processes as mathematical models for  this parameter process. We discuss, in particular, the linear autoregressive process.

We are interested in this paper by the mitosis model and from now an object will be a cell.
This mitosis model starts with one single initial cell. After a random time, this initial cell is divided into a random number of cells and the process continue. For technical reasons we will suppose that the random number of children is always larger than $2$.
We summarize our model, used throughout the paper, as follows,
\begin{itemize}
\item to each cell $v$, is associated a parameter $x_v\in \mathbb X$
(with $(\X,\cX)$ a measurable space) which determines its lifetime $\xi(x_v)$
and the number of new cells $\kappa(x_v)$ in which the cell splits at the end of its lifetime (where $\xi$ and $\kappa$ are two measurable functions with values in
$[0,+\infty)$ and in $\mathbb Z_+$ respectively);
\item
for each line $(v_n)_{n\ge 0}$ of cells, the parameters along this line
are given by copies (non necessarily independent) of
a  process $(X_n)_{n\ge 0}$ with values in $\mathbb X$;
\item  $\kappa(x)\geq 2$ for any  $x$,
i.e. each cell gives birth to more than two children.
\end{itemize}

For every $t>0$, let us consider the Belleman-Harris age-dependent branching process $(N_t)_{t\geq 0}$, that is $N_t$ is the number of cells alive at time $t$ (see \cite{BH} and \cite{H} for more about).  The exponential growth behavior of $N_t$ as $t$ tends to infinity
 was studied in the book of Harris \cite{H} in the case where $(X_n)_n$ is a sequence of i.i.d. rv's, that is when the lifetimes are modeled by a sequence of
i.i.d. random variables independent of the random numbers of the news cells which are also assumed to be i.i.d. The growth rate $\nu_0$ (also called the Malthusian parameter) was defined, in this context, as the positive root of the equation,
 \begin{equation}\label{(H)}
 \mathbb E[\kappa(X_1)]\, \EE\left[e^{-\nu_0 \xi(X_1)}\right]=1,
\end{equation}
 as soon as the distribution of $\xi(X_1)$ is not lattice (cf.
\cite[Theorem 17.1]{H}).
Louhichi and Ycart \cite{LouhichiYcart15} extend some results of Harris to the case where the lifetimes are a sequence of dependent random
variables and when  each cell is divided, after a random lifetime, into two cells: $(X_n)_n$ is a stationary process and $\kappa(x)=2$ for any $x$. Under those assumptions the Malthusian parameter $\nu_1$ is expressed in terms of the Laplace transform of
the random variable $S_n$
\begin{equation}\label{def-Sn}
S_n:=\sum_{k=0}^n\xi(X_k)
\end{equation}
which models the birth date of the $(n+1)$-th individual of a same line. More precisely,
\begin{equation}\label{(LY)}
\nu_1=\inf\left\{\gamma>0,\,\, \sum_{n\ge 0}2^n\mathbb E\left[e^{-\gamma S_n}\right]<\infty\right\}.\end{equation}

In this paper we are interested by the general case where the lifetimes and the numbers of new cells are dependent random variables.
In this case, the growth rate of $N_t$ is given by:
\begin{equation}\label{nuDEF}
\nu:=\inf\left\{\gamma>0,\,\, \sum_{n\ge 0}
g_n(\gamma)<\infty\right\}\, ,
\end{equation}
where $g_n(\gamma)$ is expressed in terms of a perturbed Laplace transform of $S_n$,
$$
g_n(\gamma):=\mathbb E\left[\left(\prod_{j=0}^{n-1}\kappa(X_j)\right)(\kappa(X_n)-1)  e^{-\gamma S_n}\right].
$$
One task of the paper, is to give exact evaluations of $\mathbb E(N_t)$ and $\mathbb E(N_tN_{t+\tau})$, for any $t,\tau\geq 0$, under general and minimal conditions on the parameter process $(X_n)$ as described above. Those calculations are the main ingredients to get the convergence almost surely of $e^{-\nu t}N_t$ to a non-negative random variable $W$. A second task of the paper is to discuss the conditions yielding the previous results and to give some Markovian models for which the growth parameter $\nu$ is finite.

This paper is organized as follows. In section \ref{first},
we  give with proofs an exact evaluation of $\mathbb E (N_t)$ (see Proposition \ref{pro:ENt}). An immediate consequence of this calculation,  when $\nu$ is supposed finite, is the convergence in mean of $e^{-\nu t}\mathbb E (N_t)$ to some constant $C_{\nu}$ given by
$$
C_\nu:=\lim_{\gamma\rightarrow 0} \frac{\gamma}{\gamma+\nu}\sum_{n\ge 0}
g_n(\gamma+\nu),
$$
see Corollary \ref{cor1} for a precise statement. As noticed in Subsection \ref{sub22}, from a multiplicative ergodicity behavior of $g_n(\gamma)$, see Definition \ref{def-mult-erg}, one can deduce that $\nu$ and $C_{\nu}$ are both finite. For this, we study this multiplicative ergodicity property in the context of  additive functional of Markov chains with Markov kernel $P$ and initial law $\mu$.
 In Subsection \ref{sub22}, we
state our spectral assumptions, after noting that $$\forall n\ge 1,\quad g_{n}(\gamma) = \mu\left(\kappa\, e^{-\gamma \xi}\,
P_\gamma^{n-1} \left(Ph_{\kappa,\gamma}\right)\right), $$ where $P_\gamma$ is a Laplace-type kernel associated with $P$, $\xi$ and $\kappa$ (see Formula (\ref{formuleFourier}) for more details about the notations).
Those spectral assumptions are  on $P_{\gamma}$ supposed to continuously act on some suitable  Banach space $\cB$ and on its spectral radius $r(\gamma)$ (see Hypotheses 2.5, 2.7 and 2.7*).  Those assumptions are needed in order to obtain a  spectral multiplicative ergodicity property for the iterated operator $P^n_{\gamma}$. This spectral property is the main tool to prove the multiplicative ergodicity behavior of $g_n(\gamma)$ and then to deduce that $\nu$ is finite and is given by
 $$
 \nu=\inf\{\gamma>0,\,\, r(\gamma)<1\},
 $$
 (see Theorem \ref{generalspectraltheorem1}). The existence and the finiteness of $C_{\nu}$ are discussed in Thereom \ref{generalspectraltheorem2}.
 Thereom \ref{generalspectraltheorem2} needs reinforcing assumptions by considering a longer chain of Banach spaces.
 We discuss in Subsection \ref{mark-ex} two Markovian examples satisfying Theorems's \ref{generalspectraltheorem1} and \ref{generalspectraltheorem2} assumptions, see Theorem \ref{pro-Knudsen} and Theorem \ref{exemp-AR} for respectively a toy model involving some Knudsen gas (see \cite{BoattoGolse} for other results on Knudsen gases) and for linear autoregressive processes, both under weak integrability assumptions on the observable  $\xi$ (the lifetime).

Our approach for Markov models is based on the method of perturbation of operators. This method, introduced by Nagaev \cite{nag1,nag2} and by Le Page and Guivarc'h \cite{lep82,GuivarchHardy} to prove a wide class of limit theorems (central limit
theorem, local limit theorem, large and moderate deviations principles), has known an impressive development in the past decades (e.g.~see \cite{broi,hulo} and the references therein).
The price to pay for our weak moment assumptions is that, in general, the classical perturbation method does not apply in our context to the family of Laplace operators and that we have to consider several Banach spaces instead of a single one (see Remark~\ref{rem-method} for details). This is allowed by the Keller and Liverani perturbation theorem \cite{KelLiv99,Bal00}(e.g.~see \cite{HerPen10} and the references therein). The fact that we work with several spaces (due to our weak moment assumptions) complicates our study compared to the classical
approach.

In Section \ref{second},
we study the  behavior of  $\EE(N_tN_{t+\tau})$, for any $t>0,$ $\tau\geq 0$,  in the very general setting of dependence (cf.~Proposition \ref{pro1}).
Proposition \ref{pro1} is obtained under a suitable behavior of the sequence $(\kappa(X_i))_{1\leq i\leq n}$, the birthtimes $S_n$ and
$S_m^{(k)}$ of two cells belonging respectively to the $n$-th and $m$-th generation and having the last common ancestor in the $k$-th generation. More precisely, Proposition \ref{pro1} gives the following formula,
\begin{eqnarray*}
&& \EE(N_tN_{t+\tau}) =\E(N_{t})+\E(N_{t+\tau})-1+ \sum_{n=1}^{\infty} \sum_{m=1}^{\infty}\sum_{k=0}^{\min(n,m)-1}\EE\left(A_{n,m,k}\II_{S_{n-1}\leq t}\II_{S_{m-1}^{(k)} \leq t+\tau}\right),
\end{eqnarray*}
 for suitable integer sequence $A_{n,m,k}$ expressed only by the numbers of children as is defined  in  Subsection \ref{subsection3.1}.
 We study in Corollary  \ref{QM},  the mean quadratic and the almost sure convergence of $e^{-\nu t}N_t$ as $t$ tends to infinity.
The purpose of Subsection \ref{three} is to discuss Corollary's  \ref{QM} assumptions, yielding the almost sure convergence of $e^{-\nu t}N_t$ as $t$ tends to infinity, in the particular case where lifetimes and new cells numbers are independent and new cells numbers are  modelled by a sequence of iid random variables. A main step to check  in this particular case is to establish (see Lemma \ref{lemnux}),
\begin{eqnarray}\label{limPPPP}
&& \left| e^{-\nu t}\sum_{n\geq 0}\kappa_1^{n}(\kappa_1-1) \PP(\sum_{i=1}^{n+1}\xi(X_i)\leq t|X_0=x)-C_0(x)\right|\leq \frac{\Psi_0(x)}{2\pi}e^{-\delta t}.
\end{eqnarray}
The bound (\ref{limPPPP}) is the main tool to obtain (see Proposition \ref{PPPPP} for more details),
$$
\EE(N_t)= e^{\nu t}\EE(\kappa(X_0))\EE(e^{-\nu \xi(X_0)}{\tilde C}_0(X_0))[1+ O(e^{-\epsilon_1 t})],\,\,as\,\,t\rightarrow \infty,\,\,{\mbox{for some}}\,\, \epsilon_1>0.
$$
 Under additional assumptions, the bound (\ref{limPPPP}) allows also to obtain (see  Proposition \ref{PP}),
$$
\EE(N_tN_{t+\tau})=  e^{\nu (2t+\tau)}\kappa_2\sum_{k=0}^{\infty}\kappa_1^{k}\EE({\tilde C}^2_0(X_k)e^{-2\nu S_{k}})[1+ae^{-\epsilon_1 t} ],\,\,as\,\,t\rightarrow \infty,
$$
where $\kappa_2=\EE(\kappa(X_1)(\kappa(X_1)-1))$ and where $a, \epsilon_1$ are positive constants independent of $t$ and $\tau$. Lemma \ref{lemnux} gives sufficient conditions ensuring (\ref{limPPPP}). Theorem \ref{theo1} studies the convergence in mean quadratic and almost surely of $e^{-\nu t}N_t$ to a random variable $W$ and gives the expressions of the first and of the second moment of this limit $W$. In Subsection \ref{four} we discuss mainly the conditions of Lemma \ref{lemnux} (and then sufficient conditions for the bound (\ref{limPPPP})).

The rest of the sections are devoted to the proofs of Theorems \ref{generalspectraltheorem1} and \ref{generalspectraltheorem2} dedicated to the multiplicative ergodicity property for additive Markov chains. In Section \ref{results}, we discuss the tools and the assumptions needed to those two main theorems. We study in particular the monotonicity, the positivity and the derivative of the spectral radius $r(\gamma)$ of the Laplace kernel $P_{\gamma}$.
In Section \ref{Knudsengas}, we apply our general method to a toy model of Knudsen gas in the particular case where $\kappa\equiv 2$. The proofs concerning the less elementary case of linear autoregressive models
are given in Section \ref{proofAR}. Theorems \ref{generalspectraltheorem1} and \ref{generalspectraltheorem2} are completely proved in Section \ref{proofoperator}. Section \ref{proof-reduc-BL} gives the proof of Proposition~\ref{pro-reduc-BL} needed to check Hypothesis 2.5 which is one of the main assumptions of Theorems \ref{generalspectraltheorem1} and \ref{generalspectraltheorem2}. Finally, Section \ref{counterexample} gives an example of a kernel $P_{\gamma}$
for which its spectral radius is greater than one and hence the growth rate $\nu$ does not exist.

\section{Behavior of the first moment, multiplicative ergodicity, examples} \label{first}
\subsection{First moment of $N_t$} \label{Gen-res}
The following proposition evaluates, for any $t>0$, the expectation of $N_t$  when it exists  in terms of
the lifetimes $(\xi(X_i))_{i\geq 0}$ and of the numbers of new cells $(\kappa(X_i))_{i\geq 0}$.
\begin{apro}\label{pro:ENt}
Let $t> 0$ be fixed. If
$\sum_{n\ge 0}\mathbb E\left[\left(\prod_{j=0}^n\kappa(X_j)\right)
  \mathbf 1_{\{S_n\le t\}}\right]
  <\infty$, then $\EE(N_t)<\infty$ and
$$ \mathbb E(N_t)=1+\sum_{n\ge 0}\mathbb E\left[\left(\prod_{j=0}^{n-1}\kappa(X_j)\right)\left(\kappa(X_n)-1\right)
   \mathbf 1_{\{S_n\le t\}}\right]$$
(with the usual convention $\prod_{j=0}^{-1}\kappa(X_j)=1$).
\end{apro}
\begin{proof}
For every $n\ge 0$, we write $\Sigma_n(t)$ for the number of cells of generation $n$
alive at time $t$.
Let us write $D_0$ and $T_0$ for respectively the number of children and
the lifetime of the initial cell.
For every $k\in\{1,\cdots,D_0\}$, we write $x_{0,k}$ for the parameter
(resp. $D_{0,k}$ and $T_{0,k}$ for the
number of children and the lifetime) of the $k$-th child of the initial cell.
More generally, we write  $x_{0,k_1,\cdots,k_n}$ for the parameter
(resp. $D_{0,k_1,\cdots,k_n}$ and $T_{0,k_1,\cdots,k_n}$ for the
number of children and the lifetime) of the cell of the $n$-th generation which is the
$k_n$-th child of the $k_{n-1}$-th child of the ... of the $k_1$-th child of the initial cell.
Observe that
$\mathbb E[\Sigma_0(t)]=\mathbb P(\xi(X_0)> t)$
and that, for every $n\ge 1$,
\begin{eqnarray*}
&\ &\mathbb E[\Sigma_n(t)]=\mathbb E\left[\sum_{k_1=1}^{D_0}\sum_{k_2=1}^{D_{0,k_1}}\cdots
       \sum_{k_n=1}^{D_{0,k_1,\cdots,k_{n-1}}}
         \mathbf 1_{\{ T_0+T_{0,k_1}+\cdots+T_{0,k_1,\cdots,k_{n-1}}\le t<
             T_0+T_{0,k_1}+\cdots+T_{0,k_1,\cdots,k_{n}} \}}\right]\\
&=&\mathbb E\left[D_0\mathbb E\left[\left.\sum_{k_2=1}^{D_{0,1}}\cdots
       \sum_{k_n=1}^{D_{0,1,k_2,\cdots,k_{n-1}}}
         \mathbf 1_{\{ T_0+T_{0,1}+\cdots+T_{0,1,k_2,\cdots,k_{n-1}}\le t<
             T_0+T_{0,1}+\cdots+T_{0,1,k_2,\cdots,k_{n}} \}}\right|X_0\right]\right]\\
&=&\mathbb E\left[D_0D_{0,1}\mathbb E\left[\sum_{k_3=1}^{D_{0,1,1}}\left.\cdots
       \sum_{k_n=1}^{D_{0,1,1,k_3,\cdots,k_{n-1}}}
         \mathbf 1_{\{ T_0+T_{0,1}+\cdots+T_{0,1,1,k_3,\cdots,k_{n-1}}\le t<
             T_0+T_{0,1}+\cdots+T_{0,1,1,k_3,\cdots,k_{n}} \}}\right|X_0,X_{0,1}\right]\right]\, .
\end{eqnarray*}
Hence, for every $n\ge 1$,
\begin{eqnarray*}
\mathbb E[\Sigma_n(t)]
&=&\mathbb E\left[D_0D_{0,1}\cdots D_{0,1^{n-1}}
         \mathbf 1_{\{ T_0+T_{0,1}+\cdots+T_{0,1^{n-1}}\le t<
             T_0+T_{0,1}+\cdots+T_{0,1^n} \}}\right]\\
&=&\mathbb E\left[\left(\prod_{j=0}^{n-1}\kappa(X_j)\right)
         \left(\mathbf 1_{\{ S_{n-1}  \le t\}}-
            \mathbf 1_{\{S_{n}\le t\}}\right)\right].
\end{eqnarray*}
Since $\sum_{n\ge 0}\mathbb E\left[\left(\prod_{j=0}^n\kappa(X_j)\right)
  \mathbf 1_{\{S_n\le t\}}\right]
  <\infty$
and $N_t= \mathbf 1_{T_0>t}+ \sum_{n=1}^{\infty} \Sigma_n(t)$ a.~s., we obtain
\begin{eqnarray*}
\mathbb E[N_t]
&=& \mathbb P(\xi(X_0)> t)+ \mathbb E\left[\kappa(X_0)\mathbf 1_{\xi(X_0)\leq t}\right]+\sum_{n\ge 1}\mathbb E\left[\left(\prod_{j=0}^{n-1}\kappa(X_j)\right)(\kappa(X_n)-1)
         \mathbf 1_{\{ S_n  \le t\}}\right]   \\
 &=&   1+   \mathbb E\left[(\kappa(X_0)-1)\mathbf 1_{\xi(X_0)\leq t}\right] + \sum_{n\ge 1}\mathbb E\left[\left(\prod_{j=0}^{n-1}\kappa(X_j)\right)(\kappa(X_n)-1)
 \mathbf 1_{\{ S_n  \le t\}}\right].
\end{eqnarray*}
\end{proof}

It follows from  Proposition \ref{pro:ENt} that $\mathbb E[N_t]<\infty$ if $\nu$ defined by \eqref{nuDEF} is finite
(using the fact that $\mathbf 1_{\{S_n\le t\}}\le e^{\gamma t}
   e^{-\gamma S_n}$). Now using Proposition \ref{pro:ENt} and arguing exactly as for the proof of Theorem 2.1 in \cite{LouhichiYcart15}, we obtain the following exponential behavior in mean of $\EE(N_t)$ in a very general setting of dependence
with the use of the function $G$ given by
\begin{equation}\label{nu-gn}
G(\gamma) := \sum_{n\ge 0}
g_n(\gamma),\quad \mbox{recall that}\quad
g_n(\gamma)=\mathbb E\left[\left(\prod_{j=0}^{n-1}\kappa(X_j)\right)(\kappa(X_n)-1)  e^{-\gamma S_n}\right].
\end{equation}
%
\begin{acor}\label{cor1}
Assume that
$\nu<\infty$ and that the following limit exists
\begin{equation}\label{CDEF}
C_\nu:=\lim_{\gamma\rightarrow 0} \frac{\gamma}{\gamma+\nu}G(\nu+\gamma)\, ,
\end{equation}
\begin{equation}\label{CVMOYESP}
\mbox{then,}\quad\quad\quad\quad\quad\lim_{t\rightarrow \infty}\frac{1}{t}\int_0^te^{-\nu s}\EE(N_s)ds= C_\nu.
\end{equation}
\end{acor}
\subsection{Multiplicative ergodicity, application to Markov chains}\label{sub22}
%
{{We adapt the notion of  "multiplicative ergodicity", as introduced in \cite{KM03} and \cite{KM05}, to our context.}}
\begin{adefi} \label{def-mult-erg}
Let $\gamma_1>0$.
We say that $(S_n,\kappa(X_n))_n$ is {\bf multiplicatively ergodic} on $J=[0,\gamma_1)$ if
there exist two continuous maps $A$ and $\rho$ from $J$ to $(0,+\infty)$ such that,
for every compact subset $K$ of $(0,\gamma_1)$, there exist $M_K>0$ and $\theta_K\in(0,1)$
such that, for every $n\ge 1$,
\begin{equation} \label{formule-mult-erg}
 \forall \gamma\in K,\quad |g_n(\gamma)-A(\gamma) (\rho(\gamma))^n|\le M_K (\rho(\gamma)\theta_K)^n.
\end{equation}
When $\kappa(\cdot)$ is constant, we will simply say that $(S_n)_n$ is  multiplicatively ergodic on $J$.
\end{adefi}
\begin{arem}\label{multergodP1P2}
Assume that $(S_n,\kappa(X_n))_n$
is multiplicatively ergodic on  $J=[0,\gamma_1)$. Then
\begin{itemize}
	\item For every $\gamma\in J$ we have:
	$G(\gamma) = \sum_{n\ge 0}
g_n(\gamma) <\infty \, \Longleftrightarrow\,  \rho(\gamma)<1.$
  \item For every compact subset $K$ of $J$, we obtain from the definition of $\nu$ in (\ref{nuDEF}) that
$$\forall \gamma\in K\cap(\nu,+\infty),
\quad
 \left|G(\gamma)-\frac{A(\gamma)}{1- \rho(\gamma)}\right|\le\frac{M_K}{1- \rho(\gamma)\theta_K}.$$
  \item $\nu<\gamma_1$ means that
\begin{equation}\label{P1ter}
\nu=\inf\{\gamma\in J\ :\ \rho(\gamma)<1\}<\gamma_1.
\end{equation}
  \item If moreover $\rho$ is differentiable at $\nu$ with
$\rho(\nu)=1$ and $\rho'(\nu)\ne 0$, then
\eqref{CDEF} follows with $C_\nu=-\frac {A(\nu)}{\nu \rho'(\nu)}$.
Actually, to obtain \eqref{P1ter}, we can relax the
continuity assumptions on $A$ and $\rho$ on $J=[0,\gamma_1)$. For \eqref{CDEF}, we just need the continuity of $A$ and the differentiability of $\rho$ at $\nu$ (with $\rho'(\nu)\ne 0$).
\end{itemize}
\end{arem}
We investigate now
the geometric ergodicity property in the following context
of additive functional of Markov chains.
We assume throughout this section that $X=(X_n)_n$ is a Markov chain on $(\X,\cX)$
with Markov kernel $P(x,dy)$, invariant probability $\pi$,
and initial probability $\mu$ (i.e.~$\mu$ is the distribution of $X_0$).
Assume moreover that, for every $n\geq 1$, the random variable $\prod_{j=0}^{n}\kappa(X_j)$ is integrable.
We set $h_{\kappa,\gamma} := \big(\kappa - 1\big)\, e^{-\gamma\xi}$. Let $\gamma\in(0,+\infty)$. For $n\geq 1$,
\begin{eqnarray*}
g_n(\gamma)&=&\mathbb E\left[\left(\prod_{j=0}^{n-1}\kappa(X_j)e^{-\gamma\xi(X_j)}\right)
   h_{\kappa,\gamma}(X_n)\right]\\
&=&\mathbb E\left[\left(\prod_{j=0}^{n-1}\kappa(X_j)e^{-\gamma\xi(X_j)}\right)
   (Ph_{\kappa,\gamma})(X_{n-1})\right]\, ,
\end{eqnarray*}
with $(Ph)(x) :=\int_{\mathbb X}h(y)\, P(x,dy)$.
If $n\ge 2$, we continue and obtain
$$ g_n(\gamma)=\mathbb E\left[\left(\prod_{j=0}^{n-2}\kappa(X_j)e^{-\gamma\xi(X_j)}\right)
   (P_\gamma(Ph_{\kappa,\gamma}))(X_{n-2})\right]\, ,
$$
with $P_\gamma h:=P(h\kappa e^{-\gamma\xi})$.
An easy induction gives
\begin{equation}\label{formuleFourier}
\forall n\ge 1,\quad g_{n}(\gamma) = \mu\left(\kappa\, e^{-\gamma \xi}\,
P_\gamma^{n-1} \left(Ph_{\kappa,\gamma}\right)\right).
\end{equation}
We also define $P_{\infty}h:=P(h\kappa\mathbf 1_{\{\xi=0\}})$.

Before stating our spectral assumptions on $P_\gamma$,  we need to introduce some standard notations. When $(\cB,\|\cdot\|_{\cB})$ and $(\cB_1,\|\cdot\|_{\cB_1})$ are two Banach spaces, the space of continuous $\mathbb C$-linear operators from $\cB$ to $\cB_1$ will be written $\mathcal L(\cB,\cB_1)$. We simply write $\mathcal L(\cB)$ for $\mathcal L(\cB,\cB)$, and the topological dual space of $\cB$ is denoted by $(\cB^*,\|\cdot\|_{\cB^*})$. When, for some $\gamma\in[0,+\infty]$, the kernel $P_\gamma$  continuously acts on $\cB$, the spectral radius of $P_{\gamma|\cB}$ and its essential spectral radius are denoted by $r(\gamma)$ and $r_{ess}(P_\gamma)$ respectively, that is:
$$r(\gamma) := r(P_{\gamma|\cB}) = \lim_n\|(P_{\gamma|\cB})^n\|_{\cB}^{1/n} \quad \text{and} \quad
r_{ess}(P_\gamma):=\lim_n \inf_{F\in\cL(\cB)\ \mbox{\scriptsize compact}}\|(P_\gamma)^n-F\|^{1/n}_{\cB}.$$
\begin{acond}\label{hypcompl}
Let $\mathcal B$ be a Banach space composed of functions on $\X$ (or of classes of such functions modulo the $\pi-$almost sure equality) such that $\cB\subset \mathbb L^1(\pi)$. Let $J$ be a subinterval of $[0,+\infty]$. We assume that, for every $\gamma\in J$, $P_\gamma$ continuously acts on $\cB$ and that
\begin{enumerate}[(i)]
\item $r(\gamma):=r(P_{\gamma|\cB})>0$, and $P_\gamma$ is quasi-compact on $\cB$ (i.e.~$r_{ess}\big(P_{\gamma|\cB}\big)< r(\gamma)$)
\item $r(\gamma)$ is the only eigenvalue of modulus $r(\gamma)$ for $P_\gamma$, and $r(\gamma)$ is a first order pole of $P_\gamma$ with moreover $\dim\ker(P_\gamma-r(\gamma) I) = \dim\ker(P_\gamma-r(\gamma) I)^2 = 1$.
\end{enumerate}
\end{acond}
Theorem~\ref{cor-th1} will ensure that, under Hypothesis \ref{hypcompl}, for every $\gamma\in J$ there exists a rank-one non-negative projector $\Pi_\gamma\in\cL(\cB)$ (i.e.~the eigenprojector associated with the eigenvalue $r(\gamma)$), and some constants $\theta_\gamma\in(0,1)$ and $M_\gamma\in(0,+\infty)$ such that
\begin{equation} \label{sup-vit-ponctuel}
\forall\gamma\in J,\quad\forall f\in\cB,\quad
\big\|P_\gamma^n f - r(\gamma)^n \Pi_\gamma f\big\|_{\cB} \leq M_\gamma\, \big(\theta_\gamma\, r(\gamma)\big)^n \|f\|_{\cB}.
\end{equation}
If moreover $\mu(\kappa e^{-\gamma\xi}\cdot)\in\cB^*$,
$P(h_{\kappa,\gamma})\in\mathcal B$,
$B(\gamma) := \mu(\kappa\, e^{-\gamma \xi}\Pi_\gamma(Ph_{\kappa,\gamma}))$ is positive, then we deduce \eqref{formule-mult-erg} with $A=A/r$ and $\rho=r$ from \eqref{formuleFourier} in the specific case $K=\{\gamma\}$.
\begin{arem} \label{rem-method}
To establish the multiplicative ergodicity of Definition~\ref{def-mult-erg}, further regularity properties are needed. Due to (\ref{sup-vit-ponctuel}), a natural way is to apply the perturbation theory of linear operators. Unfortunately, the classical operator perturbation method
{{\cite{nag1,nag2,GuivarchHardy,GuivarchLepage}}}.
does not apply to our context. Indeed, because we do not assume any exponential moment condition on $\xi$ (contrarily to the papers mentioned in Introduction), the map $\gamma\mapsto P_\gamma$ is (in general) {\bf not continuous} from $(0,+\infty)$ to $\mathcal L(\mathcal B)$. For instance, for linear autoregressive models (Theorem \ref{thmAR}), we will work with Banach spaces
$\mathcal B_a=\mathcal C_{V^a}$ linked to some weighted-supremum Banach spaces.
For the Knudsen gas (Theorem \ref{pro-Knudsen}), we will work with $\mathcal B_a=\mathbb L^{a}(\pi)$. In these two cases, the map $\gamma\mapsto P_\gamma$ is
not continuous in general from $(0,+\infty)$ to $\mathcal L(\mathcal B_{a})$, but only
from $(0,+\infty)$  to $\mathcal L(\mathcal B_a,\mathcal B_b)$ for $a<b$ for the linear autoregressive models (and for $b<a$ for the Knudsen gas). This is the reason why we use the Keller-Liverani perturbation theorem \cite{KelLiv99}. The price to pay is to consider "a chain of Banach spaces" instead of a single one.
\end{arem}
We introduce two sets of assumptions (Hypotheses \ref{hypKL} and \ref{hypKL}*): both of them will be relevant for  our examples in Sections~\ref{Knudsengas} and \ref{proofAR} (Knudsen gas and linear autoregressive model). Below the notation $\mathcal B_0\hookrightarrow \mathcal B_1$ means that $\cB_0$ is continuously injected in $\mathcal B_1$.

\begin{acond} \label{hypKL}
Let $\cB_0$ and $\cB_1$ be two Banach spaces, let $J$ be
a subinterval of $[0,+\infty]$. We will say that $(P_\gamma,J,\cB_0,\cB_1)$ satisfies Hypothesis~\ref{hypKL} if
\begin{itemize}
\item $\cB_0 \hookrightarrow \cB_1$,
\item for every $\gamma\in J$, $P_\gamma\in\cL(\cB_0)\cap\cL(\cB_1)$,
\item the map $\gamma\mapsto P_\gamma$ is continuous from $J$ to
$\cL(\cB_0,\cB_1)$,
\item there exist $c_0>0$, $\delta_0>0$, $M>0$ such that
\begin{subequations}
\begin{equation} \label{cond-r-ess-direct}
\forall \gamma\in J,\quad r_{ess}\big(P_{\gamma|\cB_0}\big)\le\delta_0
\end{equation}
\begin{equation} \label{D-F-direct}
\forall \gamma\in J,\ \forall n\geq 1,\ \forall f\in\cB_0,\quad
\|P_\gamma^n f\|_{\cB_0}\le c_0\big(\delta_0^n\| f\|_{\cB_0}+M^n\| f\|_{\cB_1}\big)
\end{equation}
\end{subequations}
\end{itemize}
\end{acond}

\noindent{\bf Hypothesis~\ref{hypKL}*.}
{\it $(P_\gamma,J,\cB_0,\cB_1)$ satisfies all the conditions of Hypothesis~\ref{hypKL}, except for (\ref{cond-r-ess-direct}) and (\ref{D-F-direct}) which are replaced by the following ones:
\begin{subequations}
\begin{equation} \label{cond-r-ess-dual}
\forall \gamma\in J,\quad r_{ess}\big((P_{\gamma}^*)_{|\cB_1^*}\big)\le\delta_0
\end{equation}
\begin{equation} \label{D-F-dual}
\forall \gamma\in J,\ \forall n\geq 1,\ \forall f^*\in\cB_1^*,\quad
\|(P_\gamma^*)^n f^*\|_{\cB_1^*}\le c_0(\delta_0^n\| f^*\|_{\cB_1^*} + M^n\| f^*\|_{\cB_0^*})
\end{equation}
\end{subequations}
}

\noindent Hypothesis~\ref{hypKL}* can be seen as a dual version of
Hypothesis~\ref{hypKL}, but it is worth noticing that the conditions (\ref{cond-r-ess-dual})-(\ref{D-F-dual}) cannot be deduced from (\ref{cond-r-ess-direct})-(\ref{D-F-direct}) (and conversely). Under Hypothesis \ref{hypKL} or \ref{hypKL}* we define the following set:
\begin{equation} \label{J0}
J_0:=\{\gamma\in J : r(\gamma)>\delta_0\}.
\end{equation}
\begin{atheo}\label{cor-th1}\label{generalspectraltheorem1}
Let $\cB_0\hookrightarrow\cB_3\hookrightarrow \L^1(\pi)$ be two Banach spaces such that $\mathbf 1_\X\in\cB_0$, let $J$ be a subinterval of $[0,+\infty]$. Assume that $(P_\gamma,J,\cB_0,\cB_3)$ satisfies Hypothesis \ref{hypKL} or \ref{hypKL}* and that
\begin{itemize}
	\item Hypothesis \ref{hypcompl} holds on $J_0$ and $\cB :=\cB_0$ under Hypothesis \ref{hypKL}
	\item Hypothesis \ref{hypcompl} holds on $J_0$ and $\cB :=\cB_3$ under Hypothesis \ref{hypKL}*.
\end{itemize}
Then
\begin{equation} \label{thmkellerliverani1}
\forall \gamma_0\in J,\quad \limsup_{\gamma\rightarrow \gamma_0}r(\gamma)\le \max(\delta_0,r(\gamma_0))\, ,
\end{equation}
and the function $\gamma\mapsto r(\gamma) := r(P_{\gamma|\cB})$ is continuous on $J_0$. Moreover there exists a map $\gamma\mapsto \Pi_\gamma$ from $J_0$ to $\cL(\cB)$ which is continuous from $J_0$ to
$\mathcal L(\cB_0,\cB_3)$ such that, for every compact subset $K$ of $J_0$, there exist $\theta_K\in(0,1)$ and $M_K\in(0,+\infty)$ such that
\begin{equation} \label{sup-vit}
\forall \gamma\in K,\ \forall f\in\cB,\quad
\big\|P_\gamma^n f - r(\gamma)^n \Pi_\gamma f\big\|_{\cB} \leq M_K\, \big(\theta_K\, r(\gamma)\big)^n \|f\|_{\cB}.
\end{equation}
Consequently, under the previous assumptions, the following assertions hold:
\begin{enumerate}[(i)]
	\item If the maps $\gamma\mapsto P h_{\kappa,\gamma}$ and $\gamma\mapsto\mu(\kappa e^{-\gamma\xi}\cdot)$ are continuous from $J_0$ to $\mathcal B_0$ and to $\cB_3^*$ respectively, and if
\begin{equation} \label{def-B}
\forall \gamma\in J_0,\quad B(\gamma) := \mu\left(\kappa\, e^{-\gamma \xi}\Pi_\gamma(Ph_{\kappa,\gamma})\right)>0,
\end{equation}
then $(S_n,\kappa(X_n))_n$ is multiplicatively ergodic on $J_0$ with $A(\gamma):=\frac{B(\gamma)}{r(\gamma)}$ and $\rho(\gamma)=r(\gamma)$.
  \item If moreover $\displaystyle\inf_{\gamma\in J_0} r(\gamma) <1<\sup_{\gamma\in J_0} r(\gamma)$,
then $\nu$ is finite and
\begin{equation}\label{P1bis}
\nu=\inf\{\gamma>0\, :\, r(\gamma)<1\}.
\end{equation}
\end{enumerate}
\end{atheo}
Formula \eqref{sup-vit} can be interpreted as a
{\bf spectral multiplicative ergodicity property}. To prove the existence of $C_\nu$,
we reinforce our assumptions by considering a longer chain of Banach spaces.
\begin{atheo} \label{cor-th2}\label{generalspectraltheorem2}
Assume $\pi(\xi>0)>0$. Let $\cB_0\hookrightarrow\cB_1\hookrightarrow\cB_2\hookrightarrow \cB_3 \hookrightarrow \L^1(\pi)$ be Banach spaces containing $\mathbf 1_{\mathbb X}$ and let $J$ be a subinterval of $[0,+\infty]$. Assume that one of the two following conditions holds
\vspace*{-2mm}
\begin{enumerate}[(a)]
\item Either: for $i=0,1,2$, $(P_\gamma,J,\cB_i,\cB_{i+1})$ satisfies Hypothesis \ref{hypKL}, and Hypothesis~\ref{hypcompl} holds with
$(
J_0
,\cB_{i})$  ; in this case we set $\cB:=\cB_0$.
\item Or: for $i=0,1,2$, $(P_\gamma,J,\cB_i,\cB_{i+1})$ satisfies Hypothesis \ref{hypKL}*, and Hypothesis \ref{hypcompl} holds with
$(J_0,\cB_{i+1})$ ; in this case we set $\cB:=\cB_3$.
\end{enumerate}
Assume moreover that the map $\gamma\mapsto P_\gamma$ is continuous from $J$ to $\cL(\cB_i,\cB_{i+1})$ for $i\in\{0,2\}$ and $C^1$ from $J$ to $\cL(\cB_1,\cB_2)$ with derivative $P'_\gamma f=P_\gamma(-\xi f)$ ($f\mapsto \xi f$ being in $\cL(\cB_1,\cB_2)$).
Then (\ref{sup-vit}) holds with $C^1$-smooth maps $\gamma\mapsto r(\gamma) := r(P_{\gamma|\cB})$ and $\gamma\mapsto\Pi_\gamma$ from $J_0$ into $\R$ and into $\cL(\cB_0,\cB_3)$ respectively.
Consequently, under the previous assumptions, the assertions $(i)$-$(ii)$ in Theorem~\ref{generalspectraltheorem1} can be specified and completed as follows:
\begin{enumerate}[(i')]
	\item If the additional assumptions in Assertion~$(i)$ of Theorem~\ref{generalspectraltheorem1} hold with the present spaces $\cB_0$ and $\cB_3$, then the functions $A(\cdot)$ and $\rho(\cdot):=r(\cdot)$ are $C^1$-smooth on $J_0$.
	\item If moreover  $\inf_{\gamma\in J_0} r(\gamma) <1<\sup_{\gamma\in J_0} r(\gamma)$
and if $r'(\nu)\neq 0$, then the constant $C_\nu$ of \eqref{CDEF} is well defined and finite, and Property~\eqref{CVMOYESP} holds true.
\end{enumerate}
\end{atheo}
Note that the two above results require, not only to check the spectral property (\ref{cond-r-ess-direct}) (or (\ref{cond-r-ess-dual})) and the
Doeblin-Fortet inequalities \eqref{D-F-direct} (or \eqref{D-F-dual}), but also to prove \eqref{def-B} and moreover $r'(\nu)>0$ in Theorem~\ref{generalspectraltheorem2}. This will be discussed in section \ref{results}.

\subsection{Markovian examples} \label{mark-ex}
Here we present two examples which will be derived from the
general results of the previous section
(see Sections \ref{Knudsengas} and \ref{proofAR} for the proofs).
The first one is a toy model on which our general spectral assumptions
are easily checked. In this model, at each step, either we follow
a Markov chain $Z=(Z_n)_n$ (with probability $(1-\alpha)$) or
we generate an independent random variable with distribution
the invariant probability measure of $Z$ (with probability $\alpha$). {{ See \cite {BoattoGolse} for more about this model}}.
\begin{aex}[Knudsen gas] \label{exemp-Knud}
Let $\X:=\R^d$, let $\pi$ be some Borel probability measure on $\X$, and let $U$ a Markov operator with stationary probability $\pi$. We fix $\alpha\in(1/2,1)$.
Let $X=(X_n)_n$ be a Markov chain with transition kernel $P:= \alpha \pi + (1-\alpha)\, U$.
\end{aex}
\begin{atheo}
\label{pro-Knudsen}
Assume that $X=(X_n)_n$ is a Knudsen gas with initial distribution admitting a density with respect to $\pi$, which belongs to $\mathcal L^p(\pi)$ for some $p>1$. Moreover assume that $\kappa\equiv 2$ and that $\pi(\xi>0)=1$. Then $\nu$ defined by \eqref{nuDEF} is finite. If, moreover, $\pi(\xi^\tau)<\infty$ for some $\tau\in(1,p/(p-1))$, then
\eqref{CDEF} is well defined and Property~\eqref{CVMOYESP} holds with $C_\nu\in(0,+\infty)$.
\end{atheo}
\begin{aex}[Linear autoregressive model] \label{exemp-AR}
$\X:=\R$ and
$X_n = \alpha X_{n-1} + \vartheta_n$ for $n\ge 1$,
where $X_0$ is a real-valued random variable, $\alpha\in(-1,1)$, and $(\vartheta_n)_{n\ge 1}$ is a sequence of  i.i.d.~real-valued random variables independent of $X_0$.   Let $r_0>0$. We assume that $\vartheta_1$ has a continuous Lebesgue probability density function $p>0$ on $\X$
satisfying the following condition:
for all $x_0\in \R$, there exist a neighbourhood $V_{x_0}$ of $x_0$ and a non-negative function $q_{x_0}(\cdot)$ such that $y\mapsto (1+|y|)^{r_0}\, q_{x_0}(y)$ is Lebesgue-integrable and such that :
\begin{equation} \label{dom-nu}
\forall y\in\R,\ \forall v\in V_{x_0},\ p(y+v) \leq q_{x_0}(y).
\end{equation}
Note that $\vartheta_1$ admits a moment of order $r_0$.
\end{aex}
Recall that $\xi:\mathbb R\rightarrow[0,+\infty)$ is said to be coercive if $\lim_{|x|\rightarrow +\infty}\xi(x)=+\infty$, i.e.~if, for every $\beta$, the set  $[\xi\le\beta]$ is bounded.
\begin{atheo}[Linear autoregressive model]\label{thmAR}
Assume that $X=(X_n)_n$ is a linear autoregressive model satisfying the above assumption and that the distribution of $X_0$ is either the stationary probability measure $\pi$ or $\delta_x$ for some $x\in\mathbb R$. Let $N_0$ be a positive integer.
Assume that
$\kappa$ is bounded, that
$\xi$ is coercive, that the Lebesgue measure of the set $[\xi=0]$ is zero,
and that $\sup_{x\in\mathbb R}\frac{\xi(x)}{(1+|x|)^{r_0}}<\infty$.

Then, $\nu$ given by \eqref{nuDEF} is well defined (and is independent of the choice of the distribution of $X_0$ as above). If moreover there exists $\tau>0$ such that  $\sup_{x\in\mathbb R}\frac{\xi(x)^{1+\tau}}
{(1+|x|)^{r_0}}<\infty$, then the constant $C_\nu$
given by \eqref{CDEF}
is well defined in $(0,+\infty)$ and Property~\eqref{CVMOYESP} holds.
\end{atheo}
Recall that $\int_\mathbb R |x|^{r_0}\, d\pi(x)<\infty$ under the assumptions of Theorem \ref{thmAR}
(see \cite{diaco,duf}). Hence  $\sup_{x\in\mathbb R}\frac{\xi(x)^{1+\tau}}
{(1+|x|)^{r_0}}<\infty$ implies that $\int_\mathbb R|\xi|^{1+\tau}\, d\pi<\infty$.

\section{Behavior of the second moment and almost sure convergence}\label{second}
\subsection{Second moment and applications}\label{subsection3.1}
We make some stationarity assumption.
\begin{acond}\label{spatialstat}
For each $k\in \N$, there exists
a process $X^{(k)}=(X_n^{(k)})_{n\geq 0}$ such that
\begin{equation} \label{Xk}
\left\{
\begin{array}{ccc}
(X^{(k)}_n)_{0\leq n\leq k}= (X_n)_{0\leq n\leq k}& \,\, {\mbox{a.s.}} \\
(X_n^{(k)})_{n\geq 0} = (X_n)_{n\geq 0} & \,\, {\mbox{ in law}},
\end{array}
\right.
\end{equation}
and such that, for every couple of lines of cells $((v_n)_n,(w_n)_n)$ coinciding up to the $k$-th generation but not at the $(k+1)$-th
generation, the corresponding sequence of parameters have the same distribution as
$(X,X^{(k)})$.
\end{acond}
Now define, for any integers $n \geq 1, m\geq 1$ and $\min(n,m)-1\ge k\ge 0$ the random variables $A_{n,m,k}$ as follows:
\begin{eqnarray*}
A_{n,m,k}
&=& \left(\prod_{i=0}^{n-2}\kappa(X_i)\right)
\left(\prod_{j=\min(k+1,n-1)}^{m-2}\kappa(X^{(k)}_j)\right)\\
&\ &\ \ \ \ \ \ \ \ \ \ \
\left(\prod_{j\in\{k\}\setminus\{n-1,m-1\}}(\kappa(X_j)-1)\right) \left(\kappa(X_{n-1})-1\right)\left(\kappa(X^{(k)}_{m-1})-1\right),
\end{eqnarray*}
with the usual convention $\prod_{i=k+1}^{\ell}\cdots=1$
if $\ell\le k$.
Define also $S_{n}^{(k)}:=\sum_{j=0}^n\xi(X^{(k)}_j)$.
The main result of this section is the following proposition.
\begin{apro}\label{pro1}
Assume that Hypothesis \ref{spatialstat} holds. Let $t>0$ and $\tau\geq 0$ be fixed. If $\E(N_{t+\tau})<\infty$, then
\begin{equation}\label{Eqpro1}
 \EE[N_tN_{t+\tau}]
 =\E[N_{t}]+\E[N_{t+\tau}]-1+ \sum_{n=1}^{\infty} \sum_{m=1}^{\infty}\sum_{k=0}^{\min(n,m)-1}\EE\left[A_{n,m,k}
\mathbf 1_{\{S_{n-1}\leq t,S_{m-1}^{(k)} \leq t+\tau\}}\right]. \end{equation}
\end{apro}
\begin{proof}
We have, using the notations of the proof of Proposition \ref{pro:ENt},
$N_t=\mathbf 1_{\{T_0>t\}}+ \sum_{n=1}^{\infty} \Sigma_n(t) \quad a.s.$,
with
$$\Sigma_n(t)= \sum_{k_1=1}^{D_0}\sum_{k_2=1}^{D_{0,k_1}}\cdots
       \sum_{k_n=1}^{D_{0,k_1,\cdots,k_{n-1}}}
         \mathbf 1_{\{ T_0+T_{0,k_1}+\cdots+T_{0,k_1,\cdots,k_{n-1}}\le t\}}- \mathbf 1_{\{
             T_0+T_{0,k_1}+\cdots+T_{0,k_1,\cdots,k_{n}}\le t \}}.$$
Now due to $\E(N_{t})<\infty$ (see the proof of Proposition \ref{pro:ENt}),
$$\sum_{k_1=1}^{D_0}\sum_{k_2=1}^{D_{0,k_1}}\cdots
       \sum_{k_n=1}^{D_{0,k_1,\cdots,k_{n-1}}}
         \mathbf 1_{\{ T_0+T_{0,k_1}+\cdots+T_{0,k_1,\cdots,k_{n-1}}\le t\}}<\infty\quad a.s..$$
Therefore
$N_t=1+ \sum_{n=1}^{\infty} {\tilde{\Sigma}}_n(t)\quad a.s.\, ,$
with
$$\tilde\Sigma_n(t)= \sum_{k_1=1}^{D_0}\sum_{k_2=1}^{D_{0,k_1}}\cdots
       \sum_{k_{n-1}=1}^{D_{0,k_1,\cdots,k_{n-2}}}
         (D_{0,k_1,\cdots,k_{n-1}}-1)
         \mathbf 1_{\{ T_0+T_{0,k_1}+\cdots+T_{0,k_1,\cdots,k_{n-1}}\le t \}}.$$
Consequently,
\begin{eqnarray}
&& N_tN_{t+\tau}= 1+ \sum_{n=1}^{\infty} {\tilde{\Sigma}}_n(t)+ \sum_{n=1}^{\infty} {\tilde{\Sigma}}_n(t+\tau)+ \sum_{n=1}^{\infty} \sum_{m=1}^{\infty}{\tilde{\Sigma}}_n(t){\tilde{\Sigma}}_m(t+\tau) {\nonumber}\\
&&= N_t+ N_{t+\tau}-1+ \sum_{n=1}^{\infty} \sum_{m=1}^{\infty}{\tilde{\Sigma}}_n(t){\tilde{\Sigma}}_m(t+\tau)\, .
{\nonumber}
\end{eqnarray}
For any positive integers $n,m$ and $t>0,\tau\geq 0$, we have,
$$
\tilde\Sigma_n(t)\tilde\Sigma_m(t+\tau)=
\sum_{k= 0}^{\min(n-1,m-1)}\sum_{(\boldsymbol  \ell,\boldsymbol {\tilde  \ell})\in E_{n,m,k}} (D_{\boldsymbol\ell}-1)(D_{\boldsymbol{\tilde\ell}}-1)
         \mathbf 1_{\{ \mathcal S_{n-1}(\boldsymbol  \ell)\le t ,\mathcal S_{m-1}(\boldsymbol  {\tilde \ell})\le t+\tau \}}
$$
where $E_{n,m,k}$ is the set of {{$(\mathbf \ell,\mathbf{\tilde \ell})$,}}
with $\boldsymbol \ell=(0,\ell_1,...,\ell_{n-1})\in(\mathbf N^*)^{n}$
and $\boldsymbol {\tilde \ell}=(0,\tilde \ell_1,...,\tilde \ell_{m-1})\in(\mathbf N^*)^{m}$ having the same coordinates up to time $k$, i.e. such that {{$\min\{j=0,...,\min(n,m)\ :\ \ell_j\neq \tilde\ell_j\}=k+1$,}}
with the notation
$\mathcal S_{n-1}(\boldsymbol \ell):=T_0+T_{0,\ell_1}+\cdots+T_{0,\ell_1,\cdots,\ell_{n-1}}$.
We conclude by proceding exactly as in the proof of Proposition \ref{pro:ENt}.
\end{proof}

As it was done in \cite{H} in the case of independence
(cf. Lemma 19.1 and Theorem 21.1 there), Proposition \ref{pro1} is the main ingredient for the proofs of the quadratic mean  and of the almost sure convergence of $e^{-\nu t}N_t$, for the growth rate $\nu$ already defined in (\ref{nuDEF}).
\begin{acor}\label{QM}
Assume Hypothesis \ref{spatialstat}, that
$\nu<\infty$, that
$
\limsup_{t\rightarrow \infty}e^{-\nu t}\EE[N_t]<\infty
$
and that
there exists $K>0$ such that
\begin{equation}\label{Erreurttau}
\lim_{t\rightarrow \infty}\sup_{\tau\ge 0}\left|e^{-\nu(2t+\tau)}\sum_{n=1}^{\infty} \sum_{m=1}^{\infty}\sum_{k=0}^{\min(n,m)-1}\EE\left[A_{n,m,k}
\mathbf 1_{\{S_{n-1}\leq t,S_{m-1}^{(k)} \leq t+\tau\}}\right]-K\right|=0.
\end{equation}
Then
there exists a square integrable random variable $W$ such that $(e^{-\nu t}N_t)_{t\geq 0}$ converges in quadratic mean to $W$ as $t$ tends to infinity.

If moreover the convergence in \eqref{Erreurttau} is exponentially fast
and if $W>0$, then
 $(e^{-\nu t}N_t)_{t\geq 0}$ converges almost surely to $W$ as $t$ tends to infinity.
\end{acor}
\begin{proof}[Proof of Corollary \ref{QM}] Due to Corollary
\ref{QM} and Proposition \ref{pro1},
$$
\lim_{t\rightarrow \infty} \EE\left[\left(e^{-\nu t}N_t-e^{-\nu(t+\tau)}N_{t+\tau}\right)^2\right]=0,
$$
for any $\tau\geq 0$, uniformly in $\tau$.
The Cauchy criterion ensures then the convergence in quadratic mean of $e^{-\nu t}N_t$  as $t$ tends to infinity to a random variable $W$ with finite second moment.

For the last point, we deduce from Proposition \ref{pro1} that
$
\int_0^{\infty}\EE\left[\left(e^{-\nu t}N_t-W\right)^2\right]dt<\infty.
$
This yields (arguing as for the proof of Theorem 21.1 in \cite{H}) the almost sure convergence, as $t$ tends to infinity, of $e^{-\nu t}N_t$   to $W$.
\end{proof}
We will apply these results in the two following sections under some
additional
independence assumptions.
\subsection{Some extensions of Harris' results}\label{three}
For further results, we will make the following stronger assumption.
\begin{acond}\label{hypoHarris}
$(X_n)_n$ is a stationary Markov process,
$(\kappa(X_n))_{n}$ is a sequence of i.i.d. square integrable random variables of expectation $\kappa_1$, which is independent
of $(\xi(X_n))_n$.
Moreover, for all $k\in\mathbb N$, $(X_n^{(k)})_{n\ge k+1}$
and $(X_n)_{n\ge k+1}$ are independent given $X_k$ and $\nu$ (as defined in (\ref{nuDEF})) satisfies,
\begin{equation}\label{nux}
\forall x\in\X,\quad\nu=\inf\left\{\gamma>0,\,\, \sum_{n\geq 0}\kappa_1^{n} \EE\left[e^{-\gamma S_{n+1}}|X_0=x\right]<\infty\right\}<\infty\, .
\end{equation}
We set $\kappa_2:=\mathbb E [\kappa(X_1)(\kappa(X_1)-1)]$.
\end{acond}
Define $f_{x,0}(t)=(\kappa_1-1)e^{-\nu t}\sum_{n\geq 0}\kappa_1^{n} \PP\left(S_{n+1}-S_0\leq t|X_0=x\right).$
We will make the following assumption involving the Laplace transform ${\tilde f}_{x,0}$ of $f_{x,0}$:
\begin{equation}\label{Laplacedef}
\forall \gamma>0,\quad
{\tilde f}_{x,0}(\gamma)=\int_0^{\infty}e^{-\gamma t}f_{x,0}(t)dt=\frac{\kappa_1-1}{\gamma+\nu}\sum_{n\geq 0}\kappa_1^{n} \EE\left[e^{-(\gamma+\nu) (S_{n+1}-S_0)}|X_0=x\right]\, .
\end{equation}
\begin{acond}\label{analytiq}
Suppose that
there exist two positive reals $\delta<\nu$ and $\epsilon$ such that, for any $x$, the Laplace transform $
{\tilde f}_{x,0}$, extended on the complex plane, satisfies the following conditions:
\begin{enumerate}
\item ${\tilde f}_{x,0}$ is analytic in $\{z = u + iy , |u| <\delta+\epsilon,\,y\in \RR\} \setminus \{0\}$,
\item ${\tilde f}_{x,0}$ has a simple pole at $0$, with residue $\tilde C_0(x)$,
\item $\int_{-\infty}^{+\infty}| {\tilde f}_{x,0}(\delta+ iy)| dy <\infty$ ,
\item
$\lim_{y \rightarrow \pm \infty}
˜{\tilde f}_{x,0}(u + iy) = 0 $,
uniformly in $u \in [-\delta, \delta]$,
\item $\Psi_0(x):= \int_{-\infty}^{+\infty}|{\tilde f}_{x,0}(-\delta + iy)| dy < \infty$ .
\end{enumerate}
\end{acond}
We generalize the approach of \cite{LouhichiYcart15}
to obtain the following result
extending \cite[Theorem 19.1]{H}
to the case where the lifetimes are dependent.
\begin{atheo}\label{theo1}
Assume that Hypotheses \ref{hypoHarris}
and \ref{analytiq} are satisfied and that
$$\EE\left[e^{-\nu \xi(X_0)}\tilde C_0(X_0)\right]+\EE\left[e^{-(\nu-\delta) \xi(X_0)}\Psi_0(X_0)\right]+\sum_{k=0}^{\infty}\kappa_1^{k}
\EE\left[\tilde C^2_0(X_k)e^{-2\nu S_k}\right]<\infty \, ,$$
\begin{eqnarray}\label{Htheo1}
\mbox{and}&& \sum_{k=0}^{\infty}\kappa_1^k \EE\left[\tilde C_0(X_k)\Psi_0(X_k)e^{-(2\nu-\delta) S_k}+\Psi^2_0(X_k)e^{-2(\nu-\delta) S_{k}}\right]<\infty, {\nonumber}\\
&& \sum_{k=0}^{\infty}\kappa_1^k \EE\left[(\tilde C_0(X_k)+\Psi_0(X_k)+k)e^{-(\nu+\delta) S_{k}}\right]<\infty,
\end{eqnarray}
then
there exists a square integrable random variable $W$ such that $(e^{-\nu t}N_t)_{t\geq 0}$ converges in quadratic mean to $W$ as $t$ tends to infinity, with
\begin{eqnarray*}
&& \EE[W]=\kappa_1\EE\left[e^{-\nu \xi(X_0)}\tilde C_0(X_0) \right] \\
&& {\mbox{Var}}(W) = \kappa_2 \sum_{k=0}^{\infty}\kappa_1^{k}
\EE\left[\tilde C^2_0(X_k)e^{-2\nu S_k}\right]- \kappa_1^2\left(\EE\left[e^{-\nu \xi(X_0)}\tilde C_0(X_0)\right]\right)^2.
\end{eqnarray*}
If, moreover, $W>0$ almost surely then  $e^{-\nu t}N_t$ converges almost surely to $W$.
\end{atheo}
The rest of this subsection is devoted to the proof of Theorem \ref{theo1}. Define, for $r\leq m$, the partial sums
$ S_{r,m}:= \sum_{i=r}^m\xi(X_i),$ so that $S_n=S_{0,n}$ and $S_{r,m}^{(k)}:= \sum_{i=r}^m\xi(X_i^{(k)})$.
The following lemma is very useful in order to prove the exponential speed of convergence in
\eqref{Erreurttau} and then to obtain the almost sure convergence of $e^{-\nu t}N_t$. Its proof is an immediate consequence of Lemma  2.2  in \cite{LouhichiYcart15} and we omit it.
\begin{alem}\label{lemnux}
Assume Hypothesis \ref{analytiq}.
Then, for any $t>0$,
\begin{eqnarray}\label{limP}
&& \left| e^{-\nu t}\sum_{n\geq 0}\kappa_1^{n}(\kappa_1-1) \PP(S_{1,n+1}\leq t|X_0=x)-\tilde C_0(x)\right|\leq \frac{\Psi_0(x)}{2\pi}e^{-\delta t}.
\end{eqnarray}
\end{alem}
Lemma \ref{lemnux} together with Proposition \ref{pro:ENt} allows to give an exact asymptotic behavior of $\EE(N_t)$ as $t$ tends to infinity, as shows the following proposition.
\begin{apro}\label{PPPPP} Suppose that $(X_n)_{n\geq 0}$ is a stationary sequence of random variables, that $(\kappa(X_i))_{i\geq 0}$ is a sequence of i.i.d. random variables with finite first moment. Let $\kappa_1=\EE(\kappa(X_1))$. Suppose moreover that $(\kappa(X_i))_{i\geq 0}$ is independent of $(\xi(X_i))_{i\geq 0}$, that $\nu$ satisfies (\ref{nux}) and that Hypothesis \ref{analytiq} is satisfied for some positive $\delta$ strictly less than $\nu$. If $\EE(e^{-\nu \xi(X_0)}\tilde C_0(X_0))<\infty$ and if $\EE(e^{-(\nu-\delta) \xi(X_0)}\Psi_0(X_0))<\infty$,
then $\EE(N_t)<\infty$ and there exists $\epsilon_1>0$
such that
$$
\EE(N_t)= e^{\nu t}\kappa_1\EE(e^{-\nu \xi(X_0)}\tilde C_0(X_0))[1+ O(e^{-\epsilon_1 t})],\,\,as\,\,t\rightarrow \infty\, .
$$
\end{apro}
\begin{proof}
Recall that $
f_{x,0}(u)= e^{-\nu u}\sum_{n\ge 1}\kappa_1^{n-1}(\kappa_1-1)
   \mathbb P(S_{1,n}\le u|X_0=x)$. Moreover
\begin{eqnarray*}
&&\sum_{n\ge 1}\kappa_1^n(\kappa_1-1)
   \mathbb P(S_{n}\le t)= 
 \E\left[\sum_{n\ge 1}\kappa_1^n(\kappa_1-1)\mathbb P(S_{1,n}\le t-\xi(X_0)|X_0)\mathbf 1_{\{\xi(X_0)\leq t\}}\right]   \\
&&=    e^{\nu t}\kappa_1\E\left[ e^{-\nu \xi(X_0)} f_{X_0,0}(t-\xi(X_0))\mathbf 1_{\{\xi(X_0)\leq t\}} \right] \\
&&= e^{\nu t}\kappa_1\E\left[ e^{-\nu \xi(X_0)}\left( f_{X_0,0}(t-\xi(X_0))-\tilde C_0(X_0)\right)\mathbf 1_{\{\xi(X_0)\leq t\}} \right]+e^{\nu t}\kappa_1\E\left[ e^{-\nu \xi(X_0)}\tilde C_0(X_0)\mathbf 1_{\{\xi(X_0)\leq t\}} \right].
\end{eqnarray*}
Now Lemma \ref{lemnux} gives,
\begin{eqnarray*}
&& \E\left[ e^{-\nu \xi(X_0)}\left( f_{X_0,0}(t-\xi(X_0))-\tilde C_0(X_0)\right)\mathbf 1_{\{\xi(X_0)\leq t\}} \right] \leq \frac{e^{-\delta t}}{2\pi} \E\left[ \Psi_0(X_0)e^{-(\nu-\delta) \xi(X_0)}\right].
\end{eqnarray*}
Consequently $\sum_{n\ge 1}\kappa_1^n(\kappa_1-1)
   \mathbb P(S_{n}\le t) <\infty$ by the requirements of this proposition and
\begin{eqnarray}\label{f}
&&\sum_{n\ge 1}\kappa_1^n(\kappa_1-1)
   \mathbb P(S_{n}\le t) = e^{\nu t}\kappa_1\left[\E\left( e^{-\nu \xi(X_0)}\tilde C_0(X_0)\mathbf 1_{\{\xi(X_0)\leq t\}}\right)+e^{-\delta t} D(t)\right]\hspace{1cm}
\end{eqnarray}
with $0<D(t) \leq ({2\pi})^{-1} \E\left( \Psi_0(X_0)e^{-(\nu-\delta) \xi(X_0)}\right)$. This allows to deduce
thanks to Proposition \ref{pro:ENt} that $\EE(N_t)<\infty$ and
\begin{eqnarray*}
&& \mathbb E(N_t)= e^{\nu t}\kappa_1\EE [e^{-\nu \xi(X_0)}\tilde C_0(X_0)\mathbf 1_{\{\xi(X_0)\leq t\}}]\left[1+e^{-\nu t}A(t)+ e^{-\delta t}D(t)\right],
\end{eqnarray*}
with $0<A(t)<\sup_{t>0}A(t)<\infty, 0<D(t)<\sup_{t>0}D(t)<\infty$. \end{proof}

The following proposition gives an exact asymptotic behavior of  $\EE(N_tN_{t+\tau})$ under further assumptions.
Theorem \ref{theo1} follows directly from the following proposition
combined with Corollary \ref{QM}.
\begin{apro}\label{PP}
Assume that Hypotheses \ref{hypoHarris}
and \ref{analytiq} are satisfied. Suppose also that
$\sum_{k=0}^{\infty}\kappa_1^{k}\EE(\tilde C^2_0(X_k)e^{-2\nu S_{k}})<\infty $. If Condition (\ref{Htheo1}) is satisfied,
then, for any $t>0$, $\tau\geq 0$, $\EE(N_tN_{t+\tau})<\infty$ and $$
\EE(N_tN_{t+\tau})=  e^{\nu (2t+\tau)}\kappa_2\sum_{k=0}^{\infty}\kappa_1^{k}\EE\left(\tilde C^2_0(X_k)e^{-2\nu S_{k}}\right)[1+ae^{-\epsilon_1 t} ],\,\,as\,\,t\rightarrow \infty,
$$
where $a$ and $\epsilon_1$ are positive constants independent of $t$ and $\tau$.
\end{apro}
\begin{proof}
The task is to apply Proposition \ref{pro1}. For $0\leq k\leq \min(n,m)-1 $, we have
\begin{eqnarray*}
&& \EE\left[A_{n,m,k}\mathbf 1_{\{S_{n-1}\leq t,\,\,\,S_{m-1}^{(k)} \leq t+\tau\}}\right]= \EE\left[A_{n,m,k}\right]\PP\left({S_{n-1}\leq t},\,\, {S_{m-1}^{(k)} \leq t+\tau}\right).
\end{eqnarray*}
Observe that
$$
\EE\left[A_{n,m,k}\right]\\
= \left\{ \begin{array}{ccc}
       \mathbb E[(\kappa(X_0)-1)^2] \kappa_1^{n-1}
              \quad\quad{\mbox{ if }}  k=n-1=m-1 \\
      \mathbb \kappa_2(\kappa_1-1)\kappa_1^{\max(n-2,m-2)}
   \quad\quad{\mbox{ if }}  (k= m-1\ ou \ k= n-1),\ n\ne m\\
     \kappa_2(\kappa_1-1)^2\kappa_1^{n-2+m-2-k}
     \quad\quad{\mbox{ if }}  (k\ne  m-1\ et \ k\ne n-1)  \, .
      \end{array}
      \right.
     $$
Now we have for $k\neq m-1$,
\begin{eqnarray*}
&&\PP\left({S_{n-1}\leq t},\,\, {S_{m-1}^{(k)} \leq t+\tau}\right)= \\
&&\mathbb E\left[ \PP\left({S_{k+1,n-1}\leq t-S_k},\,\,\,S_{k+1,m-1}^{(k)} \leq t+\tau-S_k|S_k,X_k\right) \mathbf 1_{\{S_k\leq t\}}\right] \\
&&= \mathbb E\left[ \PP\left(S_{k+1,n-1}\leq t-S_k|S_k,X_k\right)  \PP\left(S_{k+1,m-1}^{(k)} \leq t+\tau-S_k|S_k,X_k\right)\mathbf 1_{\{S_k\leq t\}}\right].
\end{eqnarray*}
When $k=m-1$ (and then necessarily $m\leq n$)
$$\PP\left({S_{n-1}\leq t},\,\, {S_{m-1} \leq t+\tau}\right) =\PP\left({S_{n-1}\leq t}\right).
$$
Consequently,
\begin{eqnarray*}
&& \sum_{n=1}^{\infty} \sum_{m=1}^{n}\EE\left[A_{n,m,m-1}\mathbf 1_{\{S_{n-1}\leq t,\ S_{m-1}+S_{m,m-1}^{(m-1)} \leq t+\tau\}}\right]\\
&&= \sum_{n=1}^{\infty} \sum_{m=1}^{n-1}\kappa_1^{n-2}(\kappa_1-1)\kappa_2\PP\left({S_{n-1}\leq t}\right)+ \sum_{n=1}^{\infty}\kappa_1^{n-1}\E[(\kappa(X_0)-1)^2]
\PP\left({S_{n-1}\leq t}\right){\nonumber}\\
&&=  (\kappa_1-1)\kappa_2\sum_{n=1}^{\infty}(n-1)\kappa_1^{n-2}\PP\left({S_{n-1}\leq t}\right) + \E[(\kappa(X_0)-1)^2]\sum_{n=1}^{\infty}\kappa_1^{n-1}
\PP\left({S_{n-1}\leq t}\right),{\nonumber}
\end{eqnarray*}
and for any $0<\delta<\nu$
\begin{eqnarray}\label{F2}
&& e^{-\nu(2t+\tau)}\sum_{n=1}^{\infty} \sum_{m=1}^{n}\EE\left[A_{n,m,m-1}\mathbf 1_{\{S_{n-1}\leq t,\ S_{m-1}+S_{m,m-1}^{(m-1)} \leq t+\tau\}}\right]{\nonumber}\\
&& \leq e^{-t(\nu-\delta)}(\kappa_1-1)\kappa_2\sum_{n=1}^{\infty}(n-1)\kappa_1^{n-2}\EE[e^{-(\nu+\delta)S_{n-1}}]{\nonumber} \\
&&+ e^{-t(\nu-\delta)}\E[(\kappa(X_0)-1)^2]
\sum_{n=1}^{\infty}\kappa_1^{n-1}\EE[e^{-(\nu+\delta)S_{n-1}}].
\end{eqnarray}
We also have,
\begin{eqnarray}\label{ll}
&&\sum_{n=2}^{\infty} \sum_{m=2}^{\infty}\sum_{k=0}^{\min(n,m)-2} \EE\left[A_{n,m,k}\mathbf 1_{\{S_{n-1}\leq t,\, S_{k}+S_{k+1,m-1}^{(k)} \leq t+\tau\}}\right]\\
&&= \kappa_2\sum_{k=0}^{\infty}\kappa_1^k \mathbb E\left[\sum_{n\geq k+2}\kappa_1^{n-2-k}(\kappa_1-1) \PP\left(S_{k+1,n-1}\leq t-S_k|S_k,X_k\right)\right.\nonumber\\
&&\times \left. \sum_{m\geq k+2}\kappa_1^{m-2-k}(\kappa_1-1)\PP\left(S_{k+1,m-1}^{(k)} \leq t+\tau-S_k|S_k,X_k\right)\mathbf 1_{\{S_k\leq t\}}\right]\nonumber
\end{eqnarray}
Now the bound (\ref{limP}) gives,
\begin{eqnarray*}
&& \left|e^{-\nu(t-\sum_{i=0}^k\xi(x_i))}\sum_{n\geq k+2}\kappa_1^{n-2-k}(\kappa_1-1)\PP\left({S_{k+1,n-1}\leq t-\sum_{i=0}^k\xi(x_i)}|X_k=x_k\right)-\tilde C_0(x_k)\right|\\
&&\leq \Psi_0(x_k)e^{-\delta (t-\sum_{i=0}^k\xi(x_i))}\\
&& \left|e^{-\nu(t+\tau-\sum_{i=0}^k\xi(x_i))}\sum_{m\geq k+2}\kappa_1^{m-2-k}(\kappa_1-1)\PP\left(S_{k+1,m-1}^{(k)} \leq t+\tau-\sum_{i=0}^k\xi(x_i)|X_k=x_k\right)-\tilde C_0(x_k)\right| \\
&&\leq \Psi_0(x_k)e^{-\delta (t+\tau-\sum_{i=0}^k\xi(x_i))}.
\end{eqnarray*}
The two last bounds together with (\ref{ll}) give then
\begin{eqnarray}\label{EqproPPP}
&& \sum_{n=2}^{\infty} \sum_{m=2}^{\infty}\sum_{k=0}^{\min(n,m)-2}\EE
\left[A_{n,m,k}\mathbf 1_{\{{S_{n-1}\leq t},\,\, {S_{k}+S_{k+1,m-1}^{(k)} \leq t+\tau}\}}\right]\nonumber\\
&&= e^{\nu(2t+\tau)}\kappa_2\left[\sum_{k=0}^{\infty}\kappa_1^k\EE(\tilde C^2_0(X_k)e^{-2\nu S_{k}}\mathbf 1_{S_{k}\leq t}) + e^{-\delta t} a\right],
\end{eqnarray}
where
$
a \leq 2 \sum_{k=0}^{\infty}\kappa_1^k \EE(\tilde C_0(X_k)\Psi_0(X_k)e^{-(2\nu-\delta) S_{k}}) + e^{-\delta t}\sum_{k=0}^{\infty}\kappa_1^k \EE(\Psi^2_0(X_k)e^{-2(\nu-\delta) S_{k}}).
$ \\
Finally,
\begin{eqnarray}\label{I4}
&& \sum_{n=1}^{\infty} \sum_{m=n+1}^{\infty} \EE\left(A_{n,m,n-1}\mathbf 1_{\{S_{n-1}\leq t,\ S_{n-1}+S_{n,m-1}^{(k)} \leq t+\tau\}}\right) \\
&&=\kappa_2(\kappa_1-1)\sum_{n=1}^{\infty} \mathbb E\left[\sum_{m=n+1}^{\infty}\kappa_1^{m-2} \PP\left(S_{n,m-1}^{(n-1)}\leq t+\tau -S_{n-1}|S_{n-1},X_{n-1}\right) \mathbf 1_{\{S_{n-1}\leq t\}}\right]\nonumber\\
&&=\kappa_2(\kappa_1-1) \mathbb E\left[\sum_{n=1}^{\infty}\kappa_1^{n-1} \mathbf 1_{\{S_{n-1}\leq t\}}\sum_{m=n+1}^{\infty}\kappa_1^{m-n-1} \PP\left(S_{n,m-1}^{(n-1)}\leq t+\tau -S_{n-1}|S_{n-1},X_{n-1}\right)\right]\nonumber
\end{eqnarray}
Now the bound (\ref{limP}) gives
\begin{eqnarray}\label{I3}
&&\left|(\kappa_1-1)e^{-\nu(t+\tau-\sum_{i=0}^{n-1}\xi(x_i))}\sum_{m\geq n+1}\kappa_1^{m-n-1}\right. {\nonumber}\\
&& \left.\PP\left(S_{n,m-1}^{(n-1)} \leq t+\tau-\sum_{i=0}^{n-1}\xi(x_i)|X_{n-1}=x_{n-1}\right)-\tilde C_0(x_{n-1})\right|{\nonumber} \\
&&\leq \Psi_0(x_{n-1})e^{-\delta (t+\tau-\sum_{i=0}^{n-1}\xi(x_i))}.
\end{eqnarray}
We have also,
\begin{equation}\label{I2}
 e^{-\nu t}\sum_{n\geq 1}\kappa_1^{n-1}\E\left(\tilde C_{0}(X_{n-1})e^{-\nu S_{n-1}}\mathbf 1_{\{S_{n-1}\leq t\}}\right)
\leq e^{-(\nu-\delta) t}\sum_{n\geq 1}\kappa_1^{n-1}\E\left[\tilde C_{0}(X_{n-1})e^{-(\nu+\delta) S_{n-1}}\right],
\end{equation}
and
\begin{eqnarray}\label{I1}
&&e^{-\nu t}\E\left[\sum_{n\geq 1}\kappa_1^{n-1}\mathbf 1_{\{S_{n-1}\leq t\}}\Psi_0(X_{n-1})e^{-\nu S_{n-1}}e^{-\delta (t+\tau-S_{n-1})}\right]{\nonumber}\\
&&\leq e^{-(\nu+\delta)t}\sum_{n\geq 1}\kappa_1^{n-1}\E\left(\mathbf 1_{\{S_{n-1}\leq t\}}\Psi_0(X_{n-1})e^{-(\nu-\delta) S_{n-1}}\right){\nonumber}\\
&&\leq e^{-(\nu-\delta)t}\sum_{n\geq 1}\kappa_1^{n-1}\E\left[\Psi_0(X_{n-1})e^{-(\nu+\delta) S_{n-1}}\right]
\end{eqnarray}
We get collecting Inequalities (\ref{I4}), (\ref{I3}), (\ref{I2}) and (\ref{I1})
\begin{eqnarray}\label{F1}
&& e^{-\nu(2t+\tau)}\sum_{n=1}^{\infty} \sum_{m=n+1}^{\infty} \EE\left(A_{n,m,n-1}\mathbf 1_{\{S_{n-1}\leq t,\ S_{n-1}+S_{n,m-1}^{(k)} \leq t+\tau\}}\right){\nonumber}\\
&&\leq  e^{-(\nu-\delta)t}\kappa_2\sum_{n\geq 1}\kappa_1^{n-1}\E\left((\tilde C_{0}(X_{n-1})+\Psi_0(X_{n-1}))e^{-(\nu+\delta) S_{n-1}}\right).
\end{eqnarray}
We then obtain combining (\ref{F2}), (\ref{EqproPPP}) and (\ref{F1}),
\begin{eqnarray}\label{EqproPP}
&&e^{-\nu(2t+\tau)}\sum_{n=1}^{\infty} \sum_{m=1}^{\infty}\sum_{k=0}^{\min(n,m)-1}\EE\left[A_{n,m,k}
\mathbf 1_{\{S_{n-1}\leq t,\,\, {S_{k}+S_{k+1,m-1}^{(k)} \leq t+\tau}\}}\right]{\nonumber} \\
&&= \kappa_2\sum_{k=0}^{\infty}\kappa_1^k\EE[\tilde C^2_0(X_k)e^{-2\nu S_{k}}\mathbf 1_{\{S_{k}\leq t\}}] + e^{-(\nu-\delta) t} A+ e^{-\delta t}\kappa_2a,
\end{eqnarray}
where $a$ is as defined in (\ref{EqproPPP}) and
\begin{eqnarray*}
&& A\leq \kappa_2\sum_{n\geq 1}\kappa_1^{n-1}\E\left[(\tilde C_{0}(X_{n-1})+\Psi_0(X_{n-1}))e^{-(\nu+\delta) S_{n-1}}\right]{\nonumber}\\
&&+(\kappa_1-1)\kappa_2\sum_{n=1}^{\infty}(n-1)\kappa_1^{n-2}\EE(e^{-(\nu+\delta)S_{n-1}})+ \E[(\kappa(X_0)-1)^2]\sum_{n=1}^{\infty}\kappa_1^{n-1}
\EE[e^{-(\nu+\delta)S_{n-1}}].
\end{eqnarray*}
Consequently, we obtain
 collecting (\ref{EqproPP}) together with  (\ref{Eqpro1}) and Proposition \ref{PPPPP}, that there exists $\epsilon_1>0$,
such that, for any $t,\tau\geq 0$,
\begin{eqnarray*}\label{EqproP}
&&\EE[N_tN_{t+\tau}] = e^{\nu(2t+\tau)}\kappa_2\sum_{k=0}^{\infty}\kappa_1^k
\EE[\tilde C^2_0(X_k)e^{-2\nu S_{k}}](1+c_{t,\tau}e^{-\epsilon_1 t}),\,\,{\mbox{as}}\,\,t\rightarrow\infty,
\end{eqnarray*}
where  $\sup_{t,\tau}c_{t,\tau}<\infty$.
\end{proof}
\begin{proof}[Proof of Theorem \ref{theo1}]
The bound \eqref{EqproPP} proves that the convergence in \eqref{Erreurttau} is exponentially fast and satisfied with
$$
K=\kappa_2\sum_{k=0}^{\infty}\kappa_1^k\EE[\tilde C^2_0(X_k)e^{-2\nu S_{k}}].
$$
The proof of Theorem \ref{theo1} is complete by using Corollary~\ref{QM}, Proposition~\ref{PP}, Proposition~\ref{PPPPP},
and the fact that,
\begin{eqnarray*}
&& \EE[W]=\lim_{t\rightarrow \infty}\EE[e^{-\nu t}N_t)]=\kappa_1\EE[e^{-\nu \xi(X_0)}\tilde C_0(X_0)],\,\,{\mbox{by Proposition \ref{PPPPP},}}\\
&& \EE[W^2]=\lim_{t\rightarrow \infty}\EE [e^{-2\nu t}N_t^2]=\kappa_2\sum_{k=0}^{\infty}\kappa_1^{k}
\EE[\tilde C^2_0(X_k)e^{-2\nu S_{k}}]\,\, {\mbox{by Proposition \ref{PP}}}.
\end{eqnarray*}
\end{proof}%
\subsection{About Hypothesis \ref{analytiq}}\label{four}
The purpose of this section is to discuss Hypothesis \ref{analytiq} yielding to the key bound (\ref{limP}). We assume Hypothesis \ref{hypoHarris} along this paragraph.
\subsubsection{The i.i.d. case}
 If moreover the lifetimes are i.i.d.  then the growth rate $\nu$ as defined in (\ref{nux})
is also that defined in \eqref{(H)}, that is
$
\kappa_1\EE\left[e^{-\nu \xi(X_0)}\right]=1.
$
The following lemma gives additional assumptions ensuring (\ref{limP}) in this i.i.d. case. Although this case is classical, its study is important to make comparison with previous results (obtained with other methods of proofs). Also proofs for more general results will be obtained in the spirit of this classical case, see the paragraph below.

\begin{alem}[i.i.d. case]\label{lemiid}
Assume Hypothesis \ref{hypoHarris}, that $(\xi(X_n))_n$
is a sequence of i.i.d. random variables admitting
a density $g\in\cL^p$ for some $p>1$. If, for some $M>0,\,\nu>\delta>0$,
$\inf_{y\in \RR}\left|1-\kappa_1\EE\left[e^{-(-\delta+\nu+iy)\xi(X_0)}\right]\right|>0$ and $\inf_{|y|\geq M}\inf_{u\in[-\delta,\delta]}|1-\kappa_1\EE\left[e^{-(u+\nu+iy)\xi(X_0)}\right]|>0$,
 then Hypothesis \ref{analytiq} is satisfied  for any $x\in\X$, with
\begin{eqnarray*}
&& \tilde C_0(x)=\frac{\kappa_1-1}{\kappa_1^2\nu}\left(\EE\left[\xi(X_0)e^{-\nu \xi(X_0)}\right]\right)^{-1},\\
&& \Psi_0(x)\leq {\mbox{Const.}}\left(\inf_{y\in \RR}\left|1-\kappa_1\EE\left[e^{-(-\delta+\nu+iy)\xi(X_0)}\right]\right|\right)^{-1}\left[\int_{-\infty}^{+\infty} g^p(u)du\right]^{1/(p-1)}.
\end{eqnarray*}
\end{alem}
\noindent
\begin{rem} Note that if $\xi(X_0)$ is exponentially distributed with parameter $\lambda$, then for any $\delta>0$ sufficiently small,
\begin{eqnarray*}
&& \nu=\lambda(\kappa_1-1)\\
&& \inf_{y\in \RR}\left|1-\kappa_1\EE\left[e^{-(-\delta+\nu+iy)\xi(X_0)}\right]\right|>0\\
&& \inf_{|y|\geq M}\inf_{u\in[-\delta,\delta]}\left|1-\kappa_1\EE\left[e^{-(u+\nu+iy)\xi(X_0)}\right]\right|\geq {M\lambda\kappa_1}(\lambda+\delta)^{-1}((\delta+\lambda +\nu)^2+M^2)^{-1/2}>0
\end{eqnarray*}
\end{rem}
\begin{rem} Due to Lemma \ref{lemiid}, Propositions \ref{PPPPP} and \ref{PP} are respectively  Theorem 17.2 and Lemma 18.1 in \cite{H}
($m$, $h''(1)$ and $n_1$ there being respectively $\kappa_1$, $\kappa_2$ and the constant function $\tilde C_0$).
\end{rem}
\begin{proof}[Proof of Lemma \ref{lemiid}]
The task is to check the conditions on the Laplace transform ${\tilde f}_{x,0}$ (as calculated in (\ref{Laplacedef})). In this i.i.d. case, we have, for $\gamma=s+iy$ and $s>0$,
$$
{\tilde f}_{x,0}(\gamma)=\frac{\kappa_1-1}{\gamma+\nu}\frac{\EE[e^{-(\gamma+\nu)\xi(X_0)}]}{1-\kappa_1\EE[e^{-(\gamma+\nu)\xi(X_0)}]}.
$$
The right hand side of the last equality is analytic in a sufficiently narrow strip $\{z=u+iy, |u|\leq \delta+\epsilon,y\in \RR\}\setminus\{0\}$ for $\delta+\epsilon<\nu$. It has a simple pole at $0$ with residue $\tilde C_0(x)$, since
$$
\lim_{z\rightarrow 0}z{\tilde f}_{x,0}(z)=\frac{\kappa_1-1}{\kappa_1^2\nu}\left(\EE\left[\xi(X_0)e^{-\nu \xi(X_0)}\right]\right)^{-1}=\tilde C_0(x).$$ We have, since $|1-\kappa_1\EE(e^{-(\delta+\nu+iy)\xi(X_0)})|$ is bounded below by a strictly positive constant (this follows from $\sup_{y\in \RR}|\EE[e^{-(\delta+\nu+iy)\xi(X_0)}]|< \EE[e^{-\nu\xi(X_0)}]=1/\kappa_1 $), and letting $p'=p/(p-1)$,
\begin{eqnarray*}
&& \int_{-\infty}^{+\infty} |{\tilde f}_{x,0}(\delta+iy)|dy \leq {\mbox{Cst}} \int_{-\infty}^{+\infty} \frac{|\EE[e^{-(\delta+iy+\nu)\xi(X_0)}]|}{\sqrt{(\nu+\delta)^2+y^2}}dy\\
&& \leq {\mbox{Cst}} \left(\int_{-\infty}^{+\infty} |\EE[e^{-(\delta+iy+\nu)\xi(X_0)}]|^{p'}dy\right)^{1/p'} \left(\int_{-\infty}^{+\infty} ((\nu+\delta)^2+y^2)^{-p/2}dy\right)^{1/p}.
\end{eqnarray*}
Now we have, arguing as for the proof of Lemma 3 in Harris (1963) page 163,
$$
\int_{-\infty}^{+\infty} |\EE[e^{-(\delta+iy+\nu)\xi(X_0)}]|^{p'}dy \leq \tilde C_p\left[\int_{-\infty}^{+\infty} e^{-p(\delta+\nu)t}g^p(t)dt\right]^{1/(p-1)}.
$$
Consequently $\int_{-\infty}^{+\infty} |{\tilde f}_{x,0}(\delta+iy)|dy<\infty$ since the density $g$ is supposed to be in ${\cL}^p$. Using the same arguments, we prove that $\int_{-\infty}^{+\infty} |{\tilde f}_{x,0}(-\delta+iy)|dy<\infty$, in fact,
\begin{eqnarray*}
&&\int_{-\infty}^{+\infty} |{\tilde f}_{x,0}(-\delta+iy)|dy\\
&& \leq {\kappa_1 \tilde C_p}\left(\inf_{y\in \RR}\left|1-\kappa_1\EE[e^{-(-\delta+\nu+iy)\xi(X_0)}]\right|\right)^{-1}\left[\int_{-\infty}^{+\infty} g^p(t)dt\right]^{1/p} \left(\int_{-\infty}^{+\infty} ((\nu-\delta)^2+y^2)^{-p/2}dy\right)^{1/p},
\end{eqnarray*}
which is finite by the requirements of the lemma.
 Now, we have for any $u\in [-\delta, \delta]$ and for any $|y|>M>0$,
\begin{eqnarray*}
&& |˜{\tilde f}_{x,0}(u + iy)|\leq {(\kappa_1-1)}((\nu+u)^2+y^2)^{-1/2}\left|1-\kappa_1\EE[e^{-(u+\nu+iy)\xi(X_0)}]\right|^{-1}\\
&& \leq \frac{\kappa_1}{|y|}\left(\inf_{|y|\geq M}\inf_{u\in[-\delta,\delta]}\left|1-\kappa_1\EE[e^{-(u+\nu+iy)\xi(X_0)}]\right|\right)^{-1}
\end{eqnarray*}
this proves that,\, for all $M>0$,\,
$
\sup_{|y|>M}\sup_{u\in [-\delta, \delta]}\left|{\tilde f}_{x,0}(u + iy)\right|<\infty
$
and then
 $\lim_{y \rightarrow \pm \infty}
˜{\tilde f}_{x,0}(u + iy) = 0 $,
uniformly in $u \in [-\delta, \delta]$. Hypothesis (\ref{analytiq}) is then satisfied.
\end{proof}
\subsubsection{Multiplicative ergodic case}
 Hypothesis \ref{analytiq} can also be satisfied by  Markov Chains   having  multiplicative ergodic sums (cf. Definition 3.3 and Section 3 in \cite{LouhichiYcart15}), the multiplicative ergodic property supposes that, for any $\gamma >0$, any $x \in\mathbb X$ and any $n\in \NN$,
    \begin{equation}\label{GE}
 \EE\left[e^{-\gamma S_{1,n+1}(\xi,X)}|X_0=x\right]=\alpha(\gamma,x)L^{n+1}(\gamma)+r_{n+1}(\gamma,x)
    \end{equation}
    for  suitable non-negative functions $\alpha$, $L$ and $(r_n)_n$.
We suppose here that,
    \begin{enumerate}
    \item[(a)] For all $x\in\mathbb X$,
 the functions $\alpha(\cdot,x)$, $L$ and $r_n(\cdot,x)$ can be extended to  analytic functions in $\{z=u+iy,|u|\leq \delta+\epsilon<\nu, y\in \RR\}$
    \item[(b)] L is positive and non-increasing on $\RR^{*}_{+}$. The  equation $
    \kappa_1L(z)=1,$ has a unique positive solution in $\mathbb C$,
    denoted by $\nu$.
     \item[(c)] The mapping $L$ is holomorphic at $\nu$ and $L'(\nu)<0$.
     \item[(d)] The series $\sum_{n>0} \kappa_1^n r_n(\gamma,x)$ converges uniformly in $\gamma$ in a neighborhood of $\nu$
    uniformly in $x$.
     \item[(e)] There exists $p'>1$ such that $\int_{-\infty}^{\infty}|\alpha(\pm \delta +\nu +iy,x)L(\pm \delta +\nu +iy)|^{p'}dy <\infty$
     and that $\int_{-\infty}^{\infty}|\sum_{n\geq 0}\kappa_1^{n}r_n(\pm \delta+\nu+iy,x)|^{p'}dy<\infty$.
     \item[(f)]  $\inf_{y\in \RR}|1-\kappa_1L(\pm \delta+iy+\nu)|>0,$ $\infty>\sup_{|y|>M}\sup_{u\in [-\delta,\delta]}|\alpha(u+iy+\nu,x)L(u+iy+\nu)|>0$, $\infty>\sup_{|y|>M}\sup_{u\in [-\delta, \delta]}\sum_{n\geq 0}\kappa_1^{n}|r_n(u+iy+\nu,x)|>0$
    and $\inf_{|y|>M}\inf_{u\in [-\delta, \delta]}|1-\kappa_1L(u+iy+\nu)|>0$, for some $M>0$.
    \end{enumerate}
\begin{rem} Under Hypotheses  (\ref{hypoHarris}) and (\ref{analytiq}), the multiplicative ergodic property stated in (\ref{GE})
    ensures (\ref{formule-mult-erg}), in fact, in this case
    $$
    g_n(\gamma)=(\kappa_1-1)\kappa_1^nL^n(\gamma)\EE(\alpha(\gamma,X_0))+ (\kappa_1-1)\kappa_1^n\EE(r_n(\gamma,X_0)),\,\,for\,\, any\,\,\gamma>0.
    $$
\end{rem}
\begin{alem}\label{lemgeoerg}
Hypothesis \ref{analytiq} is satisfied under Conditions (a)-$\cdots$-(f) with
    \begin{eqnarray*}
    &&\tilde C_0(x)=-\frac{(\kappa_1-1)}{\kappa_1^2\nu L'(\nu)}\alpha(\nu,x)\\
    && \Psi_0(x) \leq  {\kappa_1}\left({\inf_{y\in \RR}\left|1-\kappa_1L(-\delta+iy+\nu,x)\right|}\right)^{-1} \times\\
   &&\left(\int_{-\infty}^{\infty}|\alpha(-\delta +\nu +iy,x)L(-\delta +\nu +iy)|^{p'}dy\right)^{1/p'}\left(\int_{-\infty}^{+\infty} ((\nu-\delta)^2+y^2)^{-p/2}dy\right)^{1/p}\\
   &&+\kappa_1 \left(\int_{-\infty}^{\infty}|\sum_{n\geq 0}\kappa_1^{n}r_{n+1}(-\delta+\nu+iy,x)|^{p'}dy\right)^{1/p'}\left(\int_{-\infty}^{+\infty} ((\nu-\delta)^2+y^2)^{-p/2}dy\right)^{1/p}
    \end{eqnarray*}
    for any $x\in\X$.
    \end{alem}

\begin{proof}[Proof of Lemma \ref{lemgeoerg}]
We have, in this case,
\begin{eqnarray*}
   &&{\tilde f}_{x,0}(\gamma)=\frac{\kappa_1-1}{\gamma+\nu}\sum_{n\geq 0}\kappa_1^{n} \EE[e^{-(\gamma+\nu) S_{1,n+1}(\xi,X)}|X_0=x] \\
   &&= \frac{\kappa_1-1}{\gamma+\nu} \frac{\alpha(\gamma+\nu,x)L(\gamma+\nu)}{1-\kappa_1L(\gamma+\nu)}+ \frac{\kappa_1-1}{\gamma+\nu}\sum_{n\geq 0}\kappa_1^{n}r_{n+1}(\gamma+\nu,x),
    \end{eqnarray*}
    which can be extended thanks to Conditions (a) and (d) to an analytic function in $\{z=u+iy,|u|\leq \delta+\epsilon<\nu, y\in \RR\}\setminus\{0\}$.
    Conditions (a), (b), (c) and (d) allow to deduce that
    $$
    \lim_{z\rightarrow 0}z{\tilde f}_{x,0}(z)=-\frac{(\kappa_1-1)}{\kappa_1^2\nu L'(\nu)}\alpha(\nu,x)=:\tilde C_0(x).
    $$
   We have arguing as for the proof of Lemma \ref{lemiid}, for any $|y|>M$ and any $u\in [-\delta, \delta]$,
   \begin{eqnarray*}
   && |{\tilde f}_{x,0}(u+iy)|\leq \frac{\kappa_1}{|y|}\frac{\sup_{|y|>M}\sup_{u\in [-\delta,\delta]}|\alpha(u+iy+\nu,x)L(u+iy+\nu)|}{\inf_{|y|>M}\inf_{u\in [-\delta, \delta]}|1-\kappa_1L(u+iy+\nu)|}\\
   && +  \frac{\kappa_1}{|y|}\sup_{|y|>M}\sup_{u\in [-\delta, \delta]}\sum_{n\geq 0}\kappa_1^{n}|r_{n+1}(u+iy+\nu,x)|,
   \end{eqnarray*}
   which proves, thanks to (f), that $\lim_{y\rightarrow \pm \infty}\sup_{u\in [-\delta,\delta]}|{\tilde f}_{x,0}(u+iy)|=0$. Now,
   \begin{eqnarray*}
   && \int |{\tilde f}_{x,0}(\pm\delta+iy)|dy \leq {\kappa_1}\left(\inf_{y\in \RR}|1-\kappa_1L(\pm\delta+iy+\nu)|\right)^{-1}\times\\
   &&\left(\int_{-\infty}^{\infty}|\alpha(\pm\delta +\nu +iy,x)L(\pm\delta +\nu +iy)|^{p'}dy\right)^{1/p'}\left(\int_{-\infty}^{+\infty} ((\nu\pm\delta)^2+y^2)^{-p/2}dy\right)^{1/p}\\
   &&+\kappa_1 \left(\int_{-\infty}^{\infty}|\sum_{n\geq 0}\kappa_1^{n}r_{n+1}(\pm \delta+\nu+iy,x)|^{p'}dy\right)^{1/p'}\left(\int_{-\infty}^{+\infty} ((\nu\pm\delta)^2+y^2)^{-p/2}dy\right)^{1/p},
   \end{eqnarray*}
which is finite thanks to Conditions {({e})} and {({f})}.
    Hypothesis (\ref{analytiq}) is then satisfied.
\end{proof}
 \subsubsection{Bifurcating Markov chains} Hypothesis \ref{hypoHarris} supposes also that the two sequences $(X_n^{(k)})_{n\geq k+1}$ and $(X_n)_{n\geq k+1}$ are independent given $X_k$ for any $k\in \NN$. This condition can be satisfied by bifurcating Markov chains as defined in Section 3 of \cite{LouhichiYcart15}.
\section{Tools for Markov models}
\label{results}
In this section, we discuss assumptions of Theorems
\ref{cor-th1} and \ref{cor-th2}.
We assume that $X=(X_n)_n$ is a Markov chain on $(\X,\cX)$
with Markov kernel $P(x,dy)$, invariant probability $\pi$,
and initial probability $\mu$ (i.e. $\mu$ is the distribution of $X_0$).
\subsection{Notations, definitions and hypotheses
on Banach spaces}\label{nota}
Recall that the Laplace-type kernels $(P_\gamma)_\gamma$ are given by
\begin{equation}\label{def-P-gamma-funct}
\forall \gamma\in[0,+\infty),\ P_\gamma f=P( \kappa e^{-\gamma\xi}f)\ \ \mbox{and}\ \
       P_{\infty}f=P(\kappa\mathbf 1_{\{\xi=0\}}f).
\end{equation}
For any normed complex vector spaces $(\cB_0,\|\cdot\|_{\cB_0})$ and $(\cB_1,\|\cdot\|_{\cB_1})$, we endow the space $\mathcal L(\cB_0,\cB_1)$ with the operator norm $\|\cdot\|_{\cB_0,\cB_1}$ given by
$$\forall Q\in\mathcal L(\cB_0,\cB_1),\ \  \|Q\|_{{\cB_0},{\cB_1}}=\sup_{f\in\cB_0,\ \|f\|_{{\cB_0}}=1}\|Qf\|_{{\cB_1}}.$$
We simply write
$(\mathcal L(\cB),\|\cdot\|_{\cB})$ for $(\mathcal L(\cB,\cB),\|\cdot\|_{\cB,\cB})$.
For any $Q\in\cL(\cB)$, we denote by $Q^*$ its adjoint operator.
We write $\sigma(Q)=\sigma(Q_{|\cB})$ for the spectrum of $Q$:
$$\sigma(Q) :=\{\lambda\in\mathbb C\ :\ (Q-\lambda\, I)\mbox{ is non invertible}\},$$
where $I$ denotes the identity operator on $\cB$. Recall that $Q$ and $Q^*$ have the same norm in $\cL(\cB)$ and $\cL(\cB^*)$ respectively, as well as the same spectrum.
The spectral radius of $Q_{|\cB}$ (resp.~its essential spectral radius) is denoted by $r(Q)=r(Q_{|\cB})$ (resp.~$r_{ess}(Q)=r_{ess}(Q_{|\cB})$). Recall that
$$r_{ess}(Q):=\sup\{|\lambda|\ :\ \lambda\in\mathbb C\ \mbox{and} \ (Q-\lambda\,  I)\mbox{ is non Fredholm}\}.$$

To check the assumptions of Theorems~\ref{generalspectraltheorem1} and \ref{generalspectraltheorem2}, we work with Banach spaces contained in $\mathbb L^1(\pi)$ (as stated in Hypothesis~\ref{hypcompl}) and
we use some assumptions involving the notion of positivity and non-negativity on such a space (or on its dual space), as defined below.
\begin{adefi}
Let $\mathcal B$ be a Banach space composed of functions $f:\X\rightarrow \C$ (or of classes of such functions modulo $\pi$).
If $f\in\mathcal B$ is a class of functions, we say that it is {\bf non-negative}
(resp. {\bf positive}) if one of its representant is so.
We say that it is {\bf non-null} if the null function is not one of its representant.
An element $\psi\in\mathcal B^*$ is said to be
{\bf non-negative} if for every non-negative $f\in\mathcal B$,
we have $\psi(f)\ge 0$.
An element $\psi\in\mathcal B^*$ is said to be
{\bf positive} if for every non-negative non-null $f\in\mathcal B$,
we have $\psi(f)> 0$.
\end{adefi}
\begin{acond}\label{(A)}
For every
$\phi\in\cB$, $\phi\geq 0$, $\phi\neq 0$, there exists $\psi\in\cB^*$, $\psi\geq 0$, such that $\psi(\phi) > 0$. For every $\psi\in\cB^*$, $\psi\geq 0$, $\psi\neq 0$, there exists $\phi\in\cB$, $\phi\geq 0$ such that $\psi(\phi) > 0$.
\end{acond}
%
%
\begin{acond}\label{(B)}
Let $J\subset [0,+\infty]$ such that: $\forall \gamma\in J,\ P_\gamma\in\cL(\cB)$. For every $\gamma\in J$ such that $r(\gamma)>0$, the following properties hold: if $\phi\in\cB$ is non-null and non-negative, then $P_\gamma \phi > 0$ (modulo $\pi$) and every non-null non-negative $\psi\in\cB^*\cap\ker(P_\gamma^*-r(\gamma)I)$ is positive.
\end{acond}
Note that Hypothesis~\ref{(A)} is quite general. Let us recall the definition of a Banach lattice.
\begin{adefi} \label{del-BL}
A complex Banach space $(\cB,\|\cdot\|_{\cB})$ of functions $f:\X\rightarrow\mathbb C$
(or of classes of such functions modulo $\pi$) is said to be a {\bf complex Banach lattice} if it is stable by $|\cdot|$, by real part and if
$$\forall f,g\in\cB,\quad f(\X)\cup g(\X)\subset\R\quad
      \Rightarrow\quad \min(f,g),\, \max(f,g)\in\cB , $$
$$\forall f,g\in\cB,\quad |f|\le|g|\quad\Rightarrow\quad \|\, |f|\, \|_\cB = \|f\|_\cB\le \|g\|_\cB = \|\, |g|\, \|_\cB.$$
\end{adefi}

Such a space satisfies Hypothesis~\ref{(A)}.
Classical instances of Banach lattices of functions are the spaces $(\L^p(\pi),\|\cdot\|_p)$ and $(\cB_V,\|\cdot\|_{V})$ (see (\ref{def-La}) and (\ref{def-BV})), as well as the space $(\mathcal L^\infty(\mathbb X),\|\cdot\|_{\infty})$ composed of all the bounded measurable $\C$-valued functions on $\X$, and equipped with its usual norm $\|f\|_{\infty}:=\sup_{x\in\X}|f(x)|.$

\subsection{About Hypothesis~\ref{hypcompl}}
%
\begin{apro} \label{pro-reduc-BL}
Let $J$ be a subinterval of $[0,+\infty]$ and let $\mathcal B$ be a Banach lattice of functions on $\X$ (or of classes of such functions modulo $\pi$). We assume that Hypothesis~\ref{(B)} holds, that $\cB\subset \mathbb L^1(\pi)$, and that for every $\gamma\in J$
\begin{enumerate}[(i)]
\item $P_\gamma\in\cL(\cB)$, $r(\gamma):=r(P_{\gamma|\cB})>0$, and $P_\gamma$ is quasi-compact on $\cB$,
\vspace*{1mm}
\item for every $f,g\in\cB$ with $f>0$, $P_\gamma f=r(\gamma)f$ and $P_\gamma g=r(\gamma)g$, we have $g\in \C\cdot f$,
\end{enumerate}
Assume moreover that the Markov kernel $P$ satisfies the following condition: 1 is the only complex number $\lambda$ of modulus 1 such that $P(h/|h|)=\lambda h/|h|$  in $\L^1(\pi)$ for some $h\in\cB$, $|h|>0$ (modulo $\pi$).
Then Hypothesis \ref{hypcompl} is fulfilled with $J$ and $\cB$.
\end{apro}
Proposition~\ref{pro-reduc-BL} is proved in Section~\ref{proof-reduc-BL}.

%
%
%
%
%
\begin{apro}\label{pro-phi-pi}
Assume that Hypothesis \ref{hypcompl} holds for some $\cB$ and $J$. Let $\gamma\in J$. Then
\begin{equation} \label{Pi-gamma-lim}
\Pi_\gamma = \lim_n r(\gamma)^{-n} P_\gamma^n\ \ \text{in }\,  \cL(\cB),
\end{equation}
and there exist some nonzero elements $\widehat\pi_\gamma\in\cB^*\cap\ker(P_\gamma^*-r(\gamma)I)$ and $\widehat\phi_\gamma\in\cB\cap\ker(P_\gamma-r(\gamma)I)$ such that
$\widehat\pi_\gamma(\widehat\phi_\gamma) = 1$ and
\begin{equation} \label{phi-gamma-pi-gamma}
\forall f\in\cB,\quad \Pi_\gamma f = \widehat\pi_\gamma(f)\, \widehat\phi_\gamma\qquad \text{and} \qquad \forall f^*\in\cB^*,\quad \Pi_\gamma^* f^* = f^*(\widehat\phi_\gamma)\, \widehat\pi_\gamma.
\end{equation}
If $\cB$ satisfies Hypothesis~\ref{(A)}, then $\widehat\phi_\gamma$ and $\widehat\pi_\gamma$ are non-negative
in $\cB$ and $\cB^*$ respectively.
Under the additional Hypothesis~\ref{(B)}, every non-null non-negative $\phi\in\cB\cap\ker(P_\gamma-r(\gamma)I)$ is positive $\pi-$a.s.~and, for every non-null and non-negative $f\in\cB$, we have $\Pi_\gamma f >0$ $\pi-$a.s..
\end{apro}
\begin{proof}
Properties~(\ref{Pi-gamma-lim}) and the existence of $\widehat\phi_\gamma$ and $\widehat\pi_\gamma$ in (\ref{phi-gamma-pi-gamma}) follow from Hypothesis~\ref{hypcompl}. Now assume that Hypothesis~\ref{(A)} holds. Then (\ref{Pi-gamma-lim}) and the first assertion in Hypothesis~\ref{(A)}, applied with $\phi = \widehat\phi_\gamma$ and the associated $\psi_\gamma\in\cB^*$, $\psi_\gamma\geq 0$, imply that, for every $g\in\cB$, $g\geq 0$,
 we have $0 \leq  \lim_n r(\gamma)^{-n} \psi_\gamma(P_\gamma^n g) = \widehat\pi_\gamma(g)\, \psi_\gamma(\widehat\phi_\gamma)$,
hence $\widehat\pi_\gamma\geq 0$ since $\psi_\gamma(\widehat\phi_\gamma) >0$. Next the second assertion in Hypothesis~\ref{(A)}, applied with $\psi=\widehat\pi_\gamma$ and the associated $\phi_\gamma\in\cB$, $\phi_\gamma\geq 0$, gives
$0 \leq \lim_n r(\gamma)^{-n} P_\gamma^n \phi_\gamma = \widehat\pi_\gamma(\phi_\gamma)\widehat\phi_\gamma$,
hence $\widehat\phi_\gamma\geq 0$ since $\widehat\pi_\gamma(\phi_\gamma)>0$.
Finally, under Hypotheses~\ref{(A)} and \ref{(B)}, if $\gamma\in J$ and if  $f\in\cB$, $f\neq 0$, $f\geq 0$, then $\widehat\pi_\gamma(f)>0$. Thus $\Pi_\gamma f = \widehat\pi_\gamma(f)\, \widehat\phi_\gamma$ is positive modulo $\pi$ since so is $\widehat\phi_\gamma>0$ from Hypothesis~\ref{(B)}.
\end{proof}
%
%
\begin{acor} \label{cor-plus-esp}
Assume that, for some subinterval $J\subset[0,+\infty)$, Hypothesis \ref{hypcompl} holds on two Banach spaces $\cB_1$ and $\cB_2$, both containing $1_\X$ and satisfying Hypotheses~\ref{(A)} and \ref{(B)}. Then
$$\forall \gamma\in J,\quad r(P_{\gamma| \cB_1}) = r(P_{\gamma| \cB_2})=\lim_{n\rightarrow +\infty}(\pi(P_\gamma^n\mathbf 1_{\mathbb X}))^{1/n}.
$$
If moreover $\cB_1\hookrightarrow \cB_2$ and if, for $i=1,2$, $\Pi_{\gamma,i}$ denotes the rank-one eigen-projector associated with $P_{\gamma|\cB_i}$ in (\ref{sup-vit-ponctuel}), then the restriction of $\Pi_{\gamma,2}$ to $\cB_1$ equals to $\Pi_{\gamma,1}$.
\end{acor}
\begin{proof}
For $i=1,2$, Proposition~\ref{pro-phi-pi} applied to $P_{\gamma|\cB_i}$ (with the notations $\widehat\phi_{\gamma,i}$ and $\widehat\pi_{\gamma,i}$) gives
$$\pi(P_\gamma^n\mathbf 1_\X) = (r(P_{\gamma | \cB_i}))^n \widehat\pi_{\gamma,i}(\mathbf 1_\X)\pi(\widehat\phi_{\gamma,i}) + \text{o}\big(r(P_{\gamma|\cB_i})^n\big)$$
with $\pi(\widehat\phi_{\gamma,i})\, \widehat\pi_{\gamma,i}(\mathbf 1_\X)>0$ from Hypothesis~\ref{(B)}. 
Hence the first assertion holds. Now let $f\in\cB_1$. Then $\Pi_{\gamma,1} f = \lim_n r(P_{\gamma|\cB_2})^{-n}(P_{\gamma|\cB_1})^{n} f$ in $\cB_1$ from Proposition~\ref{pro-phi-pi} applied to $P_{\gamma|\cB_1}$ and from the previous fact. It follows from $\cB_1\hookrightarrow \cB_2$ that this convergence holds in $\cB_2$ too. Now Proposition~\ref{pro-phi-pi} applied to $P_{\gamma|\cB_2}$ gives $\Pi_{\gamma,1} f = \Pi_{\gamma,2} f$.
\end{proof}
\subsection{About Condition \eqref{def-B}}
\begin{apro} \label{pro-B-gamma}
Assume that
$(P_\gamma)_\gamma$ and $\mu$ satisfy the assumptions of Theorem~\ref{cor-th1},  excepted (\ref{def-B}). Then the real number $B(\gamma)$ given in (\ref{def-B}) is well-defined for every $\gamma\in J_0$, and
\begin{equation} \label{fle-B-gamma}
\forall \gamma\in J_0,\quad B(\gamma)  = \widehat\pi_\gamma(Ph_{\kappa,\gamma})\, \mu\left(\kappa\, e^{-\gamma \xi}\widehat\phi_\gamma\right),
\end{equation}
where $\widehat\phi_\gamma$ and $\widehat\pi_\gamma$ are given in (\ref{phi-gamma-pi-gamma}).
Assume moreover that the space $\cB$ involved in the assumptions of  Theorem \ref{generalspectraltheorem1} satisfies Hypotheses~\ref{(A)} and \ref{(B)} on $J_0$
and that one of the following assumptions holds true
\begin{enumerate}[(i)]
	\item $\mu$ is absolutely continuous with respect to $\pi$,
	\item the first part in Hypothesis~\ref{(B)} is reinforced as follows: for every $\gamma\in J$, if $\phi\in\cB$ is non-null and non-negative,
then $P_\gamma \phi > 0$ everywhere on $\X$.
\end{enumerate}
Then (\ref{def-B}) holds.
\end{apro}
\begin{proof}
Let $\gamma\in J_0$ (thus $r(\gamma)>0$). First, by assumption $Ph_{\kappa,\gamma}\in\cB_0$. Thus $\Pi_\gamma(Ph_{\kappa,\gamma})\in
\cB_3$
since
$\Pi_\gamma\in \cL(\cB_0,
\cB_3)$ (see Theorem \ref{generalspectraltheorem1}), and $B(\gamma)$ is then  well-defined from the assumptions on $\mu$ in Theorem~\ref{cor-th1}. Formula~(\ref{fle-B-gamma}) follows from (\ref{phi-gamma-pi-gamma}).
Now, under the additional assumptions in $(i)$, we know from Proposition~\ref{pro-phi-pi} that $\widehat\phi_\gamma\geq 0$ and $\widehat\pi_\gamma\geq 0$. Moreover
$Ph_{\kappa,\gamma}\geq 0$ and $Ph_{\kappa,\gamma}\neq 0$ in $\L^1(\pi)$
since $\pi(Ph_{\kappa,\gamma}) = \pi(h_{\kappa,\gamma}) > 0$.
Since $\cB_0\hookrightarrow\cB_3\hookrightarrow \L^1(\pi)$ by hypothesis, it follows that $Ph_{\kappa,\gamma}\neq 0$ in $\cB_1$ and in $\cB_3$. Thus $\widehat\pi_\gamma(Ph_{\kappa,\gamma}) > 0$ from Hypothesis~\ref{(B)}. Also note that $\widehat\phi_\gamma > 0$ (modulo $\pi$) from Proposition~\ref{pro-phi-pi}.
Then $(i)$ follows from the previous remarks and from (\ref{fle-B-gamma}).
Assertion~$(ii)$ is obvious since, in this case, $\widehat\phi_\gamma>0$ everywhere.
\end{proof}
\subsection{About the monotonicity of the spectral radius}
\begin{apro}\label{LEMME0}
If $(\cB,\|\cdot\|_{\cB})$ is a complex Banach lattice of functions $f:\X\rightarrow\mathbb C$ (or of classes of functions modulo $\pi$), and if $P_\gamma\in\cL(\cB)$ for every $\gamma\in[0,+\infty)$, then the map $\gamma\mapsto r(\gamma)$ is non-increasing on $[0,+\infty)$.
\end{apro}
\begin{proof}
For any $0\le\gamma<\gamma'\le\infty$ and for any $f,g\in\cB$ such that
$|f|\le|g|$, we have $e^{-\gamma'\xi} |f|\le e^{-\gamma\xi} |g|$ and so
$P_{\gamma'}|f|\le P_{\gamma}|g|$, which implies by induction that $P_{\gamma'}^n|f|\le P_{\gamma}^n|f|$ for every integer $n\ge 1$. We conclude that $\|P_{\gamma'}^n\|_{\cB}\le \|P_{\gamma}^n\|_{\cB}$
since $(\cB,\|\cdot\|_{\cB})$ is a Banach lattice.
This implies that $r(\gamma')\le r(\gamma)$ and so the desired statement.
\end{proof}
\begin{apro}\label{decroissanceAB}
Assume that Hypothesis \ref{hypcompl} is fulfilled for some $J$ and for some $\cB$ satisfying Hypotheses~\ref{(A)} and \ref{(B)}.
Then $\gamma\rightarrow r(\gamma)$
is non-increasing on $J$.
\end{apro}
\begin{proof}
We use the notations of Proposition~\ref{pro-phi-pi}. Let $\gamma_1,\gamma_2\in J$ such that $\gamma_2<\gamma_1$. Then $\widehat\pi_{\gamma_1}(P_{\gamma_1}^n\widehat\phi_{\gamma_1})\le
\widehat\pi_{\gamma_1}(P_{\gamma_2}^n\widehat\phi_{\gamma_1})$ since $\widehat\pi_{\gamma_1}$ is non-negative. Moreover
$\widehat\pi_{\gamma_1}(P_{\gamma_1}^n\widehat\phi_{\gamma_1})=(r(\gamma_1))^n$
and
$$
\widehat\pi_{\gamma_1}(
P_{\gamma_2}
^n\widehat\phi_{\gamma_1})=(r(\gamma_2))^n
\widehat\pi_{\gamma_2}(\widehat\phi_{\gamma_1})\widehat\pi_{\gamma_1}(\widehat\phi_{\gamma_2})+o\left((r(
\gamma_2
)
)^n\right).
$$
Since
$\widehat\pi_{\gamma_1}$ and $\widehat\pi_{\gamma_2}$ are positive
and since $\phi_{\gamma_1}$ and $\phi_{\gamma_2}$ are non-null
non-negative,
Hypothesis~\ref{(B)} gives
$\widehat\pi_{\gamma_2}(\widehat \phi_{\gamma_1})\widehat\pi_{\gamma_1}(\widehat\phi_{\gamma_2})>0$.
Thus $(r(\gamma_1))^n\le O\left((r(\gamma_2))^n\right)$, so $r(\gamma_1)\le r(\gamma_2)$.
\end{proof}
\begin{arem}
The non-increasingness of $r$ on $J$ implies that the set $J_0$ in (\ref{J0}) is an interval. Let us also indicate that it can happen that $r(\cdot)$ is constant on $[0,+\infty)$ (see Appendix \ref{counterexample}).
\end{arem}

The next result is relevant to check the condition $r'(\nu)\neq 0$ in Theorem~\ref{cor-th2}.
\begin{apro} \label{pro-deri-stric-pos}
Assume that the assumptions of Theorem \ref{generalspectraltheorem2} hold.

\begin{enumerate}[(i)]
	\item Let $\gamma\in J_0$ be such that $\pi\big(\Pi_\gamma(\xi\, \Pi_\gamma\mathbf 1_\X)\big) > 0$. Then $r'(\gamma) < 0$.
	\vspace*{1mm}
  \item Assume moreover that the space $\cB_2$ involved in the assumptions of Theorem \ref{generalspectraltheorem2} satisfy Hypotheses~\ref{(A)} and \ref{(B)} on $J_0$,
and that $\pi(\{\xi=0\})<1$. Then $r'(\cdot)< 0$ on $J_0$, thus $\gamma\mapsto r(\gamma)$ is strictly decreasing on $J_0$.
\end{enumerate}
\end{apro}
\begin{proof}
Assertion~$(i)$ follows from Proposition~\ref{ap-pro-ap-deriv}.
Let us derive $(ii)$ from~$(i)$. Let $\gamma\in J_0$.
First note that $\xi\Pi_\gamma\mathbf 1_\X\in \cB_2$. Indeed it follows from the assumptions of Theorem \ref{generalspectraltheorem2} that $\Pi_\gamma\mathbf 1_\X\in\cB_1$ (use also Corollary~\ref{cor-plus-esp}) and that the map $f\mapsto \xi f$ is in $\cL(\cB_1,\cB_2)$. Hence $\xi\Pi_\gamma\mathbf 1_\X\in \cB_2$. Moreover, under the assumptions in $(ii)$, we know from the last assertion of Proposition~\ref{pro-phi-pi} (applied on $\cB_2$) that $\Pi_\gamma\mathbf 1_\X > 0$ modulo $\pi$,
thus $\xi\, \Pi_\gamma\mathbf 1_\X \neq 0$ in $\L^1(\pi)$
(since
$\pi(\{\xi=0\})<1$),
and so $\xi\, \Pi_\gamma\mathbf 1_\X\neq 0$ in $\cB_2$ from $\cB_2\hookrightarrow \L^1(\pi)$. Then it follows again from the last assertion of Proposition~\ref{pro-phi-pi} (applied on $\cB_2$) that $\Pi_{\gamma}(\xi\, \Pi_{\gamma} 1_\X) > 0$ modulo $\pi$, thus $\pi\big(\Pi_\gamma(\xi\, \Pi_\gamma\mathbf 1_\X)\big) > 0$. Hence $r'(\gamma)< 0$ from $(i)$.
\end{proof}
%

%
%
%
A consequence of the monotonicity of $\gamma\mapsto r(\gamma)$
the following characterisation of $\nu<\infty$.
\begin{apro} \label{cor-th1v1}
Assume that the assumptions of Theorem~\ref{cor-th1} hold
(with $J$ and $\delta_0$ given in Hypothesis \ref{hypKL} or \ref{hypKL}*); in particular for every $\gamma\in J$ the Laplace kernel $P_\gamma$ is assumed to continuously act on the Banach space $\cB$ chosen in the assumptions of Theorem~\ref{generalspectraltheorem1}.\vspace*{-2mm}
\begin{enumerate}[(i)]
	\item
If $\delta_0<1$, then
for every
$\gamma\in J$ we have: $\ G(\gamma)<\infty \Leftrightarrow r (\gamma)<1$.
\vspace*{2mm}
  \item If $r$ is non-increasing, if $J=(a,+\infty]$ for some $a\geq 0$ and if $\delta_0 <1$,
then $\nu<\infty \Leftrightarrow r(\infty)= r(P_{\infty|\cB})<1$.
\end{enumerate}

\end{apro}
\begin{proof}
First, if $\gamma\in J_0$, then $G(\gamma)<\infty \Leftrightarrow r(\gamma)<1$ due to Theorem~\ref{cor-th1}
and to Remark \ref{multergodP1P2}. Second, if $\gamma \in J\setminus J_0$, then $r(\gamma)\le \delta_0<1$, so that, for some fixed $\delta\in(\delta_0,1)$, we can deduce from the definition of $g_n(\gamma)$ and $r(\gamma)$, and from assumptions of Theorem~\ref{cor-th1}, that there exists $\tilde C_\delta>0$ such that
$G(\gamma) \le  \sum_{n=0}^{+\infty} \tilde C_\delta\delta^n<\infty$.
Hence $(i)$ is fulfilled.
Now, under the assumptions of $(ii)$,
it follows from Proposition~\ref{LEMME0} and from $(i)$ that: $\nu<\infty \Leftrightarrow \limsup_{\gamma\rightarrow +\infty}r(\gamma) < 1$. Moreover, due to
Theorem~\ref{cor-th1}, we know that $\limsup_{\gamma\rightarrow +\infty}r(\gamma)\le \max(\delta_0,r(\infty))$, and even that $\lim_{\gamma\rightarrow +\infty} r(\gamma)=r(\infty)$ if $r(\infty)>\delta_0$. Considering the cases $r(\infty)\leq \delta_0$ and $r(\infty)>\delta_0$ then gives the desired equivalence in $(ii)$.
\end{proof}
\subsection{Positivity of the spectral radius}
Recall that we have set $J_0:=\{\gamma\in J : r(\gamma)>\delta_0\}$
under the assumptions of Theorem~\ref{generalspectraltheorem1}. Another consequence of the monotonicity of $\gamma\mapsto r(\gamma)$
is the following lemma.
\begin{alem}\label{rnonnul}
Let $\gamma_1,\gamma_2,\gamma_3$ be such that $0\leq \gamma_1 < \gamma_2 < \gamma_3$. Assume that the assumptions of Theorem~\ref{generalspectraltheorem1} hold with $J=(\gamma_1,\gamma_2)$
and that $r$ is non-increasing on $J$.
Moreover suppose that, for every $\gamma\in(\gamma_1,\gamma_3)$, $P_\gamma$ continuously acts on $\cB$ and that the map $f\mapsto \pi(\kappa e^{-\gamma\xi}f)$ is in $\cB^*$, where $\cB$ is the space given in Theorem~\ref{generalspectraltheorem1}.
Finally suppose that
\begin{equation} \label{C-fini-0}
\Delta_0 := \limsup_{n} \bigg(\pi\big(\kappa\, P_{0}^n\mathbf 1_{\mathbb X}\big)\bigg)^{\frac{1}{n}} < \infty.
\end{equation}
If $ J_0\neq\{0\}$, then we have $r(\gamma)>0$
for every $\gamma\in (\gamma_1,\gamma_3)$.
\end{alem}
\begin{proof}
Let $\gamma_0\in J_0$, $\gamma_0\neq0$.
Then
$ r (\gamma)\ge r (\gamma_0)> 0$ for every
$\gamma\in(\gamma_1,\gamma_0]$.
Next let $\gamma\in(\gamma_0,\gamma_3)$ and set $p:=\gamma/\gamma_0>1$.
Since $\widehat\pi_{\gamma_0}$ is positive and $\widehat\phi_{\gamma_0}$ is non-negative and non-null (modulo $\pi$), we have
$$0<r (\gamma_0)=r\left(\frac{\gamma}p\right)
   = \lim_{n\rightarrow +\infty} \left(\pi\big(\kappa\, e^{-\frac{\gamma}{p}\xi}\, P_{\frac\gamma p}^n\mathbf 1_{\mathbb X}\big)\right)^{\frac 1n} =
	\lim_{n\rightarrow+\infty}\bigg(\E_\pi\bigg[\bigg(\prod_{j=0}^{n}\kappa(X_j)\bigg)\, e^{-\frac\gamma p S_n}\bigg]\bigg)^{\frac 1n}$$
due to \eqref{sup-vit} since $\mathbf 1_{\mathbb X}\in \mathcal B$ and $f\mapsto \pi(\kappa e^{-\frac{\gamma}{p}\xi}f)$ is in $\cB^*$, and due to (\ref{formuleFourier})
(in which we replace $h_{\kappa,\gamma}$ by $\kappa e^{-\gamma\xi}$).
Let $q=p/(p-1)$. Writing $\kappa(X_j) = \kappa(X_j)^{1/q}\kappa(X_j)^{1/p}$, it follows from the H\"older inequality that
\begin{eqnarray*}
r\left(\gamma_0\right) &\leq&  \limsup_{n\rightarrow+\infty}\bigg(\E_\pi\bigg[\prod_{j=0}^{n}\kappa(X_j)\bigg]\bigg)^{\frac{1}{nq}}\times\ \limsup_{n\rightarrow+\infty}\bigg(\E_\pi\bigg[\bigg(\prod_{j=0}^{n}\kappa(X_j)\bigg)\, e^{-\gamma S_n}\bigg]\bigg)^{\frac{1}{np}} \\
&\le& \limsup_{n\rightarrow +\infty} \bigg(\pi\big(\kappa\, P_{0}^n\mathbf 1_{\mathbb X}\big)\bigg)^{\frac{1}{nq}} \times \limsup_{n\rightarrow +\infty} \left(\pi\big(\kappa\, e^{-\gamma\xi}\, P_{\gamma}^n\mathbf 1_{\mathbb X}\big)\right)^{\frac{1}{np}}.
\end{eqnarray*}
The above first limit superior equals to $\Delta _0^{1/q}$ by hypothesis, and the second limit superior is less than $(r(\gamma))^{\frac 1p}$ from the definition of $r(\gamma)$ and from the fact that $\mathbf 1_{\mathbb X}\in \mathcal B$ and $f\mapsto \pi(\kappa e^{-\gamma\xi}f)$ is in $\cB^*$. Thus $0 < r(\gamma_0) \leq  \Delta _0^{1/q}\, (r(\gamma))^{\frac 1p}$.
\end{proof}
\begin{arem} \label{rem-C0}
Condition~(\ref{C-fini-0}) holds if $0\in J$ (in particular $P_0\in\cL(\cB)$) and if the map $f\mapsto \pi(\kappa f)$ is in $\cB^*$ since $\Delta_0 \leq r(0)$ from the definition of the spectral radius $r(0)$ of $P_0$. Moreover note that  (\ref{C-fini-0}) holds too if $\kappa$ is bounded by some constant $d>0$ since $P_0 \leq d\, P$.
\end{arem}
%
%

\section{Knudsen gas: Proof of Theorem \ref{pro-Knudsen} } \label{Knudsengas}


In this section, we apply our general results for the Knudsen gas.
Here the Laplace-type kernels $P_\gamma$ act
on the usual Lebesgue space  $(\mathbb L^a(\pi),\|\cdot\|_a)$ for some suitable $a\in[1,+\infty)$, where
\begin{equation} \label{def-La}
\|f\|_a:=\left(\int_{\X}|f(x)|^a\, d\pi(x)\right)^{\frac 1 a}.
\end{equation}

Theorem \ref{pro-Knudsen} directly follows from the next more precise theorem.
\begin{atheo}[Knudsen gas] \label{pro-Knudsenv1}
Assume that $X=(X_n)_n$ is a Knudsen gas as described in Example~\ref{exemp-Knud}, and that its initial distribution $\mu$ on $\X$ is
absolutely continuous with respect to $\pi$, with density in $\L^{p}(\pi)$.
Assume moreover that $\kappa\equiv 2$ and that the function $\xi:\X\rightarrow[0,+\infty)$ in (\ref{def-Sn}) is measurable. Then $(S_n)_n$ is multiplicatively ergodic on the interval $J_0=\{\gamma>0\, :\, r(\gamma) > 2(1-\alpha)\}$,
where $r(\gamma)$ denotes le spectral radius of $P_\gamma$ on $\L^b$ with $b:=\frac{p}{p-1}$. If moreover
$\alpha>1/2$ and if
\begin{equation} \label{cond-Zn-knud}
2\alpha \sum_{n\ge 0}(2(1-\alpha))^n
\P_\pi\bigg(\sum_{k=0}^{n}\xi(Z_k)=0\bigg) <1,
\end{equation}
where $(Z_n)_n$ is a Markov process with transition $U$, then $\nu$ defined in \eqref{nuDEF} is finite.
Finally, if  $\pi(\xi^\tau)<\infty$ for some $\tau>1$
and if $p>\frac{\tau}{\tau-1}$, then the constant $C_\nu$ in
\eqref{CDEF} is well defined and finite, and consequently the conclusion \eqref{CVMOYESP} of Corollary \ref{cor1} holds true.
\end{atheo}
Theorem~\ref{pro-Knudsenv1} straightforwardly extends to the case $\kappa(\cdot)\equiv m$, where $m\geq 2$ is any integer. To prove Theorem~\ref{pro-Knudsenv1}, we apply Theorems~\ref{cor-th1} and \ref{cor-th2}. First we prove the following.
\begin{alem}\label{prop1}
Let $1\le b<a$.
\begin{enumerate}[(i)]
\item For every $\gamma\ge 0$, $r_{ess}(P_{\gamma|\L^{a}(\pi)})\le 2(1-\alpha)$.
\item The function $\gamma\rightarrow P_\gamma $ is continuous from $(0,+\infty]$ to $\mathcal L(\L^{a}(\pi) ,\mathbb L^b(\pi))$.
\item For any $\gamma\in[0,+\infty]$ and any $f\in\L^a(\pi)$,
$\|P_\gamma f\|_a\le 2((1-\alpha)\|f\|_a + \alpha\|f\|_1)$.
\item For any $\gamma>0$, for any non-null non-negative $f\in\L^a(\pi)$ and every non-null non-negative $g\in\L^{a'}(\pi)$ with $a'= \frac{a}{a-1}$, we have
$\pi(gP_\gamma f)>0$ and $P_\gamma f >0$.
\item If $r(\gamma)>2(1-\alpha)$, for every $f,g\in\L^a(\pi)$ with $f>0$, $P_\gamma f=r(\gamma)f$ and $P_\gamma g=r(\gamma)g$, then we have $g\in \C\cdot f$.
\item 1 is the only complex number $\lambda$ of modulus 1 such that
$P(h/|h|)=\lambda h/|h|$  in $\L^1(\pi)$ for some $h\in\cB$, $|h|>0$ (modulo $\pi$).
\end{enumerate}
\end{alem}
\begin{proof} $\ $\\
{\it (i)} Observe that $P_\gamma=2(\alpha \pi(e^{-\gamma\xi}\cdot)+
(1-\alpha)U_\gamma)$ with $U_\gamma:=U(e^{-\gamma\xi}\cdot)$.
Since the sum of a Fredholm operator with a compact operator is
Fredholm, we directly obtain
$r_{ess}(P_\gamma)=2(1-\alpha)r_{ess}(U_\gamma)\le 2(1-\alpha)$. \\[0.12cm]
{\it (ii)} For every $0\le\gamma,\gamma'<\infty$ and every $f\in\cB$ such that $\|f\|_a=1$, we have
\begin{eqnarray*}
\|P_\gamma f-P_{\gamma'}f\|_b&=& 2\|P((e^{-\gamma\xi}-e^{-\gamma'\xi})f)\|_b\\
&\le& 2\|(e^{-\gamma\xi}-e^{-\gamma'\xi})f\|_b\le  2\|e^{-\gamma\xi}-e^{-\gamma'\xi}\|_c,
\end{eqnarray*}
where $c$ is such that $\frac 1a+\frac 1c=\frac 1b$.
Hence $\|P_\gamma-P_{\gamma'}\|_{\L^a(\pi),\L^b(\pi)}\le 2\|e^{-\gamma\xi}-e^{-\gamma'\xi}\|_c $,
which converges to 0 as $\gamma'$ goes to $\gamma$, by the dominated convergence theorem. In the same way, we prove that $\|P_\gamma-P_{\infty}\|_{\L^a(\pi),\L^b(\pi)}\le 2\|e^{-\gamma\xi}\|_c $ and hence the continuity of $\gamma\mapsto P_\gamma$
at infinity. \\[0.12cm]
$(iii)$ For every $\gamma\in[0,+\infty]$ and every $f\in\L^a(\pi)$,
$\|P_\gamma f\|_a\le 2\|Pf\|_a\le 2((1-\alpha)\|f\|_a + \alpha\|f\|_1)$
since $\|Uf\|_a\le\|f\|_a$.
This gives the Doeblin-Fortet inequality. \\[0.12cm]
$(iv)$ For any non-null non-negative $f\in\L^a(\pi)$, we have $P_\gamma f \geq 2\alpha\pi(e^{-\gamma\xi}f)1_{\X}>0$. The other assertion of $(iv)$ is then obvious. \\[0.12cm]
$(v)$ Let $f,g\in\L^a(\pi)$ such that $f>0$,
$P_\gamma f=r(\gamma)f$ and $P_\gamma g=r(\gamma)g$  in $\L^a(\pi)$.
Set
$\beta := \frac{\pi(e^{-\gamma\xi}g)}{\pi(e^{-\gamma\xi}f)}$ and $h:=g- \beta\, f$.
Then $\pi(e^{-\gamma\xi} h)=0$ and $P_\gamma h = r(\gamma)\, h$, which gives $r(\gamma)\, h = 2(1-\alpha)\, U(e^{-\gamma\xi} h)$, so that $r(\gamma)\, |h| \leq 2(1-\alpha)\, U(|h|)$. Since $\pi\, U = \pi$, we obtain: $r(\gamma)\, \pi(|h|) \leq 2 (1-\alpha)\, \pi(|h|)$. Finally we conclude that $\pi(|h|)=0$ because $r(\gamma) > 2(1-\alpha)$ and so $g=\beta f$ in $\L^a(\pi)$. \\[0.12cm]
$(vi)$
Let $k\in\L^1(\pi)$ and $\lambda\in \C$ be such that $|\lambda|=1$,
$|k|\equiv \mathbf 1_\X$ and $P(k)=\lambda k$. Then
$\lambda k=\alpha\pi(k)+(1-\alpha)U(k)$. Taking the modulus, we obtain
$1 \le \alpha|\pi(k)|+(1-\alpha)U(\mathbf 1_\X)\le 1$. By convexity we conclude that
$|\pi(k)|=1$ and that $k$ is constant modulo $\pi$, so that $\lambda=1$.
\end{proof}
\begin{proof}[Proof of the multiplicative ergodicity]
We apply Theorem~\ref{cor-th1}. Let $b:=\frac{p}{p-1}$ and $a>b$. From Assertion~$(i)$-(iii) of Lemma~\ref{prop1}, $P_\gamma$ satisfies Hypothesis \ref{hypKL} with $J=[0,+\infty)$, $\cB_0=\L^a(\pi)$ and $\cB_1=\L^{b}(\pi)$
(to obtain (\ref{D-F-direct}), iterate Inequality~$(iii)$ of Lemma~\ref{prop1} and use $\|\cdot\|_1 \leq \|\cdot\|_b$). Next, if $\gamma\in J_0$ (i.e.~$r(\gamma) > 2(1-\alpha)$), then $P_\gamma$ is quasi-compact from Assertion~$(i)$ of Lemma~\ref{prop1}.
Note that Assertion~$(iv)$ of Lemma~\ref{prop1} implies that Hypothesis~\ref{(B)} holds on the space $\L^a(\pi)$. Then $P_\gamma$ satisfies the hypotheses of Proposition~\ref{pro-reduc-BL} on $\cB_0=\L^a(\pi)$ from Assertions~$(v)$ and $(vi)$ of Lemma~\ref{prop1}. Thus Hypothesis~\ref{hypcompl} holds
on $J_0$ with $\cB_0=\L^a(\pi)$. Consequently $(S_n)_n$ is multiplicatively ergodic on $J_0$ with respect to $\mathbb P_\mu$, if we prove that $\mu_{\gamma} : f\mapsto \mu(e^{-\gamma\xi}f)$ is in $(\L^{b}(\pi))^*$, that the function $\gamma\mapsto \mu_{\gamma}$ is continuous from $J_0$ to $(\L^{b}(\pi))^*$, and finally that Condition~(\ref{def-B}) holds. Since $\mu$ is absolutely continuous with respect to $\pi$ with density $g_\mu$ in $\L^{p}(\pi)$, we have
$$\forall f\in\L^{b}(\pi),\quad \mu_{\gamma}(f) = \int_{\R^d} f(y)\, e^{-\gamma\xi(y)}\, g_\mu(y)\, d\pi(y)$$
thus $\mu_{\gamma}\in(\L^{b}(\pi))^*$ since $e^{-\gamma\xi}\, g_\mu \in \L^{p}(\pi)$. Moreover the norm in $(\L^{b}(\pi))^*$ of $(\mu_\gamma - \mu_{\gamma'})$  equals to $\|(e^{-\gamma\xi} - e^{-\gamma'\xi})\, g_\mu\|_p$, which converges to $0$ as $\gamma'\r\gamma$ from Lebesgue's theorem.
Finally Condition~(\ref{def-B}) holds from Assertion~$(i)$ of Proposition~\ref{pro-B-gamma}.
\end{proof}

In view of \eqref{nuDEF}, we study the spectral radius $r(\gamma)$ of $P_\gamma$. First observe that the  non-increasingness of $r(\cdot)$ follows from Proposition~\ref{LEMME0} since $\L^a(\pi)$ is a Banach lattice. Consequently the set $J_0 := \{\gamma>0\, :\, r(\gamma)>2(1-\alpha)\}$ is an interval with $\min J_0 = 0$ since $r(0)=2$.
Next set $h_\gamma:= e^{-\gamma\xi}$
for $\gamma\ge 0$ and $ h_\infty:=\mathbf{1}_{\{\xi=0\}}$.
Recall that $P_\gamma f = 2 [\alpha\pi(f\, h_\gamma)+
(1-\alpha)U(f\, h_\gamma)]$.
We set $\tilde U_\gamma(\cdot):=h_\gamma\, U(\cdot)$.
\begin{alem}\label{rayonspectralKnudsen}
Let $ \gamma\in[0,\infty]$ and $a\in(1,+\infty)$.
Let $\lambda$ be an eigenvalue of $(P_\gamma)_{|\mathbb L^a(\pi)}$
such that $\lambda>2(1-\alpha)\rho(\tilde U_\gamma)$. Then
\begin{equation}\label{eqlambda}
\lambda=2\alpha\sum_{n\ge 0}\frac{2^n(1-\alpha)^n}{\lambda^n}\pi(\tilde U_\gamma^n(h_\gamma)).
\end{equation}
In particular if $r(\gamma)>2(1-\alpha)$, then $\lambda=r(\gamma)$
satisfies \eqref{eqlambda}.
\end{alem}
\begin{proof}
Let $\gamma\in[0,\infty]$.
Let $\lambda\in\mathbb C$ and $f\in\mathbb L^a(\pi)$, $f\neq 0$, be such that
$P_\gamma f=\lambda f$ in $\mathbb L^a(\pi)$, i.e.
$\lambda f=2[\alpha\pi(f\, h_\gamma)+(1-\alpha)U(f \, h_\gamma)]$
that can be rewritten
$$\lambda f\, h_\gamma=2(\alpha\pi(f\, h_\gamma)\, h_\gamma+(1-\alpha)\tilde U_\gamma
     (f \, h_\gamma)).$$
Observe that $\pi(f\, h_\gamma)\ne 0$. Indeed $\pi(f\, h_\gamma) = 0$ would imply
$\lambda f \, h_\gamma=2(1-\alpha)\tilde U_\gamma(f \, h_\gamma)$, which contradicts the fact that
$\lambda/(2-2\alpha)$ is not in the spectrum of $\tilde U_\gamma$.
Now setting $g:=f\, h_\gamma/\pi(f\, h_\gamma)$, we have
\begin{equation}\label{equationg}
\lambda g=2(\alpha \, h_\gamma+(1-\alpha)\tilde U_\gamma (g))
\end{equation}
and so
$$\left[id- \frac{2(1-\alpha)}\lambda\tilde U_\gamma\right] (g)=\frac{2\alpha}\lambda \, h_\gamma.$$
Hence
$$ g=\frac{2\alpha }\lambda
  \left[id- \frac{2(1-\alpha)}\lambda\tilde U_\gamma\right]^{-1} (h_\gamma)
   =\frac{2\alpha}\lambda\sum_{n\ge 0}\frac{2^n(1-\alpha)^n}{\lambda^n}
      \tilde U_\gamma^n h_\gamma$$
and so
$$\lambda=\lambda\pi(g)=2\alpha\sum_{n\ge 0}\frac{2^n(1-\alpha)^n}{\lambda^n}
      \pi\left(\tilde U_\gamma^n \, h_\gamma\right).$$
\end{proof}

Let $(Z_n)_n$ be a Markov process with transition $U$.
Note that $\pi(\tilde U_\gamma^n(h_\gamma))=\mathbb E_\pi[e^{-\gamma \sum_{k=0}^nZ_k}]$ if $\gamma\in[0,\infty)$ and $
\pi(\tilde U_\infty^n(h_\infty))= P_\pi[\sum_{k=0}^nZ_k=0]$.
Hence \eqref{eqlambda} can be rewritten
\begin{equation}\label{eqlambdav2}
\lambda=2\alpha\sum_{n\ge 0}\frac{2^n(1-\alpha)^n}{\lambda^n}\mathbb E_\pi[e^{-\gamma \sum_{k=0}^nZ_k}].
\end{equation}

Assume $\alpha>1/2$ and $\pi(\xi>0)=1$ (so that $r(\infty)=0$).
Let
$c:=\sup J_0$.
Due to \eqref{thmkellerliverani1},
we know that $\lim_{\gamma\r c^{-}}r(\gamma)\le 2(1-\alpha)$ and that
$r$ is continuous on $J_0$.
So $\nu$ is well defined and is such that
$2\alpha\sum_{n\ge 0}{2^n(1-\alpha)^n}\mathbb E_\pi[e^{-\gamma \sum_{k=0}^nZ_k}]=1$
(note that this quantity is 2 when $\gamma=0$
and that its limit at infinity is smaller than 1).
\begin{proof}[Proof of the existence of the constant $C_\nu$ in
\eqref{CDEF}]
We apply Theorem~\ref{cor-th2}.
Assume that $\alpha>1/2$, that (\ref{cond-Zn-knud}) holds,
and that $\pi(\xi^\tau)<\infty$ for some $\tau>1$. Let $p>\frac{\tau}{\tau-1}$ and set $a_3:=\frac{p}{p-1}$ (ie.~$1/p+1/a_3 = 1$). Note that $a_3 < \tau$. Let $a_2$ be such that $a_3<a_2<\tau$. Since $\lim_{a\r+\infty} \frac{\tau a}{\tau+ a} = \tau$, we can chose $a_1>a_2$ such that
$a_2 < \frac{\tau a_1}{\tau+ a_1}$. Next let $a_0 > a_1$. From Lemma~\ref{prop1} we deduce  that the assumptions of Theorem \ref{generalspectraltheorem2} hold with the spaces $\cB_i=\L^{a_i}(\pi)$ for $i=0,1,2,3$, so that we can apply Theorem~\ref{cor-th2}: we conclude that $r$ is $C^1$ on $[0,\theta_1)$.
The fact that $r'<0$ can easily be proved using Proposition \ref{pro-deri-stric-pos}. In this particular case, we can also use the fact that $r$ is given by an implicit formula $F(r(\gamma),\gamma)=0$ (see \eqref{eqlambda}), with $F$
$C^1$ with non-null derivatives at $(r(\gamma),\gamma)$.
Moreover $C_\nu$ in \eqref{CDEF} is well-defined,
provided that $\mu_{\gamma} : f\mapsto \mu(e^{-\gamma\xi}f)$ is in $(\L^{a_3}(\pi))^*$ and that the function $\gamma\mapsto \mu_{\gamma}$ is continuous from $J_0$ to $(\L^{a_3}(\pi))^*$. These conditions hold since $\mu$ is absolutely continuous with respect to $\pi$ with density in $\L^{p}(\pi)$ (see the proof of the multiplicative ergodicity).
Moreover \eqref{def-B} has been proved together with the multiplicative ergodicity.
\end{proof}
%
%
%
%
%
%
%
%
\section{Linear autoregressive model: proof of Theorem \ref{thmAR}}\label{proofAR}

Let $(X_n)_{n\in\N}$ be a linear autoregressive model as described in Example~\ref {exemp-AR}.
Then $(X_n)_{n\in\N}$ is a Markov chain with transition kernel
\begin{equation} \label{P-AR}
P(x,A)=\int_{\R} \mathbf{1}_A(\alpha x + y) p(y)\, dy = \int_{\R} \mathbf{1}_A( y) p(y-\alpha x)\, dy.
\end{equation}
Recall that the density $p$ is assumed to satisfy the domination condition (\ref{dom-nu}) ensuring that $p$ has a moment of order $r_0$, that is
\begin{equation} \label{moment-AR}
\int |x|^{r_0} p(x) dx <\infty.
\end{equation}
In fact the domination condition (\ref{dom-nu}) means that $p$ satisfies (\ref{moment-AR}) under a (local) uniform domination way.

Set $V(x) := (1+|x|)^{r_0}$, $x\in\R$. Recall that, under Assumption~(\ref{moment-AR}), $P$ satisfies the following drift condition (see \cite{MeynTweedie09})
\begin{equation} \label{inequality-drift}
\forall \delta > |\alpha|^{r_0},\ \exists L\equiv L(\delta) > 0,\quad PV \leq \delta \, V + L\, \mathbf{1}_\X.
\end{equation}
Moreover it is well-known that $(X_n)_{n\in\N}$ is $V$-geometrically ergodic, see \cite{MeynTweedie09}.
Let $(\cB_V,\|\cdot\|_V)$ be the weighted-supremum Banach space
\begin{equation} \label{def-BV}
\cB_V := \big\{ \ f : \X\r\C, \text{ measurable }: \|f\|_V  := \sup_{x\in\X} |f(x)| V(x)^{-1} < \infty\ \big\}.
\end{equation}
Let $(\cC_V,\|\cdot\|_V)$ denote the following subspace of $\cB_V$:
$$\cC_V := \bigg\{ \ f\in\cB_V : \text{ $f$ is continuous and }\ \lim_{|x|\r\infty} \frac{f(x)}{V(x)}\ \text{exists in }\ \C\bigg\},$$
where the symbol $\lim_{|x|\r\infty}$ means that the limits when $x\r\pm\infty$ exist and are equal.
Note that $V\in\cC_V$ and that $\cC_V$ is a closed subspace of $(\cB_V,\|\cdot\|_V)$.       For every $f\in\cC_V$ we define
$$\ell_V(f) := \lim_{|x|\r\infty} \frac{f(x)}{V(x)}.$$
Let $\cC_{0,V} := \{f\in\cC_V : \ell_V(f)=0\}$. Finally we denote by $(\cC_b,\|\cdot\|_\infty)$ the space of bounded continuous complex-valued functions on $\R$ endowed
with the supremum norm $\|\cdot\|_\infty$.
We will see below that, for every $\gamma\in(0,+\infty]$, $P_\gamma$ continuously acts on $\cC_V$ (see Lemma~\ref{C-V-0}). For $\gamma\in(0,+\infty]$, we denote by $r(\gamma)$ the spectral radius of $P_\gamma$ on $\cC_V$, that is:
$$r(\gamma) \equiv r(P_\gamma) := \lim_n\|P_\gamma^n\|_V^{1/n} = \lim_n\|P_\gamma^n V\|_V^{1/n}$$
where $\|\cdot\|_V$ also denotes the operator norm on $\cC_V$. We will also prove that $\lim_{\gamma\r0_+}r(\gamma) \geq 2$.

Recall that $\xi:\X\rightarrow [0,+\infty)$ is a measurable function and that $S_n=\sum_{k=0}^n \xi(X_k)$.
Let $\kappa:\X\r \{2,...\}$ be a measurable function.
Theorem~\ref{proprieteAR} below directly follows from the next more precise theorem when applied with $\mu=\delta_x$ or $\mu=\pi$.
\begin{atheo}\label{proprieteAR}
Assume that the previous assumptions hold.
Assume that the distribution $\mu$ of $X_0$ belongs to $\cC_V^{\, *}$, namely satisfying $\mu(V) < \infty$.
Assume moreover that $\xi$ is coercive, that $\kappa$ is bounded, that $p$ is continuous, and that $\sup_{\mathbb R}\xi/V<\infty$. Then
\begin{enumerate}[a)]
\item
$(S_n,\kappa(X_n))_n$ is multiplicatively ergodic on $(0,+\infty)$
with $\rho=r$ on $[0,+\infty)$.
\item If moreover the Lebesgue measure of the set $[\xi=0]$ is zero,
then $\lim_{\gamma\rightarrow +\infty}r(\gamma)=0$. Hence $\nu$ is finite.
\item Moreover, if  there exists $\tau>0$ such that $\sup_{\R}\xi^{1+\tau}/V<\infty$, then $\gamma\mapsto   r(\gamma)$ admits a negative derivative on $[0,+\infty)$. Hence \eqref{CDEF} holds also with $C_\nu\in(0,+\infty)$.
\end{enumerate}
\end{atheo}
The next subsections are devoted to the proof of Theorem~\ref{proprieteAR}.
\subsection{ Study of Hypothesis \ref{hypKL}*}
%
In this subsection we prove that $(P_\gamma,J,\cB_0,\cB_1)$ satisfies Hypothesis \ref{hypKL}* with $J=(0,+\infty]$, $\cB_0=\cC_b$, and $\cB_1=\cC_V$.
\begin{alem} \label{C-V-0}
Assume that Assumption~(\ref{dom-nu})
holds (thus (\ref{moment-AR})), that $p$ is continuous and that $\xi$ is
coercive. Then, for every $\gamma\in(0,+\infty]$, $P_\gamma$ continuously acts on both $\cC_b$ and $\mathcal C_V$. For every $\gamma\in[0,+\infty]$, $P_\gamma$  is compact from $\cC_b$ into $\cC_V$. For every $\gamma\in(0,+\infty]$, we have $P_\gamma(\cB_V) \subset \cC_{0,V}$.
\end{alem}
\begin{proof}{}
Let $\gamma\in[0,+\infty]$. From (\ref{inequality-drift}) it easily follows that $P_\gamma V \leq PV \leq (\delta+L)V$, so that $P_\gamma$ continuously acts on $\cB_V$.
Let $f\in\cB_V$. Then
\begin{equation} \label{P-gam-chi}
\forall x\in\R,\quad (P_\gamma f)(x) = \int_\R \psi_\gamma(x,y)\, dy
\end{equation}
$$\text{with}\quad \left\{ \begin{array}{lcl}
\psi_\gamma(x,y) := \kappa(y)e^{-\gamma \xi(y)}\, f(y)\, p(y-\alpha x)  & \text{if} & 0 \leq \gamma < \infty  \\[0.2cm]
\psi_\gamma(x,y) := \kappa(y)1_{\{\xi=0\}}(y)\, f(y)\, p(y-\alpha x)  & \text{if} & \gamma = +\infty.
\end{array}\right.$$
Let $A>0$. We deduce from Assumption~(\ref{dom-nu}) and from a usual compactness argument ($[-A,A]$ is compact) that there exists a non-negative function $q\equiv q_A$ such that $y\mapsto V(y)\, q(y)$ is Lebesgue-integrable and
$$\forall v\in [-A,A],\ \forall y\in\R,\quad p(y+v) \leq q(y).$$
Thus we have for every $x\in [-A,A]$ and for every $y\in\R$
\begin{equation} \label{ineg-qA}
|\psi_\gamma(x,y)| \, \leq\,  \|\kappa\|_\infty\|f\|_V \left(1+|y|\right)^{r_0} p(y-\alpha x) \leq
 \|\kappa\|_\infty\|f\|_V V(y)\, q(y).
\end{equation}
Since $x\mapsto\psi_\gamma(x,y)$ is continuous for every $y\in\R$ from the continuity of $p$, we deduce from Lebesgue's theorem that the function $P_\gamma f$ is continuous on $\R$. We have proved that, if $\gamma\in[0,+\infty]$ and if $f\in\cB_V$, then $P_\gamma f$ is continuous on $\R$. Thus $P_\gamma$ continuously acts on $\cC_b$.

Now, if $\gamma\in(0,+\infty]$, then
$$\forall x\in\R,\quad \frac{(P_\gamma f)(x)}{V(x)} = \int_\R \chi_\gamma(x,y)\, dy\qquad \text{with } \ \chi_\gamma(x,y) := \theta_\gamma(\alpha x+y)\, \frac{f(\alpha x+y)}{V(x)}\, p(y)$$
$$\text{where}\quad \theta_\gamma(\alpha x+y) :=
\left\{ \begin{array}{lcl}
\kappa(y)e^{-\gamma \xi(\alpha x+y)}  & \text{if} & 0 < \gamma < \infty  \\[0.2cm]
\kappa(y)\mathbf 1_{\{\xi=0\}}(\alpha x+y)  & \text{if} & \gamma = +\infty.
\end{array}\right.$$
For every $(x,y)\in\R^2$, we obtain that
$$|\chi_\gamma(x,y)| \, \leq\, \theta_\gamma(\alpha x+y)\,  \|f\|_V\, \left(\frac{1+|x|+|y|}{1+|x|}\right)^{r_0} p(y) \, \leq\,  \|\kappa\|_\infty\|f\|_V \big(1+|y|\big)^{r_0} p(y).$$
Moreover $\lim_{|x|\r +\infty} \theta_\gamma(\alpha x+y) = 0$ since $\xi$ is coercive. It follows again from Lebesgue's theorem that
$$\lim_{|x|\r +\infty} \frac{(P_\gamma f)(x)}{V(x)} = 0,$$
thus $P_\gamma f \in \cC_{0,V}$. We have proved that, if $\gamma\in(0,+\infty]$, then
$P_\gamma(\cB_V) \subset \cC_{0,V}$, thus the last assertion of Lemma~\ref{C-V-0} holds and $P_\gamma$ continuously acts on $\cC_V$.
Finally, to prove the compactness property, let $\gamma\in[0,+\infty]$ and consider $P_\gamma$ as written in (\ref{P-gam-chi}). Since $p$ is continuous, the image by $P_\gamma$ of the unit ball $\{f\in\cC_b : \|f\|_\infty \leq 1\}$ is equicontinuous from Scheff\'e's lemma. Then $P_\gamma$
is compact from $\cC_b$ into $\cC_V$ from Ascoli's theorem and from $\lim_{|x|\r\infty} V(x) = +\infty$.
\end{proof}

With the usual convention $V^0:=1$, we have the identification $\cC_{V^0}=\cC_b$. Then the continuity of $\gamma\mapsto P_\gamma$ from $(0,+\infty]$ into $\cL(\cC_b,\cC_V)$ follows from the following.

\begin{alem} \label{lem-cont-P-gamma}
Let $0\leq a < a+b \leq 1$.
Assume that $\xi \leq cV$ for some positive constant $c$. Then the following operator-norm inequality holds for every $(\gamma,\gamma')\in[0,+\infty)^2$
$$
\|P_\gamma - P_{\gamma'}\|_{{\cC}_{V^a},{\cC}_{V^{a+b}}}\ := \sup_{f\in{\cC}_{V^a},\|f\|_{V^{a}}\leq 1}\|P_\gamma f - P_{\gamma'} f\|_{V^{a+b}} \ \ \leq \|\kappa\|_\infty (c|\gamma - \gamma'|)^b\Vert P\Vert_{V^{a+b}} \, .$$
\end{alem}
\begin{proof}{}
Let $(\gamma,\gamma')\in[0,+\infty)^2$. For all $(u,v)\in[0,+\infty)^2$, we have $|e^{-u} - e^{-v}| \leq |e^{-u} - e^{-v}|^b \leq |u-v|^b$ from Taylor's inequality. Thus we obtain for any $f\in\cC_{V^a}$
\begin{eqnarray*}
\big|(P_\gamma f)(x) - (P_{\gamma'} f)(x)\big| &\leq&
\|\kappa\|_\infty\Vert f\Vert_{V^a}  \int_\R \big|e^{-\gamma\xi(y)} - e^{-\gamma'\xi(y)}\big| (V(y))^a p(y-\alpha x)\, dy \\
&\leq& \|\kappa\|_\infty\Vert f\Vert_{V^a} (c\, |\gamma - \gamma'|)^b \int_\R (V(y))^{a+b}\,  p(y-\alpha x)\, dy \\
&\leq& \|\kappa\|_\infty\Vert f\Vert_{V^a} (c\, |\gamma - \gamma'|)^b  P V^{a+b}(x),
\end{eqnarray*}
from which we deduce the desired inequality.
\end{proof}
\begin{alem} \label{prop-1-AR}
Assume that Assumption~(\ref{dom-nu}) holds (thus (\ref{moment-AR})) and that $\xi$ is coercive. Then
$$\|P_\gamma- P_{\infty}\|_{{\cC}_b,{\cC}_V} := \sup_{f\in{\cC}_b,\|f\|_\infty\leq 1}\|P_\gamma f- P_{\infty}f\|_V \ \longrightarrow 0
\quad \text{when}\  \gamma\r+\infty.$$
\end{alem}
\begin{proof}{}
Let $\varepsilon >0$. Let $f\in\cC_b$ be such that $\|f\|_\infty \leq 1$. From $|P_\gamma f| \leq P\mathbf{1}_\X = \mathbf{1}_\X$ it follows that there exists $A\equiv A(\varepsilon)$ such that :
\begin{equation} \label{inequality-1}
|x| > A\ \Rightarrow\ \forall \gamma\in(0,+\infty),\ \frac{|(P_\gamma f)(x)|}{V(x)} \leq \varepsilon.
\end{equation}
Moreover, for any $\beta >0$ and $x\in\R$ such that $|x|\leq A$, we obtain that
\begin{eqnarray*}
\big|(P_\gamma f-P_{\infty} f)(x)\big| &\le& \|\kappa\|_\infty
e^{-\gamma\beta} \int_{[\xi > \beta]} p(y-\alpha x)\, dy + \|\kappa\|_\infty \int_{[0<\xi \leq \beta]} p(y-\alpha x)\, dy  \\
&\leq& \|\kappa\|_\infty\left( e^{-\gamma\beta}  + \int_{[0<\xi \leq \beta]} q(y)\, dy\right)\,
\end{eqnarray*}
where $q\equiv q_A$ is the function given in (\ref{ineg-qA}).
Since $q$ is Lebesgue-integrable on $\R$, we have $\int_{[0<\xi \leq \beta]} q(y)\, dy \r 0$ when $\beta\r 0$, so that there exists $\beta_0\equiv \beta_0(\varepsilon)>0$ such that
$$\|\kappa\|_\infty \int_{[0<\xi \leq \beta_0]} q(y)\, dy \leq \frac{\varepsilon}{2}.$$
Finally let $\gamma_0 \equiv \gamma_0(\varepsilon)>0$ be such that : $\forall \gamma > \gamma_0,\ \|\kappa\|_\infty e^{-\gamma\beta_0} \leq \varepsilon/2$. Then
\begin{equation} \label{inequality-2}
|x| \leq A\ \Rightarrow\ \forall \gamma\in(\gamma_0,\infty),\ \frac{|(P_\gamma f
        - P_{\infty} f)(x)|}{V(x)} \leq  |(P_\gamma f)(x) - (P_{\infty}f)(x)| \leq  \varepsilon.
\end{equation}
Inequalities (\ref{inequality-1}) and (\ref{inequality-2}) provides the desired statement.
\end{proof}
%

To study the conditions (\ref{cond-r-ess-dual}) and (\ref{D-F-dual}) of Hypothesis \ref{hypKL}* with $J=(0,+\infty]$, $\cB_0=\cC_b$, and $\cB_1=\cC_V$,  we use the duality arguments of \cite[prop.~5.4]{HerLed14}.
The topological dual spaces of $\cC_V$ and $\cC_b$ are denoted by $(\cC_V^*,\|\cdot\|_V)$ and $(\cC_b^*,\|\cdot\|_\infty)$ respectively (for the sake of simplicity we use the same notation for the dual norms). For any $\gamma>0$, we denote by $P^*_\gamma$ the adjoint operator of $P_\gamma$ on $\cC_V$.  Note that each $P_\gamma^*$ is a contraction with respect to the dual norm $\|\cdot\|_\infty$  because so is $P_\gamma$ on $\cC_b$.

In the sequel, $\delta > |\alpha|^{r_0}$ is fixed, as well as the associated constant $L\equiv L(\delta)$ in (\ref{inequality-drift}).

\begin{alem} \label{lem-D-F-P-gamma}
Assume that (\ref{moment-AR}) holds, that $\kappa$ is bounded, and that $\xi$ is coercive.
Then, for every $\gamma >0$ and for every $\beta>0$, there exists a positive constant $L_{\beta}$ such that
\begin{equation} \label{drift-P-gamma}
P_\gamma V \leq \|\kappa\|_\infty(e^{-\gamma\beta}\, \delta \, V + L_{\beta}\, \mathbf{1}_\X).
\end{equation}
Moreover
\begin{equation} \label{drift-P-gamma-infty}
P_{\infty} V \leq \|\kappa\|_\infty\left(\sup_{[\xi =0]} V\right) \mathbf{1}_\X.
\end{equation}
\end{alem}
\begin{proof}{}
We have for every $\gamma >0$ and for every $\beta>0$
\begin{eqnarray*}
P_\gamma V &=& P(\kappa e^{-\gamma\xi} V) = P\big(\kappa e^{-\gamma\xi} \mathbf{1}_{[\xi > \beta]}V\big) + P\big(\kappa e^{-\gamma\xi} \mathbf{1}_{[\xi \leq \beta]}V\big)  \\
&\leq&\|\kappa\|_\infty\left( e^{-\gamma\beta} \big(\delta \, V + L\, \mathbf{1}_\X\big) + \int_{[\xi \leq \beta]} V(y) P(\cdot,dy)\right) \qquad \qquad \text{(from (\ref{inequality-drift}))}  \\
&\leq& \|\kappa\|_\infty\left(e^{-\gamma\beta}\, \delta \, V + \big(L + \sup_{[\xi \leq \beta]} V\big)\mathbf{1}_\X \right)
\end{eqnarray*}
from which we deduce the first desired statement.
For $P_{\infty}$, we have
$$
P_{\infty} V=P(\kappa\mathbf 1_{\{\xi=0\}}V) \leq
\left(\sup_{[\xi =0]} V\right) P(\kappa)\le \left(\sup_{[\xi =0]} V\right)\|\kappa\|_\infty\mathbf 1_\X.
$$
\end{proof}

\begin{acor} \label{cor-QC-gamma}
Assume that Assumption~(\ref{moment-AR}) holds true, that $\kappa$ is bounded, and that $\xi$ is coercive. Then,
for every $\gamma_1>0$ and for every $\varepsilon > 0$, there exists a constant $D>0$
such that
\begin{equation}\label{DFgamma>0}
\forall \gamma\in[\gamma_1,+\infty],\ \forall f^*\in\cC^*_V,\quad \|P^*_\gamma f^*\|_V \leq \varepsilon \, \|f^*\|_V + D\, \|f^*\|_\infty.
\end{equation}
Moreover, for every $\gamma\in(0,+\infty]$, the essential spectral radius $r_{ess}(P_\gamma)$ is zero.
\end{acor}
\begin{proof}{}
Choose $\beta = \beta(\gamma_1,\varepsilon)>0$ such that $\|\kappa\|_\infty e^{-\gamma_1\beta}\, \delta < \varepsilon$. Then we deduce from Lemma~\ref{lem-D-F-P-gamma} that $P_{\gamma_1} V \leq \varepsilon\, V +  D\, \mathbf{1}_\X$, where $D\equiv D(L,\gamma_1,\varepsilon)$ is a positive constant.
Now let $\gamma\in[\gamma_1,+\infty]$. Since $P_\gamma V \leq P_{\gamma_1}V$, we also have $P_{\gamma} V \leq \varepsilon\, V +  D\, \mathbf{1}_\X$.
This inequality easily rewrites as (\ref{DFgamma>0}) (see the proof in \cite[p.~190]{FerHerLed13}). Finally, since $P^*_\gamma$ is compact from $\cC^*_V$ into $\cC^*_b$ (Lemma~\ref{C-V-0}), we deduce from \cite{Hen93} that
$r_{ess}(P_\gamma) \leq \varepsilon$. We obtain $r_{ess}(P_\gamma) = 0$ because $\varepsilon$ is arbitrary.
\end{proof}
\begin{arem}\label{cor-QC-gammabis}
Let $\gamma_1>0$,
$\varepsilon>0$ and $0 \le a\le a+b\le 1$.
Observe that Corollary \ref{cor-QC-gamma} holds also if we replace $V$ by $V^{a+b}$ (since $\vartheta_1$ admits a moment of order $r_0(a+b)$). Moreover notice that
\eqref{DFgamma>0}  with $V^{a+b}$
instead of $V$ directly gives
that there exists a constant $D_{\varepsilon,a+b}>0$ such that
\begin{equation}\label{DFgamma>0bis}
\forall \gamma\in[\gamma_1,+\infty],\ \forall f^*\in\cC^*_{V^{a+b}},\quad \|P^*_\gamma f^*\|_{V^{a+b}} \leq \varepsilon \, \|f^*\|_{V^{a+b}} + D_{\varepsilon,a+b}\, \|f^*\|_{V^a}
\end{equation}
since $\Vert f^*\Vert_\infty\le\Vert f^*\Vert_{V^a}$.
\end{arem}

\subsection{A preliminary useful statement on $r(\gamma)$}

\begin{apro} \label{r-grand-2}
Assume that Assumption~(\ref{moment-AR}) holds true, that $\kappa$ is bounded, that $\xi$ is coercive  and finally that the function $\xi/V$ is bounded on $\R$. Then $\displaystyle \lim_{\gamma\r0_+}r(\gamma) \geq 2$.
\end{apro}
\begin{proof}{}
We need the following lemma concerning the special case $\kappa\equiv 2$. To avoid confusion we write below $\widetilde P_\gamma$ and $\widetilde r(\gamma)$ in place of $P_\gamma$ and $r(\gamma)$ when $\kappa\equiv 2$.
\begin{alem} \label{lem-kappa-2}
Assume that (\ref{moment-AR}) holds, that $\kappa\equiv 2$, and that $\xi$ is coercive. Then $\widetilde P_0$ continuously acts on $\mathcal C_V$. Moreover the function $\gamma\mapsto \widetilde r(\gamma)$ is continuous at $\gamma=0$, with $\widetilde r(0)=2$.
\end{alem}
From $P_\gamma^n \geq \widetilde P_\gamma^n$ (since $\kappa \geq 2$), we deduce that
$r(\gamma) = r(P_\gamma) \geq r(\widetilde P_\gamma) = \widetilde r(\gamma)$.
It follows from Proposition~\ref{LEMME0} that
$\lim_{\gamma\r0_+}r(\gamma) \geq \lim_{\gamma\r0_+}r(\gamma) = 2$.
\end{proof}
\begin{proof}[Proof of Lemma~\ref{lem-kappa-2}]
Note that $\widetilde P_0 =2P$. If $f\in\cC_V$, then
$$\forall x\in\R,\quad \frac{(Pf)(x)}{V(x)} = \int_\R \chi_0(x,y)\, dy\qquad \text{with } \ \chi_0(x,y) := \frac{f(\alpha x+y)}{V(x)}\, p(y).$$
We know that $|\chi_0(x,y)| \leq \|f\|_V \big(1+|y|\big)^{r_0} p(y)$, and writing $\chi_0(x,y) := \frac{f(\alpha x+y)}{V(\alpha x+y)}\, \frac{V(\alpha x+y)}{V(x)}\, p(y)$
 proves that $\lim_{|x|\r +\infty} \chi_0(x,y) = \ell_V(f)\, |\alpha|^{r_0}\, p(y)$. It follows from Lebesgue's theorem that $\lim_{|x|\r +\infty} \frac{(Pf)(x)}{V(x)} = \ell_V(f)\, |\alpha|^{r_0}$. Thus $P(\cC_V) \subset \cC_V$ and $P$ continuously acts on $\cC_V$.

Iterating Inequality (\ref{inequality-drift}) proves that $P$ is power-bounded on $\cC_V$ (i.e.~$\sup_{n\geq 1}\|P^nV\|_V < \infty$), thus
$r(P)=1$ since $P$ is Markov. Moreover (\ref{inequality-drift}) rewrites as the following (dual) Doeblin-Fortet inequality (see the proof in \cite[p.~190]{FerHerLed13}):
\begin{equation}\label{DFgamma=0}
\forall f^*\in\cC_V^*,\quad \|P^* f^*\|_V \leq \delta  \, \|f^*\|_V + L\, \|f^*\|_\infty.
\end{equation}
Since $P$ is compact from $\cC_b$ into $\cC_V$
(apply the same argument as in Lemma~\ref{C-V-0}),
so is $P^*$ from $\cC_V^*$ into $\cC_b^*$. Then we deduce from \cite{Hen93} that, under Assumption~(\ref{moment-AR}), $P$ is a quasi-compact operator on $\cC_V$ and its essential spectral radius  $r_{ess}(P)$ satisfies the following bound (see also \cite[Sect.~8]{Wu04}): $r_{ess}(P) \leq \delta$. It follows that
\begin{equation} \label{r-ess-P}
\widetilde r(0)= r(\widetilde P_0) = 2\quad \text{and} \quad r_{ess}(\widetilde P_0) \leq 2\delta.
\end{equation}
Observe that (\ref{inequality-drift}) and Inequality $\widetilde P_\gamma V \leq 2PV$give $\widetilde P_\gamma V \leq 2\delta \, V + 2L\, \mathbf{1}_\X$, which  rewrites as the following Doeblin-Fortet inequality:
\begin{equation}\label{DFgamma=0-bis}
\forall \gamma\in[0,+\infty),\ \forall f^*\in\cC_V^*,\quad \|{\widetilde P_\gamma}^* f^*\|_V \leq 2\delta  \, \|f^*\|_V + 2L\, \|f^*\|_\infty.
\end{equation}
Using (\ref{r-ess-P}), (\ref{DFgamma=0-bis}), Corollary \ref{cor-QC-gamma}  and Lemma~\ref{lem-cont-P-gamma} (with $a=0$ et $b=1$), it follows from
Theorem~\ref{generalspectraltheorem1}
(applied with $\delta_0 = 2\delta$) that $\gamma\mapsto \widetilde r(\gamma)$ is continuous at $\gamma=0$ since $\widetilde r(0)=2 > 2\delta$.
\end{proof}

\subsection{Study of Hypothesis~\ref{hypcompl}} \label{sec-hyp-val-peri}
In this subsection we prove that Hypothesis \ref{hypcompl} holds with respect to $(J_1,\cB_1)$ with $\cB_1=\mathcal\cC_{V}$ and $J_1:=(0,\theta_1)$, where
\begin{equation} \label{theta-1}
\theta_1:=\sup\{\gamma>0\, :\, r(\gamma)>0\}.
\end{equation}
We know from Proposition~\ref{r-grand-2} that $\theta_1>0$. Since $\mathcal\cC_{V}$ is a Banach lattice, we use Proposition~\ref{pro-reduc-BL} to prove Hypothesis \ref{hypcompl}.

Observe that, for any $\gamma\in(0,\theta_1)$, $P_\gamma$ is quasi-compact on $\cC_V$ from Corollary~\ref{cor-QC-gamma} since $r(\gamma)>0$.
The other conditions of Proposition~\ref{pro-reduc-BL}
follow from Remark~\ref{rqe-dec-fl} and Lemmas~\ref{lem-val-prop-simple}-\ref{lem-val-prop-simple2} below. First we state the following.
\begin{alem} \label{dec-fl}
For any non-null $e^*\in \cC_V^*$, $e^*\geq 0$, there exists a nonnegative measure $\mu\equiv \mu_{e^*}$ on $(\R,\cX)$ such that
\begin{equation} \label{dec-fl-exp}
\forall f\in\cC_V,\quad e^*(f) = \mu\left(\frac{f}{V} - \ell_V(f)\, \mathbf{1}_\R\right) + e^*(V)\, \ell_V(f).
\end{equation}
\end{alem}
\begin{arem}\label{rqe-dec-fl}
Due to Lemma~\ref{dec-fl},
Hypothesis~\ref{(B)}
is fulfilled with $J_1$ and $\cB=\cC_V$. Indeed, let $\gamma\in J_1$ and let $\phi\in\cC_V$ be non-null and non-negative. Then, we have $P_\gamma \phi > 0$ everywhere from the definition of $P$ and the strict positivity of the function $p(\cdot)$. Now prove that, if $\psi\in\cB^*\cap\ker(P_\gamma^*-r(\gamma)I)$ is non-null and non-negative, then $\psi(P_\gamma\phi)>0$, so $\psi$ is positive. Let $\gamma>0$. First observe that $\psi\neq c\, \ell_V$ for every $c\in\C$ because $r(\gamma) > 0$ and $P_\gamma^*(\ell_V) = 0$ from Lemma~\ref{C-V-0}. Second note that $\mu=0$ in (\ref{dec-fl-exp}) implies that $e^*=e^*(V)\, \ell_V$. Thus the nonnegative measure $\mu\equiv \mu_{\psi}$ associated with $\psi$ in (\ref{dec-fl-exp}) is non-null. Since $\ell_V(P_\gamma\phi)=0$ from Lemma~\ref{C-V-0}, we deduce from (\ref{dec-fl-exp}) (applied with $e^*=\psi$) and from $P_\gamma \phi > 0$ that $\psi(P_\gamma\phi) = \mu(P_\gamma\phi/V)>0$.
\end{arem}
\begin{proof}[Proof of Lemma~\ref{dec-fl}]
Let $(\cC,\|\cdot\|)$ denote the following space
$$\cC := \bigg\{ \ g : \R\r\C\  \text{ continuous }: \|g\| := \sup_{x\in\R} |g(x)| < \infty\ \text{ and }\ \lim_{|x|\r\infty} g(x)\ \text{exists in}\ \C\bigg\}.$$
For every $g\in\cC$, we set: $\ell(g) := \lim_{|x|\r\infty} g(x)$. We denote by $\cC^*$ the topological dual space of $\cC$. Let $e^*\in \cC_V^*$, $e^*\geq 0$, and let $\widetilde e^*\in\cC^*$ be defined by:
$$\forall g\in\cC,\quad \widetilde e^*(g) := e^*(gV).$$
Next let $\widetilde e^*_0$ be the restriction of $\widetilde e^*$ to $\cC_0:=\{g\in\cC : \ell(g)=0\}$. From the Riesz representation theorem, there exists a unique positive measure $\mu$ on $(\R,\cX)$ such that
$$\forall g\in\cC_0,\quad \widetilde e^*_0(g) = \mu(g) := \int_\R g\, d\mu.$$
Then, writing $g = (g-\ell(g)\, \mathbf{1}_\R)+ \ell(g)\, \mathbf{1}_\R$ for any $g\in\cC$, we obtain that
$$\widetilde e^*(g) = \mu\big(g-\ell(g)\, \mathbf{1}_\R\big) +  \widetilde e^*(\mathbf{1}_\R)\, \ell(g).$$
We conclude by observing that, for any $f\in\cC_V$, we have $e^*(f) = \widetilde e^*(f/V)$.
\end{proof}

\begin{alem} \label{lem-val-prop-simple}
If $f,g\in \cC_V$ are such that
$P_\gamma f=r(\gamma)f$ and $P_\gamma g=r(\gamma)g$ with $f>0$, then $g\in \C\cdot f$.
\end{alem}
\begin{proof}{}
Let $f,g\in\ker(P_\gamma - r(\gamma) I)$ with $f>0$.
Let $\beta\in\C$ be such that $h:=g - \beta f$ vanishes at $0$. Since $h\in\ker(P_\gamma - r(\gamma) I)$ we deduce from Proposition~\ref{firstorder}  that $P_\gamma |h| = r(\gamma) |h|$. Then $|h|(0)=0$, the positivity of $p(\cdot)$ and finally the continuity of $|h|$ show that  $h=0$.
\end{proof}
\begin{alem} \label{lem-val-prop-simple2}
Let $h\in\cC_V$ with $|h|>0$ and $\lambda\in\mathbb C$ be such that
$|\lambda|=1$ and $P\frac{h}{|h|}=\lambda \frac{h}{|h|}$ in $\mathbb L^1(\pi)$. Then $\lambda=1$.
\end{alem}
\begin{proof}{}
Observe that $\frac h{|h|}$ is in $\mathcal C_b$ so in $\cB_V$.
But it is known from \cite{MeynTweedie09} that $(X_n)_n$
is $V$-geometrically ergodic, so $\lambda$=1.
\end{proof}

%
\subsection{Proof of Theorem \ref{proprieteAR}}
To prove Assertion $a)$ of Theorem \ref{proprieteAR}, we apply Theorem~\ref
{cor-th1}. Let $\gamma_1$ be such that $0<\gamma_1<\theta_1$, with $\theta_1$ given in (\ref{theta-1}). Let $\varepsilon\in(0,1)$. Then, from the results of the previous subsections, the assumptions of  Theorem \ref{generalspectraltheorem1} hold with $J=(\gamma_1,\theta_1)$, $\cB_0=\cC_b$, $\cB_1=\cC_V$, $\delta_0= \varepsilon$, thus with
$J_0:=\{\gamma\in J : r(\gamma)>\varepsilon\}$. A first consequence is that $\theta_1=+\infty$ from Lemma \ref{rnonnul} and Remark~\ref{rem-C0}. Moreover, for every $\gamma\in J_0$, the map $\mu_{\gamma} : f\mapsto \mu(\kappa e^{-\gamma\xi}f)$ is in $(\cC_V)^*$ from $\mu(V) < \infty$, and we have for every $\gamma,\gamma'\in J_0$ and for every $f\in\cC_V$
$$\left|\mu_{\gamma}(f) - \mu_{\gamma'}(f)\right| \leq \|\kappa\|_\infty\|f\|_V\mu\left(\big|e^{-\gamma\xi} - e^{-\gamma\xi}\big|V\right)$$
so that the norm of $(\mu_{\gamma} - \mu_{\gamma'})$ in $(\cC_V)^*$ is less than $\|\kappa\|_\infty\, \mu(|e^{-\gamma\xi} - e^{-\gamma\xi}|V)$ which converges to $0$ as $\gamma'\r\gamma$ from Lebesgue's theorem. Finally note that the additional condition in Assertion~$(ii)$ of Proposition~\ref{pro-B-gamma} clearly holds from the form of $P$ (see~(\ref{P-AR})) and from the positivity of the density $p$.
Thus \eqref{def-B} holds.
Since $\gamma_1$ and $\varepsilon$ are arbitrarily small, we deduce from Theorem~\ref{cor-th1} that $(S_n,\kappa(X_n))_n$ is multiplicatively ergodic on $(0,+\infty)$ with $\rho_Y(\gamma)=r(\gamma)>0$
on $(0,+\infty)$. We have proved Assertion~$a)$ of Theorem \ref{proprieteAR}.

For Assertion~$b)$, observe that $\Leb(\xi=0)=0$ implies that $P_{\infty}=0$, in particular the spectral radius $r(\infty)$ of $P_{\infty}$ is zero.
Then Lemma~\ref{prop-1-AR} and
Theorem~\ref{generalspectraltheorem1}
give $\lim_{\gamma\rightarrow +\infty}r(\gamma)=r(\infty)=0$.
Consequently $\nu$ is finite and satisfies \eqref{P1bis}, and so \eqref{nuDEF}, with respect to $\P_\mu$, provided that $\mu$ is a probability distribution belonging to $\mathcal C_V^*$.

Finally, to
prove Part~(3) of Theorem \ref{proprieteAR}, we assume now that $\xi\in\cB_{V^{\frac 1{1+\tau}}}$ for some $\tau>0$
and that $[\xi=0]$ has Lebesgue measure 0. Consider any
$$0<a_0<a_1<a_1+\frac{1}{1+\tau}<a_2<a_3=1.$$
Let us prove that the additional assumptions of Theorem \ref{generalspectraltheorem2}
hold true with $\cB_i:=\mathcal \cC_{V^{a_i}}$
for $i\in\{0,1,2,3\}$. Let $i\in\{0,1,2\}$.
The fact that $(P_\gamma)_\gamma$ satisfies the conditions of Hypothesis \ref{hypKL}* on $(J,\cB_{i},\cB_{i+1})$  comes from Lemma \ref{lem-cont-P-gamma} and
Remark \ref{cor-QC-gammabis}.
The fact that Hypothesis \ref{hypcompl} is satisfied on $\cB_{i+1}$
follows from Proposition~\ref{pro-reduc-BL}: apply the results of Subsection~\ref{sec-hyp-val-peri}
with $V^{a_{i+1}}$ in place of $V$.
Observe that
$$\|\xi f\|_{\cB_2}=\sup\frac{\|\xi f\|}{V^{a_2}}
      \le\sup\frac {\|\xi \|}{V^{\frac 1{1+\tau}}}\, \sup\frac{\| f\|}{V^{a_1}}
      \le \|f\|_{\cB_1}\sup \frac{\|\xi \|}{V^{\frac 1{1+\tau}}}.$$
Hence we have proved that $f\mapsto \xi f$ is in $\cL(\cB_1,\cB_2)$.
The fact that $\gamma\mapsto P_\gamma$ is $C^1$ from $(0,+\infty)$
to $\cL(\cB_1,\cB_2)$ and that $P'_\gamma:=P_\gamma(-\xi f)$
comes from the proof of \cite[Lemma 10.4]{HerPen10}.
Finally, due to Proposition~\ref{pro-deri-stric-pos}, we have $r'(\nu)<0$.
We conclude by Theorem~\ref{cor-th2}.

\section{Proof of Theorems~\ref{generalspectraltheorem1} and \ref{generalspectraltheorem2}} \label{proofoperator}

Let us state the Keller-Liverani perturbation theorem.
\begin{atheo}[Keller-Liverani Perturbation Theorem \cite{KelLiv99,Liv03,Ferre}]
\label{thmkellerliverani}
Let $(\mathcal X_0,\|\cdot\|_{\cX_0})$ be a Banach space and $(\mathcal X_1,\|\cdot\|_{\mathcal X_1})$ be a normed space such that $\mathcal X_0\hookrightarrow \mathcal X_1$. Let $J\subset[-\infty,+\infty]$ be an interval and let $(Q(t))_{t\in J}$ be a family of operators. We assume that
\begin{itemize}
\item For every $t\in J$, $Q(t)\in \mathcal L(\mathcal X_0)\cap\mathcal L(\mathcal X_1)$,
\item $t\mapsto Q(t)$ is a continuous map from $J$ in $\mathcal L( \mathcal X_0,\mathcal X_1)$,
\item There exist $\delta_0>0$, $c_0,M_0>0$ such that for every $t\in J$
$$\forall f\in\mathcal X_0,\ \forall n\in \mathbb Z_+,
    \quad
       \|(Q(t))^n f\|_{\mathcal X_0}\le c_0\big(\delta_0^n\|f\|_{\cX_0}
           + M_0^n\Vert  f\Vert_{\cX_1}\big).$$
\end{itemize}
Let $t_0\in J$. Then, for every $\varepsilon>0$ and every $\delta>\delta_0$, there exists $I_0\subset J$ containing $t_0$ such that
$$\sup_{t\in I_0,\, z\in\mathcal D(\delta,\varepsilon)}\|(zI-Q(t))^{-1}\|_{\cX_0}<\infty,$$
with $\mathcal D(\delta,\varepsilon):=\{z\in \mathbb C,\
       d(z,\sigma(Q(t_0)_{|\cX_0}))>\varepsilon,\ |z|>\delta\}$.

Furthermore the map $t\mapsto(zI-Q(t))^{-1}$ from $J$ to
$\mathcal L(\mathcal X_0,\mathcal X_1)$ is continuous at $t_0$ in a uniform way with respect to $z\in\mathcal D(\delta,\varepsilon)$, i.e.
$$\lim_{t\rightarrow t_0,\, t\in J} \sup\left\{\|(zI-Q(t))^{-1}-(zI-Q(t_0))^{-1}\|_{\cX_0,\cX_1}\ :\ z\in \mathcal D(\delta,\varepsilon)\right\}=0.$$
In particular, $\limsup_{t\rightarrow t_0}
r((Q(t))_{|\mathcal X_0})\le \max(\delta_0,r((Q(t_0))_{|\mathcal X_0}))$.
Finally the map
$t\mapsto r((Q(t))_{|\mathcal X_0})$ is continuous
on $\{t\in J : r((Q(t))_{|\cX_0}) > \delta_0\ge r_{ess}((Q(t))_{|\cX_0})\}$.
\end{atheo}

We use the notations of Section \ref{nota} and identify $(X_n)_n$ with the canonical Markov chain. From now on, to simplify notations, we write $R_z(\gamma):= (zI-P_\gamma)^{-1}$ for the resolvent when it is well defined. Recall that $J_0:=\{t\in J : r(\gamma)>\delta_0\}$.

\subsection{Proof of Theorem~\ref{generalspectraltheorem1} under Hypothesis \ref{hypKL}} \label{ap-proof-general-th}
%
Here we assume that $(P_\gamma)_{\gamma}$ satisfies Hypothesis~\ref{hypKL} with $(J,\cB_0,\cB_1)$ and that Hypothesis~\ref{hypcompl} is fulfilled on $J_0$ with $\cB :=\cB_0$. The property (\ref{thmkellerliverani1}) and the continuity on $J_0$ of the function $\gamma\mapsto r(\gamma) := r((P_\gamma)_{|\cB_0})$ follow from Theorem \ref{thmkellerliverani}. Moreover we deduce from Hypothesis~\ref{hypcompl} on $J_0$ with $\cB :=\cB_0$ that (\ref{sup-vit-ponctuel}) holds. It remains to prove that, if $K$  is a compact subset of $J_0$, then the constants $\theta_\gamma$ and $M_\gamma$ and  in (\ref{sup-vit-ponctuel}) are uniformly bounded by some $\theta_K\in(0,1)$ and $M_K\in(0,+\infty)$, and that $\gamma\mapsto \Pi_\gamma$ is continuous from $J_0$ to $\mathcal L(\cB_0,\cB_1)$. To that effect we use below the spectral definition of $\Pi_\gamma$.

Let $\chi: J_0 \rightarrow (0,+\infty)$ be defined by
$\chi(\gamma) := \max\big(\delta_0,\lambda(\gamma))$,
where we have set $\lambda(\gamma) :=
\max\{|\lambda| : \lambda\in\sigma(P_{\gamma|\cB_0})\setminus\{r(\gamma)\}\}$.
Due to Theorem \ref{thmkellerliverani}, $\chi$ is continuous on $J_0$.
Let $K$ be a compact subset of $J_0$. We set $\theta := \max_{K} \frac{\chi}{r}$. Since $\chi(\gamma) < r(\gamma)$ for every $\gamma\in K$ and since $r(\cdot)$ and $\chi(\cdot)$ are continuous, we conclude that $\theta\in(0,1)$. Next we consider any $\eta>0$ such that $\theta + 2\eta <1$. Next let us construct the map
$\gamma\mapsto \Pi_\gamma$ from $K$ to $\cL(\cB_0)$. Let $\gamma_0\in K$. Since $r$ is continuous on $K$,
there exists $\varepsilon>0$ such that, for every $\gamma\in K$ such that $|\gamma-\gamma_0|\le\varepsilon$,
we have $|r(\gamma)-r(\gamma_0)|<\eta r(\gamma_0)$.
Let us write $K(\gamma_0)$ for the set of  $\gamma\in K$
such that $|\gamma-\gamma_0|\le\varepsilon$.
Observe that, for any $\gamma\in K(\gamma_0)$,
$$\chi(\gamma)\le\theta r(\gamma) < \theta(1+\eta)r(\gamma_0)
 < (\theta +\eta)r(\gamma_0) <(1 -\eta)r(\gamma_0)  $$
and so the eigenprojector $\Pi_\gamma$ on $\ker (P_\gamma-
r(\gamma)I)$ can be defined by
\begin{equation}\label{formulaCauchy1}
\Pi_\gamma = \frac{1}{2i\pi}\oint_{\Gamma_1(\gamma_0)} R_z(\gamma)\, dz,
\end{equation}
where $\Gamma_1(\gamma_0)$ is the oriented circle centered on $r(\gamma_0)$
with radius $\eta\, r(\gamma_0)$. Due to Theorem
\ref{thmkellerliverani}, $\gamma\mapsto\Pi_\gamma$ is well
defined from $K(\gamma_0)$ to $\cL(\cB_0)$ and is continuous
from $K(\gamma_0)$ to $\cL(\cB_0,\cB_1)$.

Now, for every $\gamma\in K$, we define the oriented
circle $ \Gamma_0(\gamma)  := \big\{z\in\C : |z| = (\theta + \eta)\, r(\gamma)\big\}$.
By definition of $\theta$, for every $\gamma\in K$, we have $\chi(\gamma)\leq \theta\,  r(\gamma)$ and so
$\chi(\gamma) < (\theta + \eta)\,  r(\gamma)< r(\gamma)$.
Hence, by definition of $\chi(\gamma)$,  $R_z(\gamma)$ is well-defined in $\cL(\cB_0)$ for every $\gamma\in K$ and $z\in\Gamma_0(\gamma)$. From spectral theory, it comes that
\begin{equation}\label{formulaCauchy}
N_\gamma^n:=P_\gamma^n - r(\gamma)^n\Pi_\gamma = \frac{1}{2i\pi}\oint_{\Gamma_0(\gamma)} z^n\, R_z(\gamma)\, dz
\end{equation}
and so
\begin{equation}\label{ineg-norme}
\|P_\gamma^n - r(\gamma)^n\Pi_\gamma\|_{\cB_0} \leq M_\gamma\, \big((\theta + \eta)\, r(\gamma)\big)^{n+1}\quad \text{with}\quad M_\gamma := \sup_{|z| = (\theta + \eta)\, r(\gamma)} \|R_z(\gamma)\|_{\cB_0}.
\end{equation}
We have to prove
that
\begin{equation} \label{sup-M}
M_K := \sup_{\gamma\in K} M_\gamma < \infty.
\end{equation}
Let $\gamma_0\in K$. Since $\gamma\mapsto r(\gamma)$ is continuous at $\gamma_0$, there exists $\alpha\equiv\alpha(\gamma_0)>0$ such that, for every $ \gamma\in K$ such that $|\gamma-\gamma_0| < \alpha$, we have
$$\frac{\theta + \frac{\eta}{2}}{\theta + \eta}\, r(\gamma_0) < r(\gamma) <
\frac{\theta + \frac{3\eta}{2}}{\theta + \eta}\, r(\gamma_0).$$
Set  $\delta := \frac{\eta}{2} \, r(\gamma_0)$.
If $|\gamma-\gamma_0| < \alpha$ and if $|z|=(\theta + \eta)\, r(\gamma)$, we obtain  since $\delta_0\leq\chi(\gamma_0) \leq \theta\,  r(\gamma_0)$ and $\theta + 2\eta <1$:
$$\delta_0+\delta \leq \chi(\gamma_0) +  \delta \leq \big(\theta + \frac{\eta}{2}\big) \, r(\gamma_0) < |z| < \big(\theta + \frac{3\eta}{2}\big) \, r(\gamma_0) <
r(\gamma_0)- \delta.$$
From the previous inequalities, let us just keep in mind that
$\chi(\gamma_0) +  \delta < |z| <  r(\gamma_0)- \delta$. Then, by definition of $\chi(\gamma_0)$, we conclude that every complex number $z$
such that $|z| = (\theta + \eta)\, r(\gamma)$ satisfies
$$|z| > \delta_0+ \delta\ \ \text{and}\ \ d\big(z,\sigma(Q(\gamma_0))\big) > \delta.$$
Hence, up to a change of $\alpha$, due to Theorem \ref{thmkellerliverani}, we obtain that
$$\sup_{\gamma>0\,:\,|\gamma-\gamma_0| < \alpha} M_\gamma = \sup\left\{ \|R_z(\gamma)\|_{\cB_0}\ :\ |\gamma-\gamma_0| < \alpha,\ |z| = (\theta + \eta)\, r(\gamma)\right\} < \infty.$$
By a standard compacity argument, we have proved
(\ref{sup-M}). Consequently, with $\theta_K := \theta + \eta$, we deduce from \eqref{ineg-norme} that
$$\|P_\gamma^n - r(\gamma)^n\Pi_\gamma\|_{\cB_0} \leq M_K\, \big(\theta_K\, r(\gamma)\big)^n$$
from which we derive \eqref{sup-vit}. The proof of Theorem~\ref{generalspectraltheorem1} under Hypothesis \ref{hypKL} is achieved.

\subsection{Proof of Theorem \ref{generalspectraltheorem2} under Hypothesis \ref{hypKL}}
First we prove the following lemma.
\begin{alem} \label{rad-egaux}
For every  $\gamma\in J_0$ and for $i=1,2$, the spectral radius of $P_{\gamma| \cB_i}$ is equal to $r(\gamma):=r\big(P_{\gamma| \cB_0}\big)$.
\end{alem}
\begin{proof}
For $i=0,1,2$ set $r_i(\gamma):=r((P_\gamma)_{|\cB_i})$.
Due to Theorem \ref{generalspectraltheorem1} applied to $(P_\gamma,J,\cB_i,\cB_{i+1})$, there exists $c_i>0$
such that
$  \pi(P_\gamma^n\mathbf{1}_\X)\sim  c_i \, r_i(\gamma)^n $ as $n$ goes to infinity.
This proves the equality of the spectral radius.
\end{proof}
\begin{proof}[Proof of Theorem \ref{generalspectraltheorem2} under Hypothesis \ref{hypKL}]
We define $\chi_i$ as
$\chi$ in the proof of Theorem \ref{generalspectraltheorem1} for each $\cB_i$ ($i=0,1,2$). We define now $\chi:=\max(\chi_0,\chi_1,\chi_2)$. Let us prove the differentiability of $r$ and $\Pi$ on $J_0$. Let  $\gamma_0 \in J_0$. Let $\eta>0$ be such that
$r(\gamma_0)>\chi(\gamma_0)+2\eta$ and let $\varepsilon>0$
be such that for every $\gamma\in J_0$ satisfying $|\gamma-\gamma_0|<\varepsilon$, we have
$r(\gamma)>r(\gamma_0)-\eta>\chi(\gamma_0)+\eta>\chi(\gamma)$.
We set
$I_0 := J_0\cap (\gamma_0-\varepsilon,\gamma_0-\varepsilon)$ and
\begin{equation} \label{D0}
\cD_0:=\{z\in\C\, :\, \chi(\gamma_0)+\eta<|z|<r(\gamma_0)-\eta\} \cup \{z\in\C : |z-r(\gamma_0)| = \eta\}.
\end{equation}
Due to the hypotheses of Theorem \ref{generalspectraltheorem2} and to an easy adaptation of \cite[Lemma A.2]{HerPen10} (see Remark~\ref{KL-explication}), we obtain that, for every $z\in\cD_0$, the map $\gamma\mapsto R_z(\gamma)$ is $C^1$ from
$I_0$ to $\mathcal L(\cB_0,\cB_3)$ with $R'_z(\gamma)=R_{z}(\gamma)P'_\gamma R_{z}(\gamma)$
and
\begin{equation} \label{deri-unif}
\lim_{h\rightarrow 0}
    \sup_{z\in\cD_0} \frac{\|R_z(\gamma_0+h)-R_z(\gamma_0)-hR'_z(\gamma_0)\|_{{\cB}_0,{\cB}_3}}{|h|}=0.
\end{equation}
Moreover, for every $\gamma\in I_0$, we deduce from spectral theory that
$$ \Pi_\gamma=\frac 1{2i\pi }\oint_{\Gamma_1} R_z(\gamma)\, dz\quad\mbox{and}\quad
     N_\gamma=\frac 1{2i\pi}
  \oint_{\Gamma_0}z R_z(\gamma)\, dz,$$
where $\Gamma_1$ is the oriented circle centered at $r(\gamma_0)$ with radius $\eta$ and $\Gamma_0$ is the oriented circle centered at $0$ with some radius
$\vartheta_0$ satisfying $\chi(\gamma_0)+\eta<\vartheta_0<r(\gamma_0)-\eta$.
Thus $\gamma\mapsto\Pi_\gamma$ and $\gamma\mapsto N_\gamma$ are
$C^1$-smooth from $J_0$ to $\cL(\cB_0,\cB_3)$.
Since $1_\X\in\cB_0$ by hypothesis this implies the continuous differentiability of $\gamma\mapsto N_\gamma\mathbf 1_\X$ and of $\gamma\mapsto \Pi_\gamma\mathbf 1_\X$ from $J_0$ to $\cB_3$. Since
$r(\gamma)=\frac{(P_\gamma-N_\gamma)(\mathbf 1_\X)}{\Pi_\gamma(\mathbf 1_\X)}$ and $\gamma\mapsto P_\gamma\mathbf 1_\X$ is $C^1$ from $I_0$ to $\cB_3$ by hypothesis, we obtain the continuous differentiability of $r$ on
$I_0$.
\end{proof}

\begin{arem}[Proof of the differentiability of $\gamma\mapsto R_z(\gamma)$]
\label{KL-explication}
We adapt the arguments of \cite[Lemma A.2]{HerPen10}, writing
$$R_z(\gamma) = R_z(\gamma_0) \ +\  R_z(\gamma_0)\, [P_\gamma-P_{\gamma_0}]\, R_z(\gamma_0) \ + \  \vartheta_z(\gamma),$$
$$\text{with}\qquad \vartheta_z(\gamma) := R_z(\gamma_0)\, [P_\gamma - P_{\gamma_0}]\,R_z(\gamma_0)\, [P_{\gamma} - P_{\gamma_0}]\, R_z(\gamma).$$
Then
\begin{equation} \label{ineg-deri-res}
\frac{\Vert \vartheta_z(\gamma)\Vert_{{\cB}_0,{\cB}_3}}{|\gamma-\gamma_0|}
\le
\Vert R_z(\gamma_0) \Vert_{{\cB}_2}
\left\Vert\frac{P_\gamma - P_{\gamma_0}}{\gamma-\gamma_0}
\right\Vert_{{\cB}_1,{\cB}_2}
\Vert R_z(\gamma_0)\Vert_{{\cB}_1}
\Vert P_\gamma - P_{\gamma_0}\Vert_{{\cB}_0,{\cB}_1}
\Vert R_z(\gamma)\Vert_{{\cB}_0}.
\end{equation}
From the hypotheses of Theorem \ref{generalspectraltheorem2} and from the resolvent bounds derived from Theorem~\ref{thmkellerliverani}, the last term goes to 0, uniformly in $z\in\cD$, when $\gamma$ goes to $\gamma_0$.  Similarly we have:
\begin{eqnarray*}
&\ & \left\Vert R_z(\gamma_0)(P_\gamma - P_{\gamma_0})R_z(\gamma_0) - (\gamma-\gamma_0)R_z(\gamma_0) P'_{\gamma_0} R_z(\gamma_0)
\right\Vert_{{\cB}_0,{\cB}_3}   \\
&\ & \qquad \qquad \qquad \qquad \qquad \le \ M
 \Vert P_\gamma - P_{\gamma_0} - (\gamma-\gamma_0)P'_{\gamma_0}\Vert_{{\cB}_1,{\cB}_2} = \text{o}(\gamma-\gamma_0)
\end{eqnarray*}
when again the finite positive constant $M$ is derived from the resolvent bounds of Theorem~\ref{thmkellerliverani}. This shows that $R_z'(\gamma_0)=R_z(\gamma_0)P'_{\gamma_0}R_z(\gamma_0)$ in $\cL({\cB}_0,{\cB}_3)$. To prove that $\gamma\mapsto R_z'(\gamma)$ is continuous from $J_0$ to $\cL({\cB}_0,{\cB}_3)$ in a uniform way with respect to $z\in\cD$, observe that $\gamma\mapsto R_z(\gamma)$ is $\cC^0$ from $J_0$ to $\cL(\cB_0,\cB_1)$ (use Theorem~\ref{thmkellerliverani}), that $\gamma\mapsto P'_{\gamma}$ is $\cC^0$ (uniformly in $z\in\cD$) from $J_0$ to $\cL(\cB_1,\cB_2)$ by hypothesis, and finally that $\gamma\mapsto R_z(\gamma)$ is $\cC^0$ (uniformly in $z\in\cD$) from $J_0$ to $\cL(\cB_2,\cB_3)$ (again use Theorem~\ref{thmkellerliverani}).
Observe that (\ref{ineg-deri-res}) gives the differentiability at $\gamma_0$ of the map $\gamma\mapsto R_z(\gamma)$ considered from $J$ to $\cL(\cB_0,\cB_2)$. The additional space $\cB_3$ is only required to obtain the continuous differentiability.
\end{arem}
%
\subsection{Proof of Theorems~\ref{generalspectraltheorem1} and \ref{generalspectraltheorem2} under Hypothesis \ref{hypKL}*}

Here the Keller-Liverani perturbation theorem must be applied to the dual family $(P_\gamma^*)_\gamma$. Actually the hypotheses of Theorem~\ref{generalspectraltheorem1} are:
\begin{itemize}
\item $\cB_1^* \hookrightarrow \cB_0^*$,
\item For every $\gamma\in J$, $P_\gamma^*\in \mathcal L(\mathcal B_0^*)\cap\mathcal L(\mathcal B_1^*)$,
\item $\gamma\mapsto P_\gamma^*$ is a continuous map from $J$ in $\mathcal L( \mathcal B_1^*,\mathcal B_0^*)$,
\item There exist $\delta_0,c_0,M_0>0$ such that, for all $\gamma\in J$,  $r_{ess}\big((P_{\gamma})^*_{|\cB_1^*}\big)\le\delta_0$
and
$$\forall n\geq 1,\ \forall f^*\in\cB_1^*,\quad
\|(P_\gamma^*)^n f^*\|_{\cB_1^*}\le c_0(\delta_0^n\| f^*\|_{\cB_1^*} + M^n\| f^*\|_{\cB_0^*})
.$$
\item Hypothesis \ref{hypcompl} holds on $(
J_0
,\cB_1)$.
\end{itemize}

\begin{proof}[Proof of Theorem \ref{generalspectraltheorem1} under Hypothesis \ref{hypKL}*]
Under these assumptions it follows from Theorem~\ref{thmkellerliverani} applied to $(P_\gamma^*)_{\gamma\in J}$ with respect to $(\mathcal B_1^*,\mathcal B_0^*)$ that, for every $\varepsilon>0$ and every $\delta>\delta_0$, the map $t\mapsto(zI-P_\gamma^*)^{-1}$ is well defined from $J_0$ to $\mathcal L(\mathcal B_1^*)$, provided that $z\in\mathcal D(\delta,\varepsilon)$ with
$$\mathcal D(\delta,\varepsilon):=\{z\in \mathbb C,\
       d(z,\sigma\big((P_{\gamma_0}^*)_{|\cB_2^*})\big)>\varepsilon,\ |z|>\delta\} = \{z\in \mathbb C,\ d(z,\sigma\big((P_{\gamma_0})_{|\cB_2})\big)>\varepsilon,\ |z|>\delta\}.$$
In addition, the map $t\mapsto(zI-P_\gamma^*)^{-1}$, considered from $J_0$ to $\mathcal L(\mathcal B_1^*,\mathcal B_0^*)$, is continuous at every $\gamma_0\in J_0$ in a uniform way with respect to $z\in\mathcal D(\delta,\varepsilon)$. By duality this implies that $t\mapsto(zI-P_\gamma)^{-1}$ is well defined from $J_0$ to $\mathcal L(\mathcal B_1)$. Moreover, when this map is considered from $J_0$ to
$\mathcal L(\mathcal B_0,\mathcal B_1)$, it is continuous at $\gamma_0$ in a uniform way with respect to $z\in\mathcal D(\delta,\varepsilon)$.
Finally Hypothesis \ref{hypcompl} on $(J_0,\cB_1)$ enables us to identify the spectral elements associated with $r(\gamma) := r\big((P_\gamma)_{|\cB_1}\big)$. Consequently one can prove as
in Subsection~\ref{ap-proof-general-th}
that there exists a map $\gamma\mapsto \Pi_\gamma$ from $J_0$ to $\mathcal L(\cB_1)$, which is continuous from  $J_0$ to $\mathcal L(\cB_0,\cB_1)$, such that (\ref{sup-vit}) holds with $\cB:=\cB_1$.
\end{proof}

\begin{proof}[Proof of Theorem \ref{generalspectraltheorem2} under Hypothesis~\ref{hypKL}*] When Theorem \ref{generalspectraltheorem2} is stated with Hypothesis~\ref{hypKL}*, then Theorem \ref{generalspectraltheorem1} applies on $(\cB_0,\cB_1)$, $(\cB_1,\cB_2)$ and $(\cB_2,\cB_3)$ (with Hypothesis~\ref{hypKL}* in each case). Thus, for every $\gamma\in J_0$,  the spectral radius $r_i(\gamma):=r((P_\gamma)_{|\cB_i})$ are equal for $i=1,2,3$ (See the proof of Lemma~\ref{rad-egaux}). Observe that,
from our hypotheses,
Hypothesis \ref{hypcompl} holds on $(J_0,\cB_i)$ for $i=1,2,3$.
Since $P_\gamma^*$ on $\cB_i^*$ inherits the spectral properties of $P_\gamma$ on $\cB_i$, we can prove as above that, for every $\gamma_0\in J_0$ and for every $\varepsilon>0$ and $\delta>\delta_0$, the map $\gamma\mapsto(zI-P_\gamma^*)^{-1}$ is well defined from some subinterval $I_0$ of $J_0$ containing $\gamma_0$ into $\mathcal L(\mathcal B_3^*)$, provided that $z\in\cD_0$ where the set $\cD_0$ is defined in (\ref{D0}). In addition, by applying Remark~\ref{KL-explication} with the adjoint operators $(P_\gamma^*)_\gamma$ and the spaces $\cB_3^*\hookrightarrow\cB_2^*\hookrightarrow\cB_1^*
\hookrightarrow\cB_0^*$, we can prove that the map $\gamma\mapsto(zI-P_\gamma^*)^{-1}$, considered from $J_0$ to $\mathcal L(\mathcal B_3^*,\mathcal B_0^*)$, is $\cC^1$ in a uniform way with respect to $z\in\cD_0$. By duality, this gives (\ref{deri-unif}). We conclude the differentiability of $\gamma\mapsto
\Pi_{\gamma}^*$ from $J_0$ to $\mathcal L(\mathcal B_3^*,\mathcal B_1^*)$ and so  the differentiability of $\gamma\mapsto
\Pi_{\gamma}$ from $J_0$ to $\mathcal L(\mathcal B_1,\mathcal B_3)$.
\end{proof}

\subsection{Complements on the derivative of $r(\cdot)$}
Let us first prove the following.
%

%
\begin{alem}\label{deriveenegative}
Let
$J_1=(a,b)\subset[0,+\infty)$
and let $ \cB_1\hookrightarrow \cB_2$ be two Banach spaces
such that
$f\mapsto \xi f\in\cL(\cB_1,\cB_2)$. Assume that, for every $\gamma\in J_1$, $P_\gamma\in\cL(\cB_1)\cap\cL(\cB_2)$ and that there exist elements $\phi_\gamma\in\cB_1$ and $\pi_\gamma\in \cB^*_2$ such that
$P_\gamma\phi_\gamma=r(\gamma)\phi_\gamma$
and $P^*_\gamma\pi_\gamma=r(\gamma)\pi_\gamma$. Moreover assume that
$\gamma\mapsto P_\gamma$   and  $\gamma\mapsto r(\gamma)$ are differentiable from $J_1$ to $\cL(\cB_1,\cB_2)$ and to $\C$  respectively, with
respective derivatives at $\gamma\in J_1$ given by $P_{\gamma}' : f\mapsto P_{\gamma}(-\xi f)$ and $r'(\gamma)$.
Finally assume that $\gamma\mapsto\phi_\gamma$ is continuous from $J_1$ to $\cB_1$ and differentiable from $J_1$ to $\cB_2$ with derivative $\gamma\mapsto\phi'_{\gamma}$.

Then we have for every $\gamma\in J_1$: $\ r'(\gamma)\pi_{\gamma} \left( \phi_{\gamma}\right)=-r(\gamma)\pi_{\gamma} \big( \xi \phi_{\gamma}\big)$. In particular,
if $r(\gamma)>0$,
$\pi_{\gamma}(\phi_{\gamma})\ge 0$
and
$\pi_{\gamma}\big( \xi \phi_{\gamma}\big)>0$,
then $r'(\gamma)<0$.
\end{alem}
\begin{proof}
Let $\gamma,\gamma_0\in J_1$.
We have $P_\gamma\phi_\gamma = r(\gamma)\phi_\gamma$ in $\cB_2$.
From $P_\gamma\phi_\gamma-
P_{\gamma_0}\phi_{\gamma_0}=P_{\gamma_0}(\phi_\gamma-\phi_{\gamma_0})+
(P_\gamma-P_{\gamma_0})(\phi_{\gamma})$, we
obtain that
$$  P_{\gamma_0}(\phi'_{\gamma_0}) +
P_{\gamma_0}(-\xi \phi_{\gamma_0})
=r(\gamma_0)\phi'_{\gamma_0}+r'(\gamma_0)\phi_{\gamma_0}\ \ \ \mbox{in}\ \cB_2.$$
We conclude by
composing by $\pi_{\gamma_0}$
and using the fact that $\pi_{\gamma_0}\circ P_{\gamma_0}=r(\gamma_0)\pi_{\gamma_0}$.
\end{proof}
\begin{apro} \label{ap-pro-ap-deriv}
Assume that the assumptions of Theorem~\ref{generalspectraltheorem2} hold. For every $\gamma\in J_0$, set $\phi_\gamma :=\Pi_\gamma\mathbf 1_\X$ and $\pi_\gamma := \Pi^*_\gamma\pi$. Then the assumptions of Lemma~\ref{deriveenegative} hold with $J_1=J_0$ and with respect to the spaces $\cB_1\hookrightarrow\cB_2$ (resp.~$\cB_1\hookrightarrow\cB_3$) when the assumptions of Theorem~\ref{generalspectraltheorem2} hold with Hypothesis~\ref{hypKL} (resp.~with Hypothesis~\ref{hypKL}*). Consequently, for every $\gamma\in J_0$,
$$\pi(\Pi_{\gamma}(\xi\, \Pi_{\gamma} 1_\X)) > 0\ \Longrightarrow\, r'(\gamma) < 0.$$
\end{apro}
\begin{proof}[Proof of  Proposition~\ref{ap-pro-ap-deriv} under Hypothesis \ref{hypKL}]
We have
$\pi_\gamma\in \cB^*_2$ since $\pi\in\cB_2^*$ and $\Pi^*_\gamma$ is well defined in $\cL(\cB_2^*)$. Moreover $\phi_\gamma\in\cB_1$ since $\mathbf 1_\X\in\cB_1$ and
$\Pi_\gamma\in \cL(\cB_1)$,
and $\gamma\mapsto \phi_\gamma$ is continuous from $J$ to $\cB_1$ by Theorem~\ref{generalspectraltheorem1}. Finally $\gamma\mapsto \phi_\gamma$ is differentiable from $J$ to $\cB_2$ (see the end of Remark~\ref{KL-explication}).
We have proved that the assumptions of Lemma~\ref{deriveenegative} hold as stated under Hypothesis~\ref{hypKL}.
Finally, since $r(\gamma) > 0$ when $\gamma\in J_0$, the desired implication in Proposition~\ref{ap-pro-ap-deriv} follows from the conclusion of Lemma~\ref{deriveenegative} because $\pi_{\gamma}(\phi_{\gamma}) = \pi(\phi_{\gamma})$ and $\pi_{\gamma}(\xi\, \phi_{\gamma}) = \pi(\Pi_{\gamma}(\xi\, \Pi_{\gamma} 1_\X))$.
\end{proof}
\begin{proof}[Proof of  Proposition~\ref{ap-pro-ap-deriv} under Hypothesis \ref{hypKL}*]
Note that $\pi_\gamma:=\Pi^*_\gamma\pi\in \cB^*_3$ since $\pi\in\cB_3^*$ and $\Pi^*_\gamma$ is well defined in $\cL(\cB_3^*)$. The function $\gamma\mapsto P_\gamma$ is differentiable from $J_0$ to $\cL(\cB_1,\cB_2)$, thus from $J_0$ to
$\cL(\cB_1,\cB_3)$. We have $\phi_\gamma:=\Pi_\gamma\mathbf 1_\X\in\cB_1$ since $1_\X\in\cB_1$ and $\Pi_\gamma$ is well defined in $\cL(\cB_1)$. Moreover $\gamma\mapsto \phi_\gamma$ is continuous from $J$ to $\cB_1$ since $\Pi_\gamma$ is well defined in $\cL(\cB_1)$, continuous from $J_0$ to $\cL(\cB_0,\cB_1)$, and $\mathbf 1_\X\in\cB_0$. Finally
$\gamma\mapsto \phi_\gamma$ is differentiable from $J$ to $\cB_3$ since $\Pi_\gamma$ is well defined in $\cL(\cB_3)$ and differentiable from $J_0$ to $\cL(\cB_1,\cB_3)$ and $\mathbf 1_\X\in\cB_1$.
\end{proof}

\section{Proof of Proposition~\ref{pro-reduc-BL}} \label{proof-reduc-BL}

%

Proposition~\ref{pro-reduc-BL} directly follows from the following statement.

\begin{apro} \label{firstorder}
Let $\mathcal B$ be a non null complex Banach lattice of functions
 $f:\X\rightarrow \C$ (or of classes of such functions
modulo $\pi$).
Let $Q$ be a (nonnull) nonnegative quasicompact operator on $\mathcal B$ such that $r(Q)\ne 0$ and such that
for every nonnull nonnegative $f\in\mathcal B$ and for every nonnull nonnegative $\psi\in\mathcal B ^*\cap\ker(Q^*-r(Q)I)$, we have $Qf>0$
and $\psi(Qf)> 0$.
Then
\begin{enumerate}[(a)]
	\item $r(Q)$ is a first order pole of $Q$, and there exists a positive $\phi\in\cB$ and a  positive $\psi\in\cB ^*$ such that
\begin{equation} \label{psi-phi}
\psi(\phi)=1, \qquad Q \phi = r(Q) \phi \qquad\text{and}\qquad Q^*\psi = r(Q) \psi.
\end{equation}
\item Let $\lambda\in\C$ and $h\in\cB$ such that $|\lambda|=r(Q)$ and $Q h = \lambda h$. Then $Q |h| = r(Q) |h|$ in $\cB$.
\item If moreover $Q$ is of the form $Q=P(\kappa e^{-\gamma\xi}\cdot)$ where $P$ is the operator associated with
a Markov kernel, if $1_{\X}\in \cB\hookrightarrow \L^1(\pi)$, if $\ker(Q-r(Q)I)=\C\cdot \phi$ and
if 1 is the only complex number $\lambda$ of modulus 1 such that
 $P(h/|h|)=\lambda h/|h|$ in $\L^1(\pi)$ for some $h\in\cB$ with $|h|>0$, then $r(Q)$ is the only eigenvalue of modulus
$r(Q)$ of $Q$.
\end{enumerate}
\end{apro}
\begin{proof}
The fact that $r(Q)$ is a finite pole of $Q$ is classical for a nonnegative quasi-compact operator $Q$ on a Banach lattice. Let us just remember the main arguments. From quasi-compactness we know that there exists a finite pole $\lambda\in\sigma(Q)$ such that $|\lambda|=r(Q)$. Thus, setting $\lambda_n:=\lambda(1+1/n)$ for any $n\geq 1$, we deduce from $\lambda\in\sigma(Q)$ that
$$\lim_n \|(\lambda_nI - Q)^{-1}\|_\cB = +\infty.$$
Since $\cB$ is a Banach lattice, we deduce from the Banach-Steinhaus theorem that there exists a nonnegative and nonnull element $f\in\cB$ such that
$$\lim_n \|(\lambda_nI - Q)^{-1}f\|_\cB = +\infty.$$
Next define $r_n :=r(Q)(1 +1/n)$ and observe that
$$\big|(\lambda_nI - Q)^{-1}f\big| = \big|\sum_{k\geq 0}  \lambda_n^{-(k+1)}\, Q^{\, k}f\big| \leq \sum_{k\geq 0}  r_n^{-(k+1)}\, Q^{\, k}f.$$
Since $\cB$ is a Banach lattice, the last inequality is true in norm, that is
$$\big\|(\lambda_nI - Q)^{-1}f\big\| \leq \big\|\sum_{k\geq 0}  r_n^{-(k+1)}\, Q^{\, k}f\big\|$$
from which we deduce that $\lim_n \|(r_nI - Q)^{-1}\|_\cB = +\infty$, thus $r(Q)\in\sigma(Q)$. Finally $r(Q)$ is a finite pole of $Q$ from quasi-compactness.

Let $q$ denote the order of the pole $r(Q)$, namely $r(Q)$ is a pole of order $q$ of the resolvent function $z\mapsto (zI-Q)^{-1}$. Then there exists $\rho>0$ such that $(zI-Q)^{-1}$ admits the following Laurent series provided that $|z-r(Q)| < \rho$ and $z\neq r(Q)$:
$$(zI-Q)^{-1} = \sum_{k=-q}^{+\infty} (z-r(Q))^{k} A_k,$$
where $A_k$ are bounded linear operators on $\cB$. By quasi-compactness, $A_{-1}$ is a projection onto the finite subspace $\ker(Q-r(Q) I)^q$. Moreover we know that
\begin{equation} \label{A-q-res}
A_{-q} = \big(Q-r(Q) I\big)^{q-1}\circ A_{-1} =  A_{-1}\circ\big(Q-r(Q) I\big)^{q-1}.
\end{equation}
and that, setting $r_n :=r(Q)(1 +1/n)$,
\begin{eqnarray}
A_{-q}  &=& \lim_{n\r+\infty} \big(r_n-r(Q)\big)^{q}\big(r_nI-Q\big)^{-1} \nonumber \\
&=&  \lim_{n\r+\infty} \big(r_n-r(Q)\big)^{q}\sum_{k\geq 0}  r_n^{-(k+1)}\, Q^{\, k}. \label{neuman}
\end{eqnarray}
Since $Q$ is a nonnull nonnegative operator on $\cB$, so is $A_{-q}$. Since $A_{-q}\neq0$,
we take a nonnegative $h_0\in\mathcal B$ such that $\phi:= A_{-q} h_0\ne 0$ in $\cB$.
We have $(Q-r(Q)I)A_{-q}=0$, so
$r(Q) \phi = Q\phi$. Similarly there exists a nonnegative $\psi_0\in\cB^*$ such that
$\psi_1 := A_{-q}^*\psi_0$ is a nonzero and nonnegative element of
$\ker(Q^*-r(Q)I)$, where $A_{-q}^*$ is the adjoint operator of $A_{-q}$.  Note that $\psi_1(\phi)=\psi_1(Q\phi)/r(Q)>0$ from our hypotheses, so that $\phi$ and $\psi:=\psi_1/\psi_1(\phi)$ satisfy (\ref{psi-phi}).
To conclude the proof of Assertion~$(a)$, let us prove by reductio ad absurdum that $q=1$. Assume that $q\geq 2$. Then  $A_{-q}^{\ 2}=0$ from (\ref{A-q-res}) and $A_{-1}(\cB) = \ker(Q-r(Q) I)^q$, so that $\psi_1(\phi)=(A^*_{-q}\psi_0)(A_{-q}h_0)=\psi_0(A_{-q}^2h_0)=0$. This contradicts the above fact.

To prove $(b)$, recall that, from our hypotheses, we have $\psi(g)=\psi(Qg)/r(Q)>0$ for every nonnull nonnegative $g\in\cB$.
Let $\lambda\in\C$ and $h\in\cB$ such that $|\lambda|=r(Q)$ and $Q h = \lambda h$.
The positivity of $Q$ gives $|\lambda h| = r(Q) |h| = |Q h| \leq  Q |h|$,
thus $g_0:=Q |h|-r(Q) |h| \geq 0$. From $\psi(g_0)=0$, it follows that $g_0=0$, that is: $Q |h| = r(Q) |h|$ in $\cB$.

Now let us prove Assertion~$(c)$ of Proposition~\ref{firstorder}. Recall that the above nonnull nonnegative function $\phi\in\cB$ is such that $Q\phi = r(Q)\phi$. From our hypotheses we deduce that $\phi>0$ (modulo $\pi$). Let $\lambda\in\C$ and $h\in\cB$ be such that $|\lambda|=r(Q)$, $h\neq 0$   and
$Q h = \lambda h$.
Due to the previous point and to our assumptions, we obtain that
$Q |h| = r(Q)\, |h|$ and $|h| = \beta \phi$ for some $\beta>0$.
In particular $h\neq 0\ \pi-$ a.s..
One may assume that $\beta=1$ for the sake of simplicity.
Let $A=\{x\in\X : |h(x)| = \phi(x) >0\}$ and
$$B=\{x\in\X : (P_\gamma \phi)(x) = r(Q)\phi(x)\},\quad C=\{x\in\X : (P_\gamma h)(x) = \lambda h(x)\}.$$
Let $A^c = \X\setminus A$. It follows from $\pi(A^c)=0$ and from the invariance of $\pi$ that
$$ \pi\left(P\big(1_{A^c}(\phi+|h|)\, \kappa e^{-\gamma\xi}\big)\right) = \pi\left(1_{A^c}\big(\phi+|h|\big)\,\kappa e^{-\gamma\xi}\right) = 0.$$
Let $D:= \{x\in\X : \big(P (1_{A^c}\, (\phi+|h|)\,\kappa e^{-\gamma\xi})\big)(x) = 0\}$. Then we have  $\pi(D)= 1$. Now define $E = A\cap B\cap C \cap D$. Then $\pi(E)= 1$, and we obtain that
\begin{subequations}
\begin{gather}
\forall x\in E,\quad |h(x)| = \phi(x)>0 \label{E1} \\
\forall x\in E,\quad
\lambda\, h(x) = \big(P(1_{A} h\, \kappa e^{-\gamma\xi})\big)(x) = \int_A h(y)\, \kappa(y)e^{-\gamma\xi(y)}\, P(x,dy) \label{E2} \\
\forall x\in E,\quad
r(Q)\, \phi(x) = \big(P(1_{A} \phi\, \kappa e^{-\gamma\xi})\big)(x) = \int_A \phi(y)\,\kappa(y)e^{-\gamma\xi(y)}\, P(x,dy). \label{E3}
\end{gather}
\end{subequations}
Let $x\in E$ and define the probability measure:
$\eta_x(dy) := (r(Q)\, \phi(x))^{-1}\phi(y)\, \kappa(y)e^{-\gamma\xi(y)}\, P(x,dy)$.
We have
$$\int_A \frac{r(Q)\, \phi(x)\, h(y)}{\lambda\, \phi(y)\, h(x)}\ \eta_x(dy) = 1.
$$
Since $|h(x)|=\phi(x)$ and $|h| = \phi$ on $A$, the previous integrand is of modulus one. Then a standard convexity argument ensures that the following equality holds for $P(x,\cdot)-$almost every $y\in \X$:
$$r(Q)\,  \phi(x)\, h(y) = \lambda\, \phi(y)\,  h(x).$$
This implies that $r(Q)P\frac{h}{|h|}=\lambda \frac{h} {|h|}$ everywhere on $E$, thus $r(Q)P\frac{h}{|h|}=\lambda \frac{h} {|h|}$ in $\L^1(\pi)$. So $\lambda=r(Q)$ from the hypothesis of Assertion~$(c)$ of Proposition~\ref{firstorder}.
\end{proof}
%

\section{A counter-example} \label{counterexample}
Assume that $(\X,d)$ is a metric space equipped with its Borel $\sigma$-algebra.
Let $\mathcal \cL^\infty$ denote the set of bounded functions $f:\mathbb X\rightarrow\mathbb C$, endowed with the
supremum norm.
\begin{apro}
Assume that $P$ is a Markov kernel satisfying the following condition : there exists $S\in(0,+\infty)$ such that, for every $x\in\X$, the support of $P(x,dy)$ is contained in the ball $B(x,S)$ centered at $x$ with radius $S$. Assume that $\kappa\equiv 2$ and that $\xi(y)\r0$ when $d(y,x_0)\r +\infty$, where $x_0$ is some fixed point in $\X$.
Then, for every  $\gamma\in[0,+\infty)$, the kernel $P_\gamma :=
2
P(e^{-\gamma\xi}\cdot)$ continuously acts on $\cL^\infty$ and its spectral radius $r(\gamma)=r((P_\gamma)_{|\cL^\infty})$ satisfies the following
$$\forall \gamma\in[0,+\infty),\quad r(\gamma) = 2.$$
\end{apro}
\begin{proof}
We clearly have $r(\gamma) \leq 2$ since $P_\gamma\leq P$  and $P$ is Markov.
For any $\beta>0$, we obtain with $f = \mathbf{1}_{[\xi \leq \beta]}$
$$\forall x\in\X,\quad (P_\gamma f)(x) =
2
\int_{[\xi \leq \beta]} e^{-\gamma\xi(y)}\, P(x,dy) \geq
2
e^{-\gamma\beta}\, P\big(x,[\xi \leq \beta]\big).$$
The set $[\xi \leq \beta]$ contains $\X\setminus B(x_0,R)$ for some $R>0$ since $\xi(y)\r0$ when $d(y,x_0)\r +\infty$. Thus, for $d(x,x_0)$ sufficiently large ($d(x,x_0)>R+S$), we
have $P\big(x,[\xi \leq \beta]\big) = 1$, so that
$\|P_\gamma\|_{\cL^\infty} \geq \|P_\gamma f\|_{\cL^\infty} \geq
2
e^{-\gamma\beta}.$
This gives $\|P_\gamma\|_{\cL^\infty} =
2$
when $\beta \r 0$. Similarly we obtain with $f = \mathbf{1}_{[\xi \leq \beta]}$, that,
$\forall x\in\X\setminus B(x_0,R+2S)$,
\begin{eqnarray*}
\quad (P_\gamma^2 f)(x) &=&
4
\int e^{-\gamma(\xi(y)+\xi(z))}\, \mathbf{1}_{[\xi \leq \beta]}(z)\, P(y,dz)\, P(x,dy) \\
&\geq&
4
e^{-\gamma\beta}\, \int_{\X\setminus B(x_0,R+S)} e^{-\gamma\xi(y)}\, P\big(y,[\xi \leq \beta]\big) P(x,dy)
\geq
4
e^{-2\gamma\beta}
\end{eqnarray*}
and so
$\forall \beta>0,\quad \|P_\gamma^2\|_{\cL^\infty}
\geq \|P_\gamma^2 f\|_{\cL^\infty} \geq
4
e^{-2\gamma\beta}.$
Again this provides $\|P_\gamma^2\|_{\cL^\infty} = 4$ since $\beta$ can be taken arbitrarily small. Similarly we can prove that $\|P_\gamma^n\|_{\cL^\infty} = 2^n$ for every $n\geq 1$, thus $ r(\gamma) = 2$.
\end{proof}
\section*{Acknowledgements}
We wish to thank Bernard Ycart for interesting
discussions.


\end{document}